\RequirePackage{amsthm}

\documentclass[sn-basic,Numbered,oneside]{sn-jnl}

\usepackage{graphicx}%
\usepackage{multirow}%
\usepackage{amsmath,amssymb,amsfonts}%
\usepackage{amsthm}%
\usepackage{mathrsfs}%
\usepackage{xcolor}%
\usepackage{textcomp}%
\usepackage{manyfoot}%
\usepackage{booktabs}%
\usepackage{algorithm}%
\usepackage{algorithmicx}%
\usepackage{algpseudocode}%
\usepackage{listings}%

\usepackage{lmodern}
\usepackage{anyfontsize}
\usepackage{enumitem}
\usepackage{multicol}
\usepackage[nameinlink,noabbrev,capitalize]{cleveref}
\usepackage{dsfont}%
\usepackage{subfigure}

 \everymath{\displaystyle}

\newcommand{\norm}[2]{\left\lVert #1\right\rVert_{#2}}

\newcommand{\Lip}[1]{\mathrm{Lip}(#1)}

\newcommand{\Ntol}{\textup{\texttt{tol}}}

\newcommand{\Id}{\mathrm{Id}}

\newcommand{\diam}{\mathrm{diam}}

\newcommand{\LL}{\mathcal{L}}

\newcommand\N{\mathbb{N}}
\newcommand\R{\mathbb{R}}

\definecolor{darkorange}{rgb}{8.,.4,0.}

\theoremstyle{thmstyleone} 

\newtheorem{theorem}{Theorem}[section]
\newtheorem{lemma}[theorem]{Lemma}%
\newtheorem{corollary}[theorem]{Corollary}%
\newtheorem{assumption}[theorem]{Assumption}%
\newtheorem{remark}[theorem]{Remark}%
\newtheorem{definition}[theorem]{Definition}
\crefname{assumption}{Assumption}{Assumptions}

\begin{document}


\makeatletter
\renewenvironment{proof}[1][\proofname]{\par
  \vspace{-0.00cm}
  \pushQED{\qed}%
  \normalfont
  \topsep0pt \partopsep0pt 
  \trivlist
  \item[\hskip\labelsep
        \itshape
    #1\@addpunct{.}]\ignorespaces
}{%
  \popQED\endtrivlist\@endpefalse
  \addvspace{6pt plus 6pt} 
}
\makeatother


\title[Article Title]{%
A Globalized Inexact Semismooth Newton Method for Nonsmooth Fixed-point Equations involving Variational Inequalities
}


\author*[1]{\fnm{Amal} \sur{Alphonse}}\email{alphonse@wias-berlin.de}
\author*[2]{\fnm{Constantin} \sur{Christof}}\email{christof@cit.tum.de}
\author[3]{\fnm{Michael} \sur{Hinterm{\"u}ller}}\email{hintermueller@wias-berlin.de}
\author[1]{\fnm{Ioannis} \sur{P.~A.~Papadopoulos}}\email{papadopoulos@wias-berlin.de}

\affil*[1]{\orgname{Weierstrass Institute}, \orgaddress{\street{Mohrenstra{\ss}e 39}, \city{Berlin}, \postcode{10117}, \country{Germany}}}

\affil[2]{\orgdiv{CIT, Department of Mathematics}, 
\orgname{Technical University of Munich}, \orgaddress{\street{Boltzmannstraße 3}, 
\city{Garching b.\ {M\"u}nchen}, \postcode{85748}, \country{Germany}}}

\affil[3]{\orgname{Weierstrass Institute}, \orgaddress{\street{Mohrenstra{\ss}e 39}, \city{Berlin}, \postcode{10117}, \country{Germany}}; \orgname{Humboldt-Universit\"at zu Berlin}, \orgaddress{\street{Unter den Linden 6}, \city{Berlin}, \postcode{10117}, \country{Germany}}}

\abstract{
We develop a semismooth Newton framework for the 
numerical solution of fixed-point equations
that are posed in Banach spaces. 
The framework is motivated by applications
in the field of obstacle-type quasi-variational 
inequalities and implicit obstacle problems. 
It is discussed in 
a general functional analytic setting and  
allows for inexact function evaluations and 
Newton steps. Moreover, if a certain contraction assumption holds, we show that it is possible to globalize the algorithm 
by means of the Banach fixed-point theorem and to 
ensure $q$-superlinear convergence to the problem solution 
for arbitrary starting values. By means of a localization technique, our 
Newton method can also be used to determine
solutions of fixed-point equations that are only locally contractive and 
not uniquely solvable. We apply our algorithm
to a quasi-variational inequality which arises
in thermoforming 
and which not only involves the obstacle problem as a 
source of nonsmoothness but
also
a semilinear PDE
containing a nondifferentiable Nemytskii operator. 
Our analysis is accompanied by numerical experiments 
that illustrate the mesh-independence and 
$q$-superlinear convergence of the developed solution algorithm.  
}



\maketitle
\section{Introduction}
\label{sec:1}

This paper is concerned with the design, analysis, and numerical realization of 
semismooth Newton methods
 for fixed-point equations
of the type
\begin{equation}
\label{eq:general_FP_H_intro}
\tag{F}
\text{Find } \bar x \in X \text{ such that } \bar x = H(\bar x)
\end{equation}
that are posed in a real Banach space $X$ and involve
a Newton differentiable
 operator $H\colon X \to X$.
A main focus of our work is on the case where the map $H$ can be written as the composition
\begin{equation}\label{eq:HasSPhi}
    H = S \circ \Phi
\end{equation}
of two Newton differentiable functions $S$ and $\Phi$, with a special emphasis on 
the situation where $S$ is the solution map of an elliptic variational inequality (VI) with pointwise constraints.
Our main motivation for considering this kind of structure 
is that it arises naturally when studying elliptic quasi-variational inequalities (QVIs) of obstacle type, i.e., 
variational problems of the form 
\begin{equation}
\label{eq:QVI}  
\tag{Q}
\begin{gathered}
\text{Find $u \in K(\Phi(u))$ such that } 
~\langle -\Delta u - f, v - u \rangle_{H^{-1}(\Omega),H_0^1(\Omega)} \geq 0~
\forall v \in K(\Phi(u)),
\\
\text{ with }
K(\Phi(u)) := 
\{
        v \in H_0^1(\Omega) \mid v 
        \leq \Phi(u) + \Phi_0 \text{ a.e.\ in }\Omega
 \},
\end{gathered}
\end{equation}
where $\Omega \subset \R^d$ for $d \in \N$ denotes a nonempty open bounded set,
the Sobolev space 
$H_0^1(\Omega)$ is defined as in \cite[\S 5.1]{Attouch2006},  
$\Phi\colon H_0^1(\Omega) \to H_0^1(\Omega )$ is a given operator,
$f$ and $\Phi_0$ are given functions,
 and the symbols  $\Delta$ and $\smash{\langle \cdot, \cdot \rangle_{H^{-1}(\Omega),H_0^1(\Omega)}}$
 denote the distributional Laplacian and the dual pairing in $H_0^1(\Omega)$, respectively (see \cref{sec:obstacle-type-QVIs} for the precise setting). Indeed, 
if we define $S$ to be the solution operator $S\colon H_0^1(\Omega) \to H_0^1(\Omega)$,
$\phi \mapsto u$, of the classical obstacle problem 
 \begin{equation}
\label{eq:upper_obst_prob_again}  
\text{Find $u \in K(\phi)$ such that  }
~\langle -\Delta u - f, v - u \rangle_{H^{-1}(\Omega),H_0^1(\Omega)} \geq 0~
\forall v \in K(\phi),
\end{equation}
then \eqref{eq:QVI} can clearly be recast as $u =H(u)$ 
with $H := S \circ \Phi\colon H_0^1(\Omega) \to H_0^1(\Omega)$,
and thus \eqref{eq:QVI} can immediately be seen as a fixed-point equation of the form 
\eqref{eq:general_FP_H_intro}. 

Note that the salient feature of QVIs of the type 
\eqref{eq:QVI} is that the solution $u$ 
enters the problem not only via the Laplacian 
but also via the obstacle $\Phi(u) + \Phi_0$
defining the pointwise constraint in $K(\Phi(u))$. 
This creates a variational structure 
in which the set of admissible test functions 
depends implicitly on the problem solution 
and which distinguishes 
\eqref{eq:QVI} quite drastically
from standard partial differential equations (PDEs) and VIs. 
In applications, the dependence of the admissible set $K(\Phi(u))$ of \eqref{eq:QVI}
on  
$u$ allows to incorporate feedback effects into the 
problem formulation that reflect, for instance, how the obstacle 
$\Phi(u) + \Phi_0$ interacts with $u$
in zones of contact. 
Because of the ability to capture 
such an 
interplay between the problem solution and the 
constraints, 
QVIs have proved to be important instruments for modeling 
processes in various areas of physics and 
economics, e.g., thermoforming \cite{AHR}, 
sandpile growth 
\cite{Barrett2013,Barret2013, Prigozhin1996-1}, 
impulse control \cite{Bensoussan1975,Bensoussan1975-2,LionsBensoussan,Lions1986}, 
and 
superconductivity \cite{Prigozhin1996-2,Barret2010,Rodrigues2000}. 
Compare also with the classical works 
\cite{LionsBensoussan, MR556865, Mosco, MR745619} in this context.


While the feedback effects in a QVI like \eqref{eq:QVI} 
are very desirable from the application and modeling point of view, 
mathematically, 
they often pose serious challenges. 
Establishing the
Hadamard well-posedness of problems of the type \eqref{eq:QVI}, 
for example, is typically a hard task.
In fact,
in many situations, it is
 possible  
that a QVI possesses
multiple solutions---indeed a whole 
continuum of solutions---or no solutions at all. 
Compare, e.g., with the results 
on the Lipschitz and differential stability 
of QVI-solutions in 
\cite{AHR,Alphonse2022,ChristofWachsmuthQVI2022,AHRW,Alphonse2022-2,WachsmuthQVIs,Alphonse2020} 
in this context, 
and in particular with the examples in 
\cite[\S 6.1]{ChristofWachsmuthQVI2022} and \cref{subsec:7.3}.
The implicit relationship between 
the problem solution and the set of test functions 
also causes the numerical solution of QVIs to be a very delicate topic.
As far as problems posed in infinite-dimensional spaces are concerned, 
the algorithmic approaches
that are currently used for this 
purpose in the literature 
are primarily based on 
fixed-point arguments 
or regularization/penalization techniques---and, as a consequence, 
 are either slow or inexact.
See, for example, 
\cite[Chapter IV]{Bensoussan1982},
\cite[\S 2.1]{Alphonse2022},
\cite[\S 6.4]{AHR},
\cite[\S 5]{Kanzow2019},
\cite{Zeng1998},
and the references therein
for particular instances of such algorithms.

A main goal of the present paper is to show that  
recent advances in the field of generalized differentiability 
properties of solution maps of obstacle-type VIs make it possible to 
set up and analyze 
semismooth Newton methods 
for the numerical solution of QVIs of the type \eqref{eq:QVI}
and, thus, 
to solve obstacle-type quasi-variational inequalities 
 in function space with superlinear convergence speed. 
More generally, we develop a semismooth Newton framework 
for fixed-point equations of the type \eqref{eq:general_FP_H_intro}
that is tailored to applications
in the field of obstacle-type QVIs. A main feature of our semismooth Newton method for \eqref{eq:general_FP_H_intro}  
is that it is provably locally $q$-superlinearly convergent, 
robust with respect to inexactness,
and 
mesh-independent. Moreover, if a certain contractivity assumption holds, the algorithm can be made globally convergent;
see \cref{th:convergence,th:obst_QVI}. 
Note that the inclusion of inexactness is of special importance in 
the context of obstacle-type QVIs as 
evaluations of the outer map $S$
in \eqref{eq:HasSPhi}
arising in the context of  \eqref{eq:QVI}
are typically subject to numerical 
errors due to a discretization of \eqref{eq:upper_obst_prob_again}
or the introduction of an easy-to-evaluate surrogate model $S_\epsilon$;
see \cref{rem:inexactness}. 

To the best of our 
knowledge, this paper is the first to 
develop a semismooth Newton method
that is suitable for the solution of unregularized obstacle-type QVIs
in the infinite-dimensional setting and provably 
$q$-superlinearly convergent.
For related approaches in finite dimensions,
we refer the reader to  
\cite{Facchinei2015,MandlmayrPhd2018,Ni2014,Izmailov2012} and the references therein.
We remark that the example that we construct in 
\cref{subsec:7.3} to obtain an instance of a (generalized) thermoforming QVI
with multiple solutions
is also of independent interest as it provides 
an important benchmark problem for numerical solution 
algorithms. 

As we conduct our convergence analysis 
in a general Banach space setting and 
for the abstract fixed-point equation \eqref{eq:general_FP_H_intro}, 
the results that we prove in this paper 
are, of course, not only applicable to QVIs 
but also to other problems with comparable 
continuity and contractivity properties. We mention exemplarily 
VIs and PDEs with semilinearities, implicit VIs, and
operator equations that arise as optimality conditions 
of optimal control problems with
$H_0^1(\Omega)$-controls; see \cite[\S 5]{ChristofWachsmuthUniObstacle}. 

\subsection{Main results}
The main results of this paper are concerned with 
the convergence of semismooth Newton methods
for the solution of fixed-point equations of the type 
\eqref{eq:general_FP_H_intro} and the applicability 
of the developed abstract theory to 
QVIs of the form \eqref{eq:QVI}.
For the highlights, we refer the reader to:
\smallskip
\begin{itemize}
  \setlength\itemsep{0.2cm}
    \item \cref{th:convergence}, which establishes
    the global $q$-superlinear convergence/finite convergence 
    to a given tolerance of an inexact globalized 
    semismooth Newton method (\cref{algo:semiNewtonAbstract}) 
    for the solution of \eqref{eq:general_FP_H_intro}.
    This result relies on the assumption that $H\colon X \to X$
    is Newton differentiable and 
    globally contractive, i.e., $\gamma$-Lipschitz for some $\gamma \in [0,1)$; 
    see \cref{ass:standing_general_H}. 

	\item \cref{th:loc_conv}, which localizes \cref{algo:semiNewtonAbstract}
	by means of a projection 
	onto a
  nonempty closed convex set
 $B \subset X$. This result makes it possible to apply our convergence analysis 
	to equations \eqref{eq:general_FP_H_intro} that satisfy 
	a contraction assumption only locally and have multiple 
	solutions, a structure that is prevalent
	in many QVI-applications; cf.\  \cite{AHR,WachsmuthQVIs,Alphonse2020,AHRW,Alphonse2022-2,Alphonse2022}.
	For the precise setting for this result, see \cref{ass:standing_general_proj_H}.

\item \cref{th:obst_QVI,th_T_probs,th_T_probs_1d_local},
which demonstrate that obstacle-type QVIs
of the form \eqref{eq:QVI} are indeed covered 
by our general abstract semismooth Newton framework provided 
the involved quantities are sufficiently well behaved
(see \cref{ass:standing_obstacle_map,ass:SemilinearObstacleMap}). 
\end{itemize}
\medskip

We remark that the numerical realization of semismooth Newton methods 
for obstacle-type QVIs is also an interesting field on its own, 
in particular as the residue function $R(x) := x - H(x)$ arising 
in the context of problems like \eqref{eq:QVI}   
involves the solution map $S$ of the variational inequality 
\eqref{eq:upper_obst_prob_again}  and since the 
Newton derivatives that appear depend on the active, inactive,
and strictly active set of the current Newton iterate. 
For a detailed discussion of how we tackle these challenging problems 
in our numerical implementation, we refer the reader to \cref{sec:7}.

\subsection{Notation and basic concepts}
\label{sec:2}

Throughout this paper, 
we denote by $\|\cdot\|_X$ the norm of a (real) normed vector space $X$. For the closed 
ball of radius $r>0$ centered at a point $c \in X$,
we write $B_r^X(c)$. A sequence $\{x_n\} \subset X$ is said to converge 
$q$-superlinearly to $x \in X$ if $x_n \to x$ and 
$\|x_{n+1} - x\|_X \leq o(1)\|x_n - x\|_X$ 
hold for $n \to \infty$,
where
the Landau notation $o(1)$ 
represents a term that vanishes in the limit. 
Given two normed spaces $X$ and $Y$,  
we use the symbol $\LL(X,Y)$ to denote the 
space of linear and continuous functions from $X$ to $Y$.
We write 
$X^* := \LL(X, \R)$ for the topological dual space 
of $X$. 
The evaluation of an element $x^* \in X^*$
at $x \in X$ is denoted by the dual pairing 
$\langle x^*, x\rangle_{X^*,X}$. 
For the identity map, 
we use the symbol $\Id$. 
If $X$ is Hilbert, then 
$(\cdot,\cdot)_X$ denotes
the inner product of $X$,
$V^\perp$ stands for the orthogonal 
complement of a closed subspace $V$ of $X$,
and $P_B(x) := \mathrm{argmin}_{z \in B} \|x - z\|_X$ is
 the metric projection 
onto a closed convex nonempty set $B \subset X$.

Given mappings $F\colon X \to Y$ and $G\colon Y \to Z$
between normed spaces $X$, $Y$, and $Z$,
the composition 
of $G$ and $F$
is denoted by $G\circ F\colon X \to Z$. In the case of linear 
operators, the symbol $\circ$ is often dropped. 
The image of a set $D\subset X$
under $F$ is denoted by $F(D)$.
Recall that a function $F\colon D \to Y$
defined on a nonempty subset $D$ of a normed space $X$
with values in 
a normed space $Y$ is 
called \emph{Newton differentiable} 
with \emph{(Newton) derivative}
$G_F\colon D \to \LL(X,Y)$ if\vspace{-0.1cm}
\begin{equation}
\label{eq:NewtonDiffDef}
\lim_{\substack{0 <  \|h\|_{X} \to 0,\\x + h \in D}}
\frac{\|F(x + h) - F(x) - G_F(x+h)h\|_{Y}}{\|h\|_{X}} = 0\qquad \forall x \in D.\vspace{-0.1cm}
\end{equation}
We remark that Newton derivatives 
are often defined 
as set-valued mappings in the literature;
see, e.g., \cite[Definition 2.11]{ChristofWachsmuthUniObstacle}---we make a slight abuse of notation and assume $G_F$ is the realization of one of the elements in the set.

Given a nonempty open set $\Omega \subset \R^d$,
$d \in \N$, we denote by 
$C(\bar \Omega)$, 
$C_c^\infty(\Omega)$, and 
$C^m(\Omega)$ for $m \in \N$
the usual subspaces of the space 
$C(\Omega)$ of real-valued continuous functions on $\Omega$.
For the boundary, 
Lebesgue measure, and diameter
of $\Omega$, we use the symbols $\partial \Omega$,
$|\Omega|$, and $\diam(\Omega)$, respectively.
The real Lebesgue and Sobolev spaces 
on $\Omega$ are denoted as usual 
by $L^p(\Omega)$, $W^{m,p}(\Omega)$, and $H^m(\Omega)$ for 
$m \in \N$, $p \in [1,\infty]$.
If $\Omega$ is bounded, 
then we
define $H_0^1(\Omega)$
to be the closure of 
$C_c^\infty(\Omega)$ in 
\smash{$(H^1(\Omega), \norm{\cdot}{H^1(\Omega)})$} and endow it
with the norm
$ \|v\|_{H_0^1(\Omega)} := 
\| |\nabla v |\|_{L^2(\Omega)}$.
Here, 
$\nabla$ denotes the weak gradient and
$|\cdot|$ the 
 Euclidean norm.
 We write $\smash{H^{-1}(\Omega)}$ for the dual space of $H_0^1(\Omega)$. 
 The (distributional)
 Laplacian and the (weak) normal derivative
 are denoted by $\Delta$ and $\partial_\nu$, respectively. 
 For $d=1$, derivatives are denoted by a prime. We use $C_P(\Omega)$ to denote the constant in the Poincaré--Friedrichs inequality $\|v\|_{L^2(\Omega)} \leq C_P(\Omega) \| v \|_{H_0^1(\Omega)}$ for $v \in H^1_0(\Omega)$.

 Given a locally
 Lipschitz continuous function $g\colon \R \to \R$,
 we  define 
 $\mathrm{Lip}(g, [a,b])
 := \min\{c \geq 0 \colon$ $|g(s_1) - g(s_2)| \leq 
 c|s_1 - s_2|~\forall s_1, s_2 \in [a,b]\}$
 to be the Lipschitz constant of $g$ on $[a,b]$,
 $a < b$. 
 In this case, we further write 
 $\partial_c g(x) \subset \R$
 for Clarke's generalized differential of $g$
 at $x$ in the sense of \cite[\S 2.1]{Clarke:1990}.
 If $g$ is globally Lipschitz,
 then $\Lip{g}$ denotes
 the Lipschitz constant of $g$ on 
$\R$. 

\section{Semismooth Newton methods for fixed-point problems}
\label{sec:framework_fp_H}
In this section, 
we develop an inexact 
semismooth Newton framework for the obstacle-type QVI \labelcref{eq:QVI} by addressing the more general 
fixed-point equation \labelcref{eq:general_FP_H_intro}.

\subsection{Vanilla inexact semismooth Newton method}\label{sec:vanillaSSN}
We begin with \cref{algo:semiNewtonAbstractGeneral} which constitutes a 
standard inexact semismooth Newton method for solving \labelcref{eq:general_FP_H_intro} and its convergence properties are stated in \cref{th:convergenceGeneral}. Throughout this paper, we use the symbol $R$ to denote the residue function
\begin{align*}
R\colon X \to X,\qquad R(x) := x - H(x),
\end{align*}
of the equation \eqref{eq:general_FP_H_intro}.
\vspace{-0.25cm}

\begin{algorithm}[H]
\caption{Vanilla inexact semismooth Newton method for the solution of \eqref{eq:general_FP_H_intro}}
\label{algo:semiNewtonAbstractGeneral}
\begin{algorithmic}[1]
 \State {\bf Input:} Initial guess $x_0 \in X$, 
                tolerance $\Ntol \geq 0$,
               and sequence  $\{\rho_i\} \subset [0,\infty)$.
\State {\bf Output:} $x^* \in X$ satisfying $\|R(x^*)\|_X \leq \Ntol$, where $R(x) := x - H(x)$.
    \For{$i = 0,1,2,3,\ldots$}
        \If{$\|R(x_i)\|_{X} \leq \Ntol$}\label{algo1:term_general}
            \State Set $x^* := x_i$
             and stop the iteration (convergence reached). 
        \Else
        \State\label{algo1:x_N_general}
            Compute $z_i \in X$ that satisfies 
            $R(x_i) + G_R(x_i)z_i \approx 0$ in the following sense:
            \begin{align*}
                \|R(x_i) + G_R(x_i)z_i\|_X \leq \rho_i \|R(x_i)\|_X.
            \end{align*}
        \State\label{algo1:update_formula}
        Set $x_{i+1} := x_i + z_i$.
        \EndIf  
    \EndFor
\end{algorithmic}
\end{algorithm}
\vspace{-0.4cm}

\begin{theorem}[{Local convergence of 
\cref{algo:semiNewtonAbstractGeneral}}]
\label{th:convergenceGeneral}
Consider
the fixed-point equation 
\eqref{eq:general_FP_H_intro} involving 
 a Banach space $X$
and a map $H\colon X \to X$. 
Let $R\colon X \to X$, $R(x) := x - H(x)$, denote the residue function of \eqref{eq:general_FP_H_intro}. 
Suppose that the following holds:
\begin{enumerate}[label=\roman*)]
\item $B \subset X$ is an open set 
containing a solution $\bar x$ of \eqref{eq:general_FP_H_intro}, i.e., $\bar x=H(\bar x)$;
\item $R\colon X \to X$ is Newton differentiable 
on $B$ with Newton derivatives
$G_R(x)$ for $x \in B$;
\item\label{ass:general_inexact:ii} 
$R\colon X \to X$ is $L$-Lipschitz on $B$ for some 
 $L \in [0,\infty)$, i.e., 
\begin{equation}
\label{eq:Phi_L_Lip_inexact}
\| R(x_1) - R(x_2)\|_X
\leq
L \|x_1 - x_2\|_X
\quad
\forall x_1, x_2 \in B;
\end{equation}
\item $G_R(x)$ is invertible for all $x \in B$  and there exists a number $M \in [0,\infty)$ such that 
\[\norm{G_R(x)^{-1}}{\LL(X,X)} \leq M \quad \forall x \in B;\]
\item $\{\rho_i\}$ satisfies $\{\rho_i\} \subset [0, \rho^*]$ for some 
$\rho^* \in [0,\infty)$ with $ML\rho^* < 1$.
\end{enumerate}\pagebreak
Then there exists $r >0$ such that
the standard semismooth Newton method, i.e.,    \cref{algo:semiNewtonAbstractGeneral}, 
satisfies the following for all $x_0 \in B_r^X(\bar x)$:%
\begin{enumerate}[label=\roman*)]
\item\label{vanilla:assertion:i}
If $\Ntol > 0$ holds, then  \cref{algo:semiNewtonAbstractGeneral} 
terminates after finitely many steps.
\item\label{vanilla:assertion:ii} If $\Ntol = 0$ holds, then \cref{algo:semiNewtonAbstractGeneral}  
produces  iterates 
that converge finitely or q-linearly to $\bar x$.
\item\label{vanilla:assertion:iii} 
If $\Ntol = 0$ holds and $\rho_i \to 0$ for $i \to \infty$, 
 then \cref{algo:semiNewtonAbstractGeneral}  
produces a sequence of iterates $\{x_i\}$
that converges finitely or $q$-superlinearly to $\bar x$.
\end{enumerate}
\end{theorem}
Since the proof is relatively standard, we give it in \cref{sec:proofVanillaSSN}.  A few remarks:
\smallskip
\begin{itemize}
\setlength\itemsep{0.1cm}
\item \textbf{Selection principle.} We assume that a selection principle has been applied for the Newton derivative and hence $G_R(x_i)$ is not set-valued.
\item \textbf{Implementation of step \ref{algo1:x_N_general} of \cref{algo:semiNewtonAbstractGeneral}.} The realization of step \ref{algo1:x_N_general} (evaluating $R(x_i)$ and computing the action of $G_R(x_i)$ or $G_R(x_i)^{-1}$, respectively) is dependent on the precise form of $R$, $G_R$, and $X$. 
In \cref{subsec:7.1}, we detail how to implement step \ref{algo1:x_N_general} for a piecewise (bi)linear finite element discretization for a particular instance
 of the obstacle-type QVI \eqref{eq:QVI}.
\item \textbf{Measure of inexactness.}
Expressing the accuracy 
requirements on the update steps in terms of the norm of the residue is common practice in the 
analysis of inexact Newton methods;
cf.\ \cite{Dembo1982} and \cite[\S 3.2.4]{Ulbrich2011}.
We remark that, in the context of obstacle-type QVIs, 
ensuring that the iterates are sufficiently accurate can be a 
delicate matter. For details on this topic, we refer to \cref{sec:7}.
\item \textbf{Choice of the forcing sequence $\boldsymbol{\{\rho_i\}}$.}
Some choices for the forcing sequence $\{\rho_i\}$ are
\[
\rho_i = \norm{x_i-x_{i-1}}{X}
\qquad
\text{or}
\qquad
\rho_i = \min(\norm{x_i-x_{i-1}}{X}, \alpha_i),
\]
where $\{\alpha_i\} \subset (0,\infty)$ 
is a sequence satisfying $\alpha_i \to 0$ used for safeguarding.

Note that the forcing sequence $\{\rho_i\}$ must be chosen carefully to avoid oversolving. If it decays too quickly, then the convergence rate does not justify the computational cost whereas decaying too slowly will impair the $q$-superlinear convergence.
\item \textbf{Globalization.}
\cref{algo:semiNewtonAbstractGeneral} only guarantees convergence if the initial guess is sufficiently close to the solution. Globalizing semismooth Newton methods is a delicate topic and there is often a tradeoff between the assumptions a particular globalization technique requires and its computational cost. A globalization will likely require many more evaluations of $H$ which might be prohibitively expensive. For example, in the context of obstacle-type QVIs \labelcref{eq:QVI}, each evaluation of $H$ requires the solve of an obstacle problem. In the next subsection, we explore a cheap globalization technique that requires a contraction assumption. We also consider a 
globalization based on a merit function and an Armijo linesearch \cite{nocedal2000, gerdts2017} in \cref{sec:7}. 
\end{itemize}


\subsection{Global \texorpdfstring{$q$-superlinear~}~convergence for contractive equations}
\label{subsec:global_contraction}

In this subsection, we show that if  \eqref{eq:general_FP_H_intro} satisfies a contraction condition, then a small modification of \cref{algo:semiNewtonAbstractGeneral} will guarantee global convergence. Consider the following assumptions:

\begin{assumption}[Global contraction assumptions]
\label{ass:standing_general_H}~
\begin{enumerate}[label=\roman*)]
\item
X is a Banach space;
\item\label{ass:standing_general:Hii} 
 $H\colon X \to X$ is a Newton differentiable
function with Newton derivative
$G_H\colon X \to \LL(X, X)$
and the residue function $R\colon X \to X$ is 
endowed with the Newton derivative $G_R := \Id - G_H$;
\item\label{ass:standing_general:Hiii} 
there exists $\gamma \in [0,1)$ such that 
$H \colon X \to X$ is globally $\gamma$-Lipschitz, i.e., 
\begin{equation}
\label{eq:H_gamma_Lip}
\| H(x_1) - H(x_2)\|_X
\leq
\gamma \|x_1 - x_2\|_X
\quad
\forall x_1, x_2 \in X,
\end{equation}
and
\begin{equation}
\label{eq:GH_bound-2}
\sup_{x \in X} \left \|G_H(x)\right\|_{\LL(X,X)} \leq \gamma.
\end{equation}
\end{enumerate}
\end{assumption}
Note that the contraction conditions in point \ref{ass:standing_general:Hiii} 
of \cref{ass:standing_general_H} are restrictive in general applications.
However, in the field of QVIs,
they 
are assumed anyway in many situations 
to guarantee the Hadamard well-posedness of the problem; 
see \cite{AHR,WachsmuthQVIs,Alphonse2020,AHRW,Alphonse2022-2,Alphonse2022}.
Regarding 
\eqref{eq:GH_bound-2}, it should be noted that,
in practice, the uniform $\gamma$-bound on 
$G_H(x)$ is often a 
direct consequence of the Lipschitz estimate \eqref{eq:H_gamma_Lip}. 
In \cref{sec:framework_fp_H_multiple_solns}, we discuss techniques for localizing the
contractivity assumption \eqref{eq:H_gamma_Lip}.
 
We begin our analysis by noting 
that, in the situation of \cref{ass:standing_general_H},
the existence and uniqueness of solutions 
of \eqref{eq:general_FP_H_intro} are immediate consequences 
of the Banach fixed-point theorem.

\begin{lemma}[Unique solvability of \eqref{eq:general_FP_H_intro}]
\label{lem:uniqueF_H}
Suppose that \cref{ass:standing_general_H} holds. Then the problem \eqref{eq:general_FP_H_intro} has a unique solution $\bar x \in X$.
\end{lemma}
 
The globalized semismooth inexact Newton method
that we propose for the solution of 
\eqref{eq:general_FP_H_intro}, under the assumptions \eqref{eq:H_gamma_Lip} and \eqref{eq:GH_bound-2}, is stated in
\cref{algo:semiNewtonAbstract}.
\vspace{-0.2cm}

\begin{algorithm}[H]
\caption{Globally convergent inexact semismooth Newton method for the solution of \eqref{eq:general_FP_H_intro}}
\label{algo:semiNewtonAbstract}
\begin{algorithmic}[1]
 \State {\bf Input:} Initial guess $x_0 \in X$, arbitrary constant $\tau^* \in [0,\infty]$, 
                tolerance $\Ntol \geq 0$,
               and sequences $\{\tau_i\} \subset [0,\infty) \cap [0,\tau^*], \{\rho_i\} \subset [0,\infty)$
               satisfying $\tau_i \to 0, \rho_i \to 0$.
\State {\bf Output:} $x^* \in X$ satisfying $\|R(x^*)\|_X \leq \Ntol$.
    \For{$i = 0,1,2,3,\ldots$}
        \If{$\|R(x_i)\|_{X} \leq \Ntol$}\label{algo1:term}
            \State Set $x^* := x_i$
            and stop the iteration (convergence reached). 
        \Else
        \State\label{algo1:x_B}
            Compute $x_B \in X$ that satisfies 
             $x_B \approx H(x_i)$
            in the following sense:
            \begin{equation}
            \label{eq:xB_update}
                \|x_B - H(x_i) \|_X \leq 
                \tau_i\|R(x_i)\|_X.
            \end{equation}
        \State\label{algo1:x_N}
        	Compute $x_N \in X$ that satisfies 
             $R(x_i) + G_R(x_i)(x_N -x_i) \approx 0$  in the following sense:
            \begin{equation}
            \label{eq:xN_update}
                \|R(x_i) + G_R(x_i)(x_N -x_i)\|_X \leq \rho_i \|R(x_i)\|_X.
            \end{equation}
              \If{$\|R(x_N)\|_{X} \leq \|R(x_B)\|_{X}$}\label{algo1:decision}
              \State Set $x_{i+1} := x_N$.
        \Else
              \State Set $x_{i+1} := x_B$.
        \EndIf
        \EndIf
    \EndFor
\end{algorithmic}
\end{algorithm}
\vspace{-0.2cm}

Before we state the convergence properties of \cref{algo:semiNewtonAbstract}, a few remarks:
\smallskip
\begin{itemize}
\setlength\itemsep{0.1cm}
\item\textbf{Vanilla semismooth Newton.}  
By choosing a very large constant $\tau^*$,
a sequence $\{\tau_i\}$ that goes to zero very slowly, 
and trial iterates $x_B$ with large residues $\|R(x_B)\|_{X}$,
one can essentially switch off the safeguarding by 
means of the fixed-point iterates in \cref{algo:semiNewtonAbstract}. 
If run in such a configuration, 
\cref{algo:semiNewtonAbstract}
effectively behaves like a vanilla semismooth Newton method
in numerical experiments, i.e., like \cref{algo:semiNewtonAbstractGeneral} in \cref{sec:vanillaSSN}. 
 \item \textbf{Choice of the sequence $\boldsymbol{\{\tau_i\}}$.}
 The proof of \cref{th:convergence} on the convergence of \cref{algo:semiNewtonAbstract} below
 hinges on the fact that, either for all or for all sufficiently large $i$,
 one has 
 $ \tau_i + \gamma + \gamma\tau_i < 1$, where $\gamma$ is the contraction factor.
 If, when solving a particular problem, the value of $\gamma$ is known, then one can simply choose the constant sequence $\tau_i  = (\lambda-\gamma)\slash(1+\gamma)$ for an arbitrary $\lambda \in (\gamma, 1)$ to ensure this inequality and then the property $\tau_i \to 0$ is not needed. 

\item \textbf{Mesh-independence.}
Note that the bound on the iteration index in \eqref{eq:iter_bound} below
is independent of discretization quantities. This is a strong indicator for
mesh-independence.
\end{itemize}

\begin{theorem}[{Finite and global 
$q$-superlinear convergence of 
\cref{algo:semiNewtonAbstract}}]
\label{th:convergence}
Suppose that \cref{ass:standing_general_H} holds. Let $x_0 \in X$ and $\tau^* \in [0,\infty]$ be arbitrary. Let $\{\tau_i\} \subset [0,\infty) \cap [0, \tau^*]$ and $\{\rho_i\} \subset [0,\infty)$
               satisfy $\tau_i \to 0$ and $\rho_i \to 0$. 

\begin{enumerate}[label=\roman*)]
\item\label{th:convergence:i}
If  $\Ntol > 0$ holds
and $\tau^*$ satisfies $\tau^* \leq (\lambda-\gamma)/(1+\gamma)$
for some $\lambda \in (\gamma, 1)$, then 
the termination criterion in line \ref{algo1:term}
of 
\cref{algo:semiNewtonAbstract} 
is triggered
for an iteration index $i \in \mathbb{N}_0$
satisfying
\begin{equation}
\label{eq:iter_bound}
0 \leq i \leq 
\begin{cases}
0 & \text{ if } \|R(x_0)\|_{X} \leq \Ntol,
\\
 \left \lceil\frac{\ln(\Ntol)- \ln(\|R(x_0)\|_X)}{\ln(\lambda) } \right \rceil 
& \text{ if } \|R(x_0)\|_{X} > \Ntol,
\end{cases}
\end{equation}
and the last produced iterate $x^* = x_i$ satisfies
\[
\|R(x^*)\|_{X} \leq \Ntol
\qquad
\text{and}
\qquad
\|x^* - \bar x\|_X \leq 
\smash{\frac{\Ntol}{1-\gamma}}.
\]
Here, $\lceil \cdot \rceil\colon (0, \infty)\to \mathbb{N}$ denotes the operation of rounding up
to the nearest larger integer. 
\item\label{th:convergence:ii} If $\Ntol = 0$ holds, then \cref{algo:semiNewtonAbstract} either terminates 
after finitely many steps with the solution $x^* = \bar x$ of \eqref{eq:general_FP_H_intro}
or produces a sequence of iterates $\{x_i\}$ that satisfies
\[
x_i \to \bar x \text{ $q$-superlinearly in $X$}
\qquad
\text{and}
\qquad
R(x_i) \to 0 \text{ $q$-superlinearly in $X$}.
\]
\end{enumerate}
\end{theorem}
The proof of \cref{th:convergence} is a careful intertwining of the convergence proofs of a 
fixed-point iteration and a semismooth Newton method. We defer it to \cref{sec:proofVanillaSSN}.

\subsection{Localization of the contraction assumptions and multiple solutions}\label{sec:framework_fp_H_multiple_solns}

Next, we discuss possibilities to localize the contraction conditions 
in \cref{ass:standing_general_H}\ref{ass:standing_general:Hiii} whilst still retaining global convergence. 
The assumptions of this subsection are
 motivated by results on QVIs with multiple,
locally stable solutions (see \cite{AHR,WachsmuthQVIs,Alphonse2020,AHRW,Alphonse2022-2,Alphonse2022})
and they make it possible to apply our Newton framework also 
to fixed-point equations \eqref{eq:general_FP_H_intro} 
that are not uniquely solvable. 
We consider the 
 following setting.
 
\begin{assumption}[Local contraction assumptions]
\label{ass:standing_general_proj_H}
~
\begin{enumerate}[label=\roman*)]
\item\label{ass:standing_general_proj_H:ii} $X$ is a Hilbert space;
\item $H\colon X \to X$ is a Newton differentiable
function with Newton derivative
$G_H\colon X \to \LL(X, X)$;
\item\label{ass:standing_general_proj_H:iv}
there exist a nonempty closed convex set 
$B \subset X$ and a number 
$\gamma \in [0,1)$ such that
\begin{equation}
\label{eq:H_gamma_Lip_loc}
\| H(x_1) - H(x_2)\|_X
\leq
\gamma \|x_1 - x_2\|_X
\quad
\forall x_1, x_2 \in B
\end{equation}
and 
\begin{equation}
\label{eq:GH_bound-2_loc}
\sup_{x \in B} \left \|G_H (x)\right\|_{\LL(X,X)} \leq \gamma;
\end{equation}
\item\label{ass:standing_general_proj_H:v} the metric projection $P_B\colon X \to B$
in $(X, \|\cdot\|_X)$ onto $B$ is Newton differentiable 
with Newton derivative 
$G_{P_B}\colon X \to \mathcal{L}(X,X)$
and it holds
$
\norm{G_{P_B}(x)}{\LL(X,X)} \leq 1
$
for all
$x \in X$.
 \end{enumerate}
\end{assumption}
We will show in \cref{lem:semismooth_balls} below that \cref{ass:standing_general_proj_H}\ref{ass:standing_general_proj_H:v} holds in particular if $B$ is a closed ball in $X$. 
\begin{remark}
    It is possible to drop the requirement
    that $X$ is Hilbert in 
    \cref{ass:standing_general_proj_H}.
    If this is done, however, one needs 
    additional assumptions on $X$ and $B$ to ensure that the projection $P_B$ is well defined,
    single-valued, and Lipschitz continuous; see 
    \cite{Alber1196} and the references therein.
    (Note that projections are 
    typically not one-Lipschitz in the Banach space setting; cf.\  
    \cite[Example 6.1]{Goebel2018}.)
    We focus on the Hilbert space case in this subsection because it simplifies
    the presentation, covers almost all 
    practical applications, and yields
    easier-to-track estimates due to 
    the non-expansiveness of $P_B$.
\end{remark}

The main idea of the following analysis 
is to resort to the global situation 
studied in \cref{subsec:global_contraction} by 
composing the function $H$
in \eqref{eq:general_FP_H_intro}
with the projection $P_B$. That is, 
instead of \eqref{eq:general_FP_H_intro}, we consider the 
fixed-point equation
\begin{equation}
\tag{F$_{loc}$}
\label{eq:fp_problem_proj_H}
\text{Find } \hat x \in X 
\text{ such that } \hat x = H_B(\hat x),
\end{equation}
where
$ H_B \colon X \to X$ 
is defined by  $ H_B := H \circ P_B$.
Note that this approach only works because
our algorithm is able to handle nonsmooth functions.
As we will see below,  
by applying \cref{algo:semiNewtonAbstract}
to the modified equation \eqref{eq:fp_problem_proj_H}, 
we obtain a numerical method 
that is able 
to determine precisely the 
intersection of the solution set 
$\{x \in X \mid x = H(x)\}$
of the fixed-point equation \eqref{eq:general_FP_H_intro}
with the set $B$.
In practical applications, 
the set $B$ in \eqref{eq:fp_problem_proj_H} 
could, for example, 
be a closed ball  
in which the estimates
in \eqref{eq:H_gamma_Lip_loc} 
and  \eqref{eq:GH_bound-2_loc}
can be proven to hold;
see \cref{subsec:7.3}.
If one is interested in an abstract 
local convergence result
similar to the classical one discussed in \cref{sec:vanillaSSN}, 
then one 
can also assume that $B$ is a closed ball $B_\varepsilon^X(\bar x)$ 
that is centered at an isolated solution $\bar x$ of 
\eqref{eq:general_FP_H_intro}. 
In the latter case, 
the conditions \eqref{eq:H_gamma_Lip_loc} 
and  \eqref{eq:GH_bound-2_loc}
take a form that is also often
encountered in the sensitivity analysis 
of obstacle-type QVIs; cf.\ \cite{AHR, WachsmuthQVIs, AHRW}. 
That \cref{ass:standing_general_H} holds for the composition  $ H_B = H \circ P_B$
is proven in the following lemma.

\begin{lemma}[Properties of $H_B$]
\label{lem:H_c_props}
Suppose that \cref{ass:standing_general_proj_H} holds. Then:
\begin{enumerate}[label=\roman*)]
\item\label{lem:H_c_props:i}
$ H_B$ is Newton differentiable 
with Newton derivative  
$G_{ H_B}(x) := G_{H}(P_B(x))G_{P_B}(x)$.
\item\label{lem:H_c_props:iii}
 It holds 
\begin{equation}
\label{eq:loc_contrac_H}
\| H_B(x_1) - H_B(x_2)\|_X
\leq
\gamma \|x_1 - x_2\|_X \qquad \forall x_1, x_2 \in X
\end{equation}
and
\begin{equation}
\label{eq:loc_contrac_2_H}
\sup_{x \in X}\norm{G_{H_B}(x)}{\LL(X,X)}  \leq \gamma.
\end{equation}
\end{enumerate}
\end{lemma}
\begin{proof}
Assertion \ref{lem:H_c_props:i}
follows from 
the conditions in \cref{ass:standing_general_proj_H}, 
the chain rule for Newton derivatives (see \cref{lem:chain_rule}),
and the global one-Lipschitz continuity of 
$P_B$. 
To prove 
\ref{lem:H_c_props:iii},
we first note that 
\eqref{eq:H_gamma_Lip_loc}, (again) 
the one-Lipschitz continuity of $P_B$,
and the definition of $H_B$ imply
\[
\| H_B(x_1) - H_B(x_2)\|_X
\leq
\gamma
\|  P_B(x_1) - P_B(x_2)\|_X
\leq
\gamma
\|x_1 - x_2\|_X\quad \forall x_1, x_2 \in X.
\]
This establishes \eqref{eq:loc_contrac_H}.
Similarly, we obtain from assertion \ref{lem:H_c_props:i} of this lemma, 
\eqref{eq:GH_bound-2_loc},
and
the bound $\norm{G_{P_B}(x)}{\LL(X,X)} \leq 1$
that
\begin{align*}
\norm{G_{H_B}(x)}{\LL(X,X)}  &= \norm{G_H(P_B(x))G_{P_B}(x)}{\LL(X,X)}
\leq \norm{G_H(P_B(x))}{\LL(X,X)}\norm{G_{P_B}(x)}{\LL(X,X)}
\leq \gamma
\end{align*}
holds for all $x \in X$.
This implies 
\eqref{eq:loc_contrac_2_H}
and completes the proof. 
\end{proof}
 \Cref{lem:H_c_props} shows that 
 \eqref{eq:fp_problem_proj_H} is 
 covered by the semismooth Newton framework 
 of \cref{subsec:global_contraction}. 
 This allows us to deduce the following from 
 \cref{th:convergence}.
 
\begin{theorem}[Convergence in the localized setting]
\label{th:loc_conv}
Suppose that \cref{ass:standing_general_proj_H} holds. Let $x_0 \in X$ and $\tau^* \in [0,\infty]$ be arbitrary. Let $\{\tau_i\} \subset [0,\infty) \cap [0, \tau^*]$ and $\{\rho_i\} \subset [0,\infty)$
               satisfy $\tau_i \to 0$ and $\rho_i \to 0$.  
 Then \cref{algo:semiNewtonAbstract},
 applied to the fixed-point problem  \eqref{eq:fp_problem_proj_H}
 with parameters 
 $x_0$, $\Ntol = 0$, $\tau^*$, $\{\tau_i\}$, and $\{\rho_i\}$, 
 converges finitely or $q$-superlinearly 
 to a point $\hat x \in X$ and the following is true:
\begin{enumerate}[label=\roman*)]
\item\label{th:loc_conv:i}
If $\hat x \in B$ holds, 
then \eqref{eq:general_FP_H_intro} possesses 
precisely one solution $\bar x$ in $B$ and
it holds $\bar x = \hat x$.
\item\label{th:loc_conv:ii}
If $\hat x \not\in B$ holds, 
then \eqref{eq:general_FP_H_intro} does not possess 
a solution in $B$.
\end{enumerate}
If, in addition,  
$H$ satisfies 
$H(B) \subset B$,
then only  case 
\ref{th:loc_conv:i} occurs.
\end{theorem}
\begin{proof}
As the results of \cref{subsec:global_contraction}
are applicable to \eqref{eq:fp_problem_proj_H} 
by \cref{ass:standing_general_proj_H} and 
\cref{lem:H_c_props},
we obtain from 
\cref{lem:uniqueF_H,th:convergence}
that 
there is a unique $\hat x \in X$
satisfying $\hat x = H_B(\hat x)$
and that \cref{algo:semiNewtonAbstract}
with $\Ntol = 0$
converges finitely or 
$q$-superlinearly to $\hat x$ when 
applied to \eqref{eq:fp_problem_proj_H}. 
Suppose now that $\hat x \in B$ holds. 
Then we have $\hat x = P_B(\hat x)$
and $\hat x$ is also a solution of 
\eqref{eq:general_FP_H_intro}. Moreover,
\eqref{eq:general_FP_H_intro} cannot possess 
any further solutions $\bar x \neq \hat x$
in $B$ as those would also 
solve \eqref{eq:fp_problem_proj_H} 
which is uniquely solvable by  
\cref{lem:uniqueF_H}.
This proves \ref{th:loc_conv:i}.
The assertion in \ref{th:loc_conv:ii}
is obtained along the same lines. 
That only case \ref{th:loc_conv:i} occurs
if $H(B) \subset B$ holds 
follows from the 
structure of \eqref{eq:fp_problem_proj_H}.
\end{proof}

Note that \cref{th:loc_conv} expresses that, 
if a sufficiently nice set $B\subset X$ (e.g., a ball)
satisfying 
\eqref{eq:H_gamma_Lip_loc} and
\eqref{eq:GH_bound-2_loc}
is given, 
then 
\cref{algo:semiNewtonAbstract}---applied to \eqref{eq:fp_problem_proj_H}---is able to determine precisely
whether \eqref{eq:general_FP_H_intro} possesses 
a solution in $B$ and, 
in the case of its existence, 
identify the unique solution 
of \eqref{eq:general_FP_H_intro} 
in $B$ with superlinear convergence speed.
In applications
in which it is important to
decide whether solutions  
are present in certain sets 
or to determine distinguished solutions of \eqref{eq:general_FP_H_intro} 
(e.g., maximal and minimal solutions),
this type of localization of 
the framework in 
\cref{subsec:global_contraction}
offers an attractive 
alternative to classical 
localization approaches; see 
\cite[Proof of Theorem~3.4]{ChenNashedQi2000}
and \cite[Proof of Theorem~3.13]{Ulbrich2011},
and compare also with
\cref{th_T_probs_1d_local}
and the experiments 
in \cref{subsec:7.3}.

We conclude this subsection by 
establishing that balls 
$B_r^X(c)$, $r>0$, $c \in X$, 
in $X$
indeed satisfy the 
conditions in \cref{ass:standing_general_proj_H}\ref{ass:standing_general_proj_H:v}. 
The proof of the following result relies heavily
on the chain and product rule for Newton derivatives.
We recall these calculus rules in \cref{sec:appendix} of this paper for the 
convenience of the reader. 

\begin{lemma}[Projections onto closed balls]
\label{lem:semismooth_balls}
Let $c \in X$ and $r>0$ be given. 
Define $B:=B_r^X(c)$. Then the projection 
$P_B\colon X \to X$ is Newton differentiable with Newton derivative 
\[
G_{P_B}(x)h := 
\begin{cases}
\displaystyle
h &\text{ if } \norm{x-c}{X} \leq r,\\
\frac{r}{\norm{x-c}{X}}
\left [h  -  
\left (\frac{x-c}{\norm{x-c}{X}}, h \right )_X \frac{x-c}{\norm{x-c}{X}} \right ] &\text{ if } \norm{x-c}{X} > r,
\end{cases}
\]
and it holds 
\begin{equation}
\label{eq:GPB_norm_bound}
\norm{G_{P_B}(x)}{\LL(X,X)} 
\leq \min \left (1,
\frac{r}{\norm{x - c}{X}}
\right ) \qquad \forall x \in X.
\end{equation}
\end{lemma}
\begin{proof}
Due to the chain rule in \cref{lem:chain_rule}, 
it suffices to prove the claim for 
$c=0$, i.e., $B = B_r^X(0)$. For such a set $B$, we have
\begin{equation}
\label{eq:radial_projection}
P_{B}(x) = 
\begin{cases}
x & \text{ if } \|x\|_X \leq r,\\
r\frac{x}{\norm{x}{X}} & \text{ if } \|x\|_X > r,
\end{cases}
\end{equation}
and, thus, $P_B(x) = g(\norm{x}{X}^2) x$
with 
$g\colon [0, \infty) \to \R$, 
$g(s) := r  \max(r, s^{1/2})^{-1}$. 
Note that $g$
is piecewise $C^1$ with a single kink at $s=r^2$.
This implies that 
$g\colon [0, \infty) \to \R$
is Newton differentiable with 
Newton derivative 
\[
G_g(s) := 
\begin{cases}
0  & \text{ if } 0 \leq s^{1/2} \leq r,\\
-\frac{r}{2s^{3/2}} &\text{ if } s^{1/2} > r;
\end{cases}
\]
see \cite[Proposition 2.1]{Pang1995}.
(This Newton differentiability can also be checked directly by a simple calculation.) As the 
function $X\ni x \mapsto \|x\|_X^2 \in [0,\infty)$
is smooth and locally Lipschitz continuous, 
it follows from the chain rule of \cref{lem:chain_rule} that 
$f\colon X \to \R$, $ x \mapsto g(\|x\|_X^2)$, 
is Newton differentiable with derivative 
\[
G_f(x)h := 
\begin{cases}
0  & \text{ if } 0 \leq \norm{x}{X} \leq r,\\
-\frac{r}{\norm{x}{X}^3}\left (x,h \right )_X &\text{ if } \norm{x}{X} > r.
\end{cases}
\]
As $f$ is continuous, the product rule of  \cref{lem:product_rule}
(with 
$U = X$,
$V = \R$,
$W = X$, 
$Z = X$,
$P = f$,
$Q = \Id$,
and $a$ as the multiplication with a scalar in $X$)
now yields that the map $P_B(x) = f(x)x$ is Newton 
differentiable with Newton derivative 
\begin{equation}
\label{eq:GP_again}
G_{P_B}(x)h := 
\begin{cases}
h  & \text{ if } 0 \leq \norm{x}{X} \leq r,\\
-\frac{r\left (x,h \right )_X }{\norm{x}{X}^3}x
+ \frac{rh}{\|x\|_X}
=
\frac{r}{\norm{x}{X}}
\left [
h - \left (\frac{x}{\|x\|_X}, h \right )_X \frac{x}{\|x\|_X}
\right ]
&\text{ if } \norm{x}{X} > r.
\end{cases}
\end{equation}
This proves the assertion on the Newton differentiability of $P_B$ and the formula for $G_{P_B}$.
To obtain \eqref{eq:GPB_norm_bound},
it suffices to note that the expression 
in the square brackets on the right-hand side of 
\eqref{eq:GP_again} is the 
projection of $h$ onto the orthogonal complement of the line $\R x \subset X$ 
and, thus, bounded in the $X$-norm by $\|h\|_X$. 
This completes the proof. 
\end{proof}

\begin{remark}
    If $X$ is merely Banach (and not necessarily Hilbert),
    then one can still proceed along the lines 
    of the proof of \cref{lem:semismooth_balls} 
    to establish that the radial projection 
    given by the formula 
    \eqref{eq:radial_projection} 
    is Newton differentiable, provided 
    the norm of $X$ is $C^1$ away from 
    the origin. If this is done, however, then
    one cannot bound the 
    $\LL(X,X)$-norm of the Newton derivatives $G_{P_B}(x)$
    by the right-hand side of \eqref{eq:GPB_norm_bound}
    but only by a worse constant;
    cf.\ the Lipschitz estimate proven in 
    \cite[\S 6.1]{Goebel2018}. In combination with the 
    lost non-expansiveness of $P_B$,
    this makes it necessary to impose more restrictive 
    assumptions on the number $\gamma$ in 
    \eqref{eq:loc_contrac_H} and
    \eqref{eq:loc_contrac_2_H} 
     if 
     a Banach space setting is considered. 
\end{remark}

\subsection{Composite fixed-point equations}
\label{sec:framework_applied_to_SPhi}

With the general convergence theory of 
\cref{sec:vanillaSSN,subsec:global_contraction,sec:framework_fp_H_multiple_solns} at hand, 
we can turn our attention to the special case
that the function $H$ in \eqref{eq:general_FP_H_intro}
is of the type $S\circ \Phi$, i.e., 
that the considered fixed-point equation has the form
\begin{equation}
\label{eq:SPhi_FP}
\tag{F$_c$}
\text{Find } \bar x \in X \text{ such that } \bar x = S(\Phi(\bar x)).
\end{equation}
As discussed before, this problem formulation is motivated by 
the structure of the QVI \eqref{eq:QVI} which 
can be recast as an equation
of the type \eqref{eq:SPhi_FP}
by means of the solution operator $S$
of the variational inequality 
\eqref{eq:upper_obst_prob_again} and the inner obstacle map $\Phi$;
see the concrete examples in \cref{sec:obstacle-type-QVIs,sec:7}
and the discussion in \cref{sec:1}.  
The next three corollaries make precise under 
which assumptions on the functions $S$ and $\Phi$ the problem 
\eqref{eq:SPhi_FP} is covered by the convergence results 
in \cref{th:convergenceGeneral,th:convergence,th:loc_conv}. Their proofs 
boil down to applications of the chain rule for Newton derivatives (\cref{lem:chain_rule})
and elementary estimates and are thus omitted. 

\begin{corollary}[Local convergence of \cref{algo:semiNewtonAbstractGeneral} for \eqref{eq:SPhi_FP}]
\label{cor:composite_classical}
Suppose that the following assumptions are satisfied:
\begin{enumerate}[label=\roman*)]
\item\label{ass:standing_general:i}  
$X$ and $Y$ are real Banach spaces, and $D \subset Y$ is a nonempty set;
\item\label{ass:standing_general:iv}  $\Phi\colon X \to D$ is 
a Newton differentiable function 
with Newton derivative $G_\Phi \colon X \to \LL(X,Y)$
and, for every $x \in X$, there exist constants  $ C, \varepsilon > 0$ satisfying
\begin{equation*}
\|  \Phi(x + h)  - \Phi(x) \|_Y
\leq
C\|h\|_X
\quad
\forall h \in B_\varepsilon^X(0);
\end{equation*}
\item\label{ass:standing_general:iii}  
$S\colon D \to X$ is a Newton differentiable
function with Newton derivative
$G_S\colon D \to \LL(Y, X)$
and, for every $y \in D$,  there exist constants  $ C, \varepsilon > 0$ satisfying
\begin{equation*}
	\sup_{w \in D \cap B_\varepsilon^Y(y)} \left \| G_S(w) \right \|_{\LL(Y, X)} \leq C;
\end{equation*}
\item there exist a nonempty open set $B \subset X$ and $\bar x \in B$ such that $\bar x=S(\Phi(\bar x))$;
\item\label{ass:general_inexact:ii_top} 
there exists a number $L \in [0,\infty)$ such that 
$R = \Id - S \circ \Phi\colon X \to X$ is $L$-Lipschitz on $B$, i.e., 
\begin{equation*}
\| R(x_1) - R(x_2)\|_X
\leq
L \|x_1 - x_2\|_X
\quad
\forall x_1, x_2 \in B;
\end{equation*}
\item $G_R(x) := \Id - G_S(\Phi(x))G_\Phi(x) \in \LL(X,X)$ is invertible for all $x \in B$  and there exists a number $M \in [0,\infty)$ with
\[\norm{G_R(x)^{-1}}{\LL(X,X)} \leq M \quad \forall x \in B;\]
\item $\{\rho_i\}$ satisfies $\{\rho_i\} \subset [0, \rho^*]$ for some 
$\rho^* \in [0,\infty)$ with $ML\rho^* < 1$.
\end{enumerate}
Then the convergence result in \cref{th:convergenceGeneral} 
applies to \eqref{eq:SPhi_FP}. In particular, the sequence of iterates $\{x_i\}$ 
produced by \cref{algo:semiNewtonAbstractGeneral} 
converges finitely/$q$-superlinearly to $\bar x$ 
if $x_0$ is sufficiently close to $\bar x$,
$\Ntol$ is chosen as zero,
and the forcing sequence $\{\rho_i\}$ 
satisfies $\rho_i \to 0$.
\end{corollary}

\begin{corollary}[Global convergence of \cref{algo:semiNewtonAbstract} for \eqref{eq:SPhi_FP} with global contraction]
\label{cor:composite_global}
Suppose that the following assumptions are satisfied:
\begin{enumerate}[label=\roman*)]
\item $X$, $Y$, $D$, $S$, and $\Phi$
satisfy the conditions 
in points 
\ref{ass:standing_general:i} to 
\ref{ass:standing_general:iii}
of \cref{cor:composite_classical};
\item\label{ass:standing_general:v} 
there exists $\gamma \in [0,1)$ such that 
$S\circ \Phi \colon X \to X$ is globally $\gamma$-Lipschitz, i.e., 
\begin{equation}
\label{eq:Phi_gamma_Lip}
\| S(\Phi(x_1)) - S(\Phi(x_2))\|_X
\leq
\gamma \|x_1 - x_2\|_X
\quad
\forall x_1, x_2 \in X,
\end{equation}
and it holds
\begin{equation}
\label{eq:GPhi_bound-2}
\sup_{x \in X} \left \|G_S (\Phi(x))G_\Phi(x)\right\|_{\LL(X,X)} \leq \gamma.
\end{equation}
\end{enumerate}
Then all of the results in \cref{subsec:global_contraction} apply to 
\eqref{eq:SPhi_FP} with $H :=S \circ \Phi.$ In particular, \cref{algo:semiNewtonAbstract} applied to \eqref{eq:SPhi_FP} satisfies the finite and global $q$-superlinear convergence result of \cref{th:convergence}.
\end{corollary}

\begin{corollary}[Global convergence of \cref{algo:semiNewtonAbstract} for \eqref{eq:SPhi_FP} with local contraction]
\label{cor:composite_local}
Suppose that the following assumptions are satisfied:
\begin{enumerate}[label=\roman*)]
\item\label{ass:standing_general_proj:i}
$X$, $Y$, $D$, $S$, and $\Phi$
satisfy the conditions 
in points 
\ref{ass:standing_general:i} to 
\ref{ass:standing_general:iii}
of \cref{cor:composite_classical};
\item\label{ass:standing_general_proj:ii} $X$ is additionally a Hilbert space;
\item\label{ass:standing_general_proj:iv}
there exist a nonempty closed convex set 
$B \subset X$ and a number 
$\gamma \in [0,1)$ satisfying
\begin{equation}
\label{eq:Phi_gamma_Lip_loc}
\| S(\Phi(x_1)) - S(\Phi(x_2))\|_X
\leq
\gamma \|x_1 - x_2\|_X
\quad
\forall x_1, x_2 \in B
\end{equation}
and 
\begin{equation}
\label{eq:GPhi_bound-2_loc}
\sup_{x \in B} \left \|G_S (\Phi(x))G_\Phi(x)\right\|_{\LL(X,X)} \leq \gamma;
\end{equation}
\item\label{ass:standing_general_proj:v} the metric projection $P_B\colon X \to B$
in $(X, \|\cdot\|_X)$ onto $B$ is Newton differentiable 
with Newton derivative 
$G_{P_B}\colon X \to \mathcal{L}(X,X)$
and it holds
$
\norm{G_{P_B}(x)}{\LL(X,X)} \leq 1
$
for all $x \in X$. 
\end{enumerate}
Then the results of \cref{sec:framework_fp_H_multiple_solns} apply to \eqref{eq:SPhi_FP}. In particular,
\cref{algo:semiNewtonAbstract},
 applied to the fixed-point problem  \eqref{eq:fp_problem_proj_H}
 with $H = S\circ \Phi$,  satisfies the
    convergence result in \cref{th:loc_conv}.
\end{corollary}

We conclude this section with a remark on 
inexact function evaluations in the context of \eqref{eq:SPhi_FP}.   

\begin{remark}\label{rem:inexactness} 
In practice, one typically cannot evaluate the composition
$H(x_i) = S(\Phi(x_i))$ (and, as a consequence, the residue $R(x_i) = x_i - S(\Phi(x_i))$) exactly, 
as required, e.g.,
in steps \ref{algo1:term_general} and \ref{algo1:x_N_general}
of \cref{algo:semiNewtonAbstractGeneral}. Instead, one only has access
to approximations $H_\epsilon  = S_\epsilon \circ \Phi_\epsilon\colon X \to X$, $\epsilon > 0$,
of the function $H = S\circ \Phi\colon X \to X$. If, for example, 
$S$ is the solution map of the obstacle problem 
\eqref{eq:upper_obst_prob_again} (as in the case of the concrete application  \eqref{eq:QVI}), then $S_\epsilon$ might be the solution map of a PDE-approximation of \eqref{eq:upper_obst_prob_again} obtained
via penalization; see \cite{HintKopacka, WachsmuthSchiela, WachsmuthKunisch}.
In the presence of such an inexact oracle $H_\epsilon$, 
one can still easily ensure the accuracy requirements in \cref{algo:semiNewtonAbstractGeneral,algo:semiNewtonAbstract}, provided 
the error $H - H_\epsilon$  is controllable by means of an \emph{a-priori} estimate. 
If we assume, for example, that  
$\|H(x_i) - H_\epsilon(x_i) \|_X \leq C \epsilon$  holds for $x_i$
with a known constant $C>0$, then the triangle inequality implies that
 the following holds
for the condition in \eqref{eq:xB_update}:
\begin{equation*}
\begin{aligned}
 &\|x_B - H_\epsilon(x_i) \|_X \leq  
 \tau_i\|x_i  - H_\epsilon(x_i)\|_X - (1+\tau_i)C\epsilon
 \\
 &\qquad \Rightarrow 
   \|x_B - H(x_i) \|_X - C\epsilon  \leq  
 \tau_i\|x_i  - H(x_i)\|_X+   \tau_i C \epsilon - (1+\tau_i)C\epsilon
   \\
   & \qquad \Rightarrow  \|x_B - H(x_i) \|_X \leq  \tau_i\|R(x_i)\|_X.
\end{aligned}
\end{equation*}
This shows that, by determining $x_B$ and $\epsilon$ with 
$\|x_B - H_\epsilon(x_i) \|_X \leq  
 \tau_i\|x_i  - H_\epsilon(x_i)\|_X - (1+\tau_i)C\epsilon$, one can calculate 
 a trial iterate $x_B$ satisfying  \eqref{eq:xB_update} without having 
 precise access to $H$ and $R$.
\end{remark}

\section{Elliptic obstacle-type QVIs}
\label{sec:obstacle-type-QVIs}

We are now in a position to prove the convergence of \cref{algo:semiNewtonAbstractGeneral,algo:semiNewtonAbstract} applied to the obstacle-type
QVIs \labelcref{eq:QVI}. Recall that
elliptic QVIs of obstacle type 
correspond to fixed-point 
equations of the form \eqref{eq:SPhi_FP} that involve as the map $S$
the 
solution operator 
$S\colon H_0^1(\Omega) \to H_0^1(\Omega)$,
$\phi \mapsto u$,
of the unilateral obstacle problem 
\begin{equation}
\label{eq:upper_obst_prob}  
\text{Find $u \in H_0^1(\Omega)$ such that  } 
u \in K(\phi),
~\langle -\Delta u - f, v - u \rangle_{H^{-1}(\Omega),H_0^1(\Omega)} \geq 0~
\forall v \in K(\phi),
\end{equation}
where
$
    K(\phi) := 
    \{
        v \in H_0^1(\Omega) \mid v 
        \leq \phi + \Phi_0 \text{ a.e.\ in }\Omega
    \}.$
In \cref{subsec:obstacle_solution_map} below, 
we show that the solution operator $S$ of \eqref{eq:upper_obst_prob}  
satisfies the Newton differentiability and Lipschitz continuity requirements 
of \cref{cor:composite_global,cor:composite_local,cor:composite_classical}.
Here, we also specify 
the properties that 
the obstacle function
$\Phi$ has to possess in the context 
of \eqref{eq:QVI}  so that 
\cref{algo:semiNewtonAbstractGeneral,algo:semiNewtonAbstract} can 
be applied; see \cref{th:obst_QVI}.
In \cref{subsec:solution_ops_as_obstacle}, we then 
establish that obstacle mappings $\Phi$ arising from 
semilinear elliptic PDEs satisfy the requirements on the inner function 
$\Phi$ in  \cref{cor:composite_global,cor:composite_local,cor:composite_classical}. Finally, in
\Cref{sec:application_nonlinearVI}, we demonstrate
that our results can also be employed in the context of 
nonlinear and implicit obstacle-type VIs. 

\subsection{The solution operator of the obstacle problem}
\label{subsec:obstacle_solution_map}

Throughout this subsection, we consider the following situation.

\begin{assumption}[QVI-assumptions]
\label{ass:standing_obstacle_map}
~
\begin{enumerate}[label=\roman*)]
\item $\Omega \subset \R^d$, $d \in \N$, is a nonempty open bounded set;
\item $p$ is a given exponent satisfying $ \max(1,2d/(d+2)) < p \leq \infty$;
\item $f \in L^p(\Omega)$ is a given function;
\item\label{ass:standing_obstacle_map:iv} $\Phi_0 \colon \Omega \to (-\infty, \infty]$ is a quasi-lower semicontinuous, 
Borel-measurable function
such that
there exists $w \in H_0^1(\Omega)$ satisfying $w \leq \Phi_0$ a.e.\ in $\Omega$ and such that, for every $v \in H_0^1(\Omega)$,
it holds $v \leq \Phi_0$ a.e.\ in $\Omega$ if and only 
if $v \leq \Phi_0$ quasi-everywhere (q.e.) in $\Omega$;
\item $Y_p$ is the space defined 
by $Y_p:= \{ v \in H_0^1(\Omega) \mid \Delta v \in L^p(\Omega)\}$ and equipped with the norm 
\begin{equation}
\label{eq:Y_norm_def}
	\|v\|_{Y_p} :=  \|v\|_{H_0^1(\Omega)} + \|\Delta v\|_{L^p(\Omega)}.
\end{equation}
\end{enumerate}
\end{assumption}

Note that \smash{$(Y_p, \|\cdot\|_{Y_p})$} is a Banach space in the above situation (as one may easily check). For the precise definitions of 
the terms ``quasi-lower semicontinuous''
and ``quasi-everywhere''
and details on the related notion of 
$H_0^1(\Omega)$-capacity, 
we refer to \cite[\S 6.4.3]{BonnansShapiro}.
We remark that \cref{ass:standing_obstacle_map}\ref{ass:standing_obstacle_map:iv}
is automatically satisfied if $\Omega$ is a 
bounded Lipschitz domain and 
$\Phi_0$ an element of $C(\bar \Omega) \cap H^1(\Omega)$ with a 
nonnegative trace. 

\begin{lemma}[Well-definedness of $S$]
\label{lem:S_obs_well_def}
Suppose that \cref{ass:standing_obstacle_map} holds. Then \eqref{eq:upper_obst_prob}  has a unique 
solution $S(\phi) := u$ for all $\phi \in H_0^1(\Omega)$.
The associated solution map $S\colon H_0^1(\Omega) \to H_0^1(\Omega)$,
$\phi \mapsto u$, satisfies
\begin{equation}
\label{eq:phi_Lipschitz_est}
	\| S(\phi_1) - S(\phi_2)\|_{H_0^1(\Omega)}
	\leq
	\| \phi_1 - \phi_2\|_{H_0^1(\Omega)}
	\qquad \forall \phi_1, \phi_2 \in H_0^1(\Omega).
\end{equation}
\end{lemma}
\begin{proof}
We know that $\Phi_0 + \phi \geq w + \phi \in H_0^1(\Omega)$ holds for all 
$\phi \in H_0^1(\Omega)$, where $w \in H_0^1(\Omega)$ is the function from 
\cref{ass:standing_obstacle_map}\ref{ass:standing_obstacle_map:iv}. 
Thus, 
$K(\phi) \neq \emptyset$ for all $\phi \in H_0^1(\Omega)$. 
The unique solvability of \eqref{eq:upper_obst_prob} for all $\phi \in H_0^1(\Omega)$
now follows 
immediately from \cite[Theorem 4:3.1]{Rodrigues}. 
Let us now assume that $\phi_1, \phi_2 \in H_0^1(\Omega)$ are given.
Define $u_j := S(\phi_j)$, $j=1,2$. Then 
it holds
$u_1 - \phi_1 + \phi_2 \leq \phi_2 + \Phi_0$
and
$u_2 - \phi_2 + \phi_1 \leq \phi_1 + \Phi_0$
a.e.\ in $\Omega$
and we obtain from the VIs satisfied by 
$u_1$ and $u_2$ that
\begin{equation*}
\begin{aligned}
0 &\leq 
\langle -\Delta u_1 - f, (u_2 - \phi_2 + \phi_1) - u_1 \rangle_{H^{-1}(\Omega),H_0^1(\Omega)} 
+
\langle -\Delta u_2 - f, (u_1 - \phi_1 + \phi_2) - u_2 \rangle_{H^{-1}(\Omega),H_0^1(\Omega)} 
\\
&
\leq
- \|u_1 - u_2\|_{H_0^1(\Omega)}^2
+
 \|u_1 - u_2\|_{H_0^1(\Omega)}  \|\phi_1 - \phi_2\|_{H_0^1(\Omega)}.
\end{aligned}
\end{equation*}
This yields \eqref{eq:phi_Lipschitz_est} and completes the proof. 
\end{proof} 

Note that, via a simple variable transformation, we
can shift the function $\phi$
in $K(\phi)$
into the source term of \eqref{eq:upper_obst_prob} 
and, thus, rewrite $S$
in terms of the solution map 
\mbox{$S_0\colon H^{-1}(\Omega) \to H_0^1(\Omega)$,}
$w \mapsto u$,
of the variational inequality 
\begin{align}
\label{eq:upper_obst_prob_rhs}  
\text{Find $u \in H_0^1(\Omega)$ such that  } 
u \in K(0),~~\langle -\Delta u - w, v - u \rangle_{H^{-1}(\Omega),H_0^1(\Omega)} \geq 0 \quad \forall v \in K(0).
\end{align}
Indeed, we have:

\begin{lemma}
\label{lem:shift_id}
Suppose that \cref{ass:standing_obstacle_map} holds. Then 
\begin{equation}
\label{eq:shift_id}
	S(\phi) = \phi +S_0(f + \Delta \phi)\qquad \forall \phi \in H_0^1(\Omega).
\end{equation}
\end{lemma}

\begin{proof}
Let $\phi \in H_0^1(\Omega)$ be given. Define $u := S(\phi)$. 
Then $\tilde u := u - \phi$ satisfies 
$\tilde u \in K(0)$
and, by \eqref{eq:upper_obst_prob}, 
\[
	\langle -\Delta \tilde u - f - \Delta \phi, v - \tilde u \rangle_{H^{-1}(\Omega),H_0^1(\Omega)}
	=
	\langle -\Delta u - f , v + \phi - u \rangle_{H^{-1}(\Omega),H_0^1(\Omega)} 
	\geq 0
	 \quad \forall v \in K(0).
\]
As \eqref{eq:upper_obst_prob_rhs}   
is uniquely solvable by \cite[Theorem 4:3.1]{Rodrigues},
this yields $S_0(f + \Delta \phi) = \tilde u = S(\phi) - \phi$.
\end{proof}

Due to \eqref{eq:shift_id}, the function $S$ inherits  various important properties from 
$S_0$. To formulate these properties, we require some notation.
First, we define the \emph{active set} associated with
\eqref{eq:upper_obst_prob} via
\begin{equation}
\label{eq:active_set_def}
A(\phi) := \{
x\in \Omega 
\mid
S(\phi)(x) = \phi(x) + \Phi_0(x) \}.
\end{equation}
(This set is also known as the \emph{coincidence set}.)
Following the lines of 
\cite[\S 1]{MR1163430} and \cite[\S 2.2]{RaulsWachsmuth},
we interpret the identity \eqref{eq:active_set_def} in the sense of 
capacity theory, i.e., we define $A(\phi)$
up to polar sets
and w.r.t.\ quasi-(lower semi)continuous representatives.
Note that, due to the quasi-lower semicontinuity
of $\Phi_0$ and 
the inequality $S(\phi) \leq \phi + \Phi_0$, this implies that 
$A(\phi)$ is a quasi-closed set. Its complement,
the quasi-open set $\Omega \setminus A(\phi)$,
is called the \emph{inactive set} 
and denoted by $I(\phi)$ in the following.
The quasi-openness of $I(\phi)$
makes it possible to sensibly define 
the space $H_0^1(I(\phi)) \subset H_0^1(\Omega)$
of all elements of $H_0^1(\Omega)$ that vanish 
in $A(\phi)$; see again \cite[\S 2.1]{RaulsWachsmuth} or \cite[\S 2]{MR1163430}. If 
$I(\phi)$ is open in the classical sense, then $H_0^1(I(\phi))$ 
coincides with the closure of $C_c^\infty(I(\phi))$
in $H_0^1(\Omega)$; see 
\cite[Theorem 9.1.3]{AdamsHedberg1996}.
By means of $I(\phi)$, we can introduce:

\begin{definition}
\label{def:GS}
We denote by 
 $Z_S\colon H_0^1(\Omega) \to \LL(H^{-1}(\Omega), H_0^1(\Omega))$,
 $\phi \mapsto Z_S(\phi)$, 
 the map defined by 
\begin{equation*}
Z_S(\phi)g  = z \qquad \text{ if and only if}
\qquad z \in H^1_0(I(\phi)),~~\langle 
-\Delta z - g, v \rangle_{H^{-1}(\Omega),H_0^1(\Omega)} = 0~\forall v \in H_0^1(I(\phi))
\end{equation*}
for given $g \in H^{-1}(\Omega)$  and $\phi \in H_0^1(\Omega)$.
We further define 
$\tilde{G}_S\colon H_0^1(\Omega) \to \LL(H_0^1(\Omega), H_0^1(\Omega))$
via
\begin{equation}\label{eq:def_GS}
\tilde{G}_S(\phi)h := h + Z_S(\phi)\Delta h \qquad \forall h, \phi \in H_0^1(\Omega).  
\end{equation}
\end{definition}
 
Informally speaking and modulo an extension 
by zero to the whole of $\Omega$, the operator 
$Z_S(\phi)$ defined above can be interpreted as the solution map of the 
Poisson problem with homogeneous Dirichlet 
boundary conditions on the inactive set 
$I(\phi)$, i.e., $z=Z_S(\phi)g$ can be characterized as  the solution of the
boundary value problem 
\begin{equation*}
        -\Delta z = g\text{ in } 
        I(\phi),
        \qquad 
        z = 0 \text{ on }\partial(I(\phi)).
\end{equation*}
As the notation suggests, the function $\tilde{G}_S$
provides a Newton derivative $G_S$
for $S$. Indeed, we have:

\begin{lemma}[Newton differentiability of $S$]
\label{lem:S_obs_Newton}
Suppose that \cref{ass:standing_obstacle_map} holds. Then the operator $S$
is Newton differentiable as a function from 
$Y_p$ to $H_0^1(\Omega)$ 
when endowed with the derivative $G_S :=\tilde{G}_S$,
$G_S\colon Y_p \to \LL(H_0^1(\Omega), H_0^1(\Omega))$. 
\end{lemma} 

\begin{proof}
Let $\phi \in Y_p$ be given. Then it holds 
\begin{equation}
\label{eq:randomeq3636}
\begin{aligned}
0&\leq \limsup_{0 <  \|h\|_{Y_p} \to 0}
\frac{\|S(\phi + h) - S(\phi) - G_S(\phi+h)h\|_{H_0^1(\Omega)}}{\|h\|_{Y_p}}
\\
&=
\limsup_{0 <  \|h\|_{Y_p} \to 0}
\frac{\| S_0(f + \Delta \phi + \Delta h) 
  - S_0(f + \Delta \phi)
  - Z_S(\phi + h)\Delta h
  \|_{H_0^1(\Omega)}}{\|h\|_{Y_p}}
\\
&\leq
\limsup_{0 <  \|\Delta h\|_{L^p(\Omega)} \to 0, h \in Y_p}
\frac{\| S_0(f + \Delta \phi + \Delta h) 
  - S_0(f + \Delta \phi)
  - Z_S(\phi + h)\Delta h
  \|_{H_0^1(\Omega)}}{\|\Delta h\|_{L^p(\Omega)}}
  \\
  &\leq
 \limsup_{0 <  \|g\|_{L^p(\Omega)} \to 0}
 \sup_{H \in \partial_B^{ss}S_0(f + \Delta \phi + g )}
\frac{\| S_0(f + \Delta \phi + g) 
  - S_0(f + \Delta \phi)
  - H g
  \|_{H_0^1(\Omega)}}{\|g\|_{L^p(\Omega)}}.
\end{aligned}
\end{equation}
Here, we have used \cref{lem:shift_id},
the definition of $\|\cdot\|_{Y_p}$,
and 
the fact that 
$Z_S(\phi + h)$
is an element of the so-called
\emph{strong-strong Bouligand differential}
$\partial_B^{ss}S_0(f + \Delta \phi + \Delta h )$
of $S_0$ at $f + \Delta \phi + \Delta h \in L^p(\Omega)$
by \cite[Theorem 4.3]{RaulsWachsmuth} 
(and a trivial direct argument in the case $d=1$). 
See \cite[Definition 2.10]{RaulsWachsmuth} for the 
precise definition of $\partial_B^{ss}S_0$.
From the Newton differentiability properties of $S_0$ proven in 
\cite[Theorem 4.4]{ChristofWachsmuthBiObstacle}
and the inclusions in 
\cite[Proposition 2.11]{RaulsWachsmuth},
we obtain that the right-hand side of \eqref{eq:randomeq3636}
is equal to zero. 
This completes the proof.
\end{proof}

For convenience, we drop the tilde in $\tilde{G}_S$ everywhere in the following.
From \eqref{eq:def_GS}, we obtain: 

\begin{lemma}[Boundedness of $G_S$]\label{lem_GS_bound_obstacle}
Suppose that \cref{ass:standing_obstacle_map} holds. Then, for every $\phi \in H_0^1(\Omega)$, it holds 
$\|G_S(\phi)\|_{\LL(H_0^1(\Omega), H_0^1(\Omega))} \leq 1$.
\end{lemma}
\begin{proof}
Let $\phi \in H_0^1(\Omega)$ be given. 
Due to its definition,  $-Z_S(\phi)\Delta\colon H_0^1(\Omega) \to H_0^1(\Omega)$ 
is precisely the 
$H_0^1(\Omega)$-orthogonal projection onto 
$H_0^1(I(\phi))$. 
This implies that $G_S(\phi)h = h + Z_S(\phi)\Delta h$ is 
the $H_0^1(\Omega)$-orthogonal projection onto 
$H_0^1(I(\phi))^\perp$.
The assertion now follows immediately from the fact that projections in Hilbert spaces
are non-expansive. 
\end{proof}

If we combine \cref{lem:S_obs_Newton,lem:S_obs_well_def,lem_GS_bound_obstacle}
and compare 
with the requirements 
of the convergence results in 
 \cref{cor:composite_classical,cor:composite_global,cor:composite_local},
then we arrive at the following main theorem.

\begin{theorem}[Convergence of \cref{algo:semiNewtonAbstractGeneral,algo:semiNewtonAbstract} for \labelcref{eq:QVI}]
\label{th:obst_QVI}
Suppose that \cref{ass:standing_obstacle_map} holds and $\Phi\colon H_0^1(\Omega) \to Y_p$,
$\gamma \in [0,1)$, and $B \subset H_0^1(\Omega)$ 
are given such that 
the following is true:
\begin{enumerate}[label=\roman*)]
\item\label{th:obst_QVI:i} $\Phi$ is 
Newton differentiable from $H_0^1(\Omega)$ to $Y_p$ with derivative 
$G_\Phi \colon H_0^1(\Omega) \to \LL(H_0^1(\Omega),Y_p)$;
\item\label{th:obst_QVI:ii} $\Phi$ is locally Lipschitz continuous from $H_0^1(\Omega)$ to $Y_p$;
\item $B$ is a closed ball of radius $r>0$ in $H_0^1(\Omega)$ 
(not necessarily centered at zero)
or $B = H_0^1(\Omega)$;
\item\label{th:obst_QVI:iii} $\Phi$ 
satisfies
$\|\Phi(v_1) - \Phi(v_2)\|_{H_0^1(\Omega)} \leq \gamma \|v_1 - v_2\|_{H_0^1(\Omega)}$
 for all $v_1, v_2 \in B$;
\item\label{th:obst_QVI:iv} $G_\Phi$ satisfies $\|G_\Phi(v)\|_{\LL(H_0^1(\Omega), H_0^1(\Omega))} \leq \gamma$
for all $v \in B$.
\end{enumerate}
Then the QVI 
\begin{equation}
\tag{Q}
\begin{gathered}
\smash{\text{Find $u \in H_0^1(\Omega)$ such that  } 
u \in K(\Phi(u)),
~\langle -\Delta u - f, v - u \rangle_{H^{-1}(\Omega),H_0^1(\Omega)} \geq 0~
\forall v \in K(\Phi(u)),}
\\
\smash{\text{ with }
K(\Phi(u)) := 
\{
        v \in H_0^1(\Omega) \mid v 
        \leq \Phi(u) + \Phi_0 \text{ a.e.\ in }\Omega
 \},} 
\end{gathered}
\end{equation} 
is equivalent to the fixed-point equation $u = S(\Phi(u))$ and the following is true:
\begin{enumerate}[label=\Roman*)]
\item\label{th:obst_QVI:I} If $B = H_0^1(\Omega)$ holds, then the assumptions of 
\cref{cor:composite_classical,cor:composite_global} are satisfied
with $X = H_0^1(\Omega)$,
$Y = D = Y_p$,
and $G_S$ as in \eqref{eq:def_GS}.  
In particular,  the results of \cref{subsec:global_contraction}
apply to \eqref{eq:QVI},
 \eqref{eq:QVI} possesses a unique solution 
 $\bar u \in H_0^1(\Omega)$,
and $\bar u$ can be identified by 
means of \cref{algo:semiNewtonAbstractGeneral,algo:semiNewtonAbstract}, 
with the convergence guarantees in \cref{th:convergenceGeneral,th:convergence}, respectively.
\item If $B$ has finite radius, then the assumptions of 
\cref{cor:composite_local} are satisfied
(with the same $X$, $Y$, etc.\ as in \ref{th:obst_QVI:I}),
and the results of \cref{sec:framework_fp_H_multiple_solns} 
apply to \eqref{eq:QVI}. 
In particular,
\cref{algo:semiNewtonAbstract},
 applied to the fixed-point problem  $u = S(\Phi(P_B(u)))$
satisfies the convergence result in \cref{th:loc_conv}.
\end{enumerate}
\end{theorem}

\begin{proof}
The assertions of the theorem follow
immediately from 
\cref{lem:semismooth_balls,lem:S_obs_well_def,lem_GS_bound_obstacle,lem:S_obs_Newton};
the assumptions on $\Phi$;
and the fact that $\|v\|_{H_0^1(\Omega)} \leq \|v\|_{Y_p}$ holds
for all $v \in Y_p$.
\end{proof}

In the context of the QVI
 \eqref{eq:QVI}, the Newton derivative $G_R$ appearing
 in steps \ref{algo1:x_N_general} and \ref{algo1:x_N} of \cref{algo:semiNewtonAbstractGeneral,algo:semiNewtonAbstract}, respectively, is given by
\[G_R(u)h = (\Id- G_\Phi(u))h - Z_S(\Phi(u))\Delta G_\Phi(u)h
\qquad \forall u, h \in H_0^1(\Omega)
.\]
Details on how this and related objects can be realized numerically are given in \cref{subsec:7.1}. 

Note that \cref{th:obst_QVI} leaves considerable 
freedom regarding the precise form of $\Phi$
and also covers 
cases in which the pointwise-a.e.\ constraint 
in \eqref{eq:QVI} is only imposed 
in certain parts of $\Omega$ as 
\cref{ass:standing_obstacle_map}\ref{ass:standing_obstacle_map:iv} allows to
consider functions $\Phi_0$
that take the value $+\infty$.
In what follows, 
we focus primarily on the case that 
$\Phi$
is the solution map of a (potentially nonsmooth)
PDE.

\subsection{Solution operators of semilinear PDEs as obstacle maps}
\label{subsec:solution_ops_as_obstacle}
Next, we show
that solution operators of certain
semilinear PDEs give rise to
obstacle maps $\Phi$ that satisfy 
the assumptions of \cref{th:obst_QVI} and, thus, 
yield obstacle-type QVIs that are covered 
by the general semismooth Newton framework of \cref{sec:framework_fp_H}. 
The operators $\Phi$ that we focus on in 
this subsection are given by 
\begin{equation}
\label{eq:thermo_obst_mapprob}
\Phi(u) := \varphi T \text{ with } T \text{ as the solution of }\quad
kT-\Delta T   = g(\Psi_0 + \psi T- u) \text{ in $\Omega$},\quad
\partial_\nu T = 0 \text{ on $\partial \Omega$}.
\end{equation}
 
\begin{assumption}\label{ass:SemilinearObstacleMap}
~
\begin{enumerate}[label=\roman*)]
\item $\Omega \subset \R^d$, $d \in \N$, is a bounded Lipschitz domain;
\item $\varphi \in C^2(\bar\Omega)$ is a given function satisfying 
$\varphi = 0$ on $\partial \Omega$;
\item $k>0$ is a given constant;
\item $\Psi_0$ is a given function satisfying 
$\Psi_0 \in L^{2+\varepsilon}(\Omega)$
for some $\varepsilon > 0$;
\item $\psi \in L^\infty(\Omega)$ is a given 
function satisfying $\psi \geq 0$ a.e.\ in $\Omega$;
\item\label{ass:SemilinearObstacleMap:vi} $g\colon \R \to \R$ is
locally Lipschitz continuous, 
nonincreasing, and 
Newton differentiable with a 
derivative $G_g\colon \R \to \R$
that is Borel-measurable 
and satisfies $G_g(s) \in \partial_c g(s)$ for all $s \in \R$;
\item $Y_2$ is defined by $Y_2 := \{ v \in H_0^1(\Omega) \mid \Delta v \in L^2(\Omega)\}$ and equipped with the 
norm 
$\|\cdot\|_{Y_2}$
from \eqref{eq:Y_norm_def}.  
\end{enumerate}
\end{assumption} 

Note that, as the choice $g(s) = -s$
satisfies \cref{ass:SemilinearObstacleMap}\ref{ass:SemilinearObstacleMap:vi},
\eqref{eq:thermo_obst_mapprob} also 
covers linear elliptic PDEs with homogeneous
Neumann boundary conditions. We remark that 
the arguments that we use 
in the following can be extended easily
to, e.g., 
more general differential operators and 
Dirichlet boundary conditions;
cf.\ the general assumptions in \cref{th:obst_QVI}.
We focus on \eqref{eq:thermo_obst_mapprob} 
because this setting 
covers QVIs such as thermoforming problems \cite[\S 6]{AHR}.
We first consider 
\eqref{eq:thermo_obst_mapprob} for arbitrary $d \in \N$ in the next result. (We later consider $d=1$ as a special case.)

\begin{theorem}[Properties of \eqref{eq:thermo_obst_mapprob}]
\label{th_T_probs}
Assume, in addition to 
\cref{ass:SemilinearObstacleMap}, 
that the function $g\colon \R \to \R$
is globally Lipschitz continuous. 
Then the PDE in \eqref{eq:thermo_obst_mapprob}, i.e.,
\begin{equation}
\label{eq:T_PDE}
kT-\Delta T  = g(\Psi_0 + \psi T- u) \text{ in $\Omega$}, 
\qquad
\partial_\nu T  = 0 \text{ on $\partial \Omega$},
\end{equation}
possesses a unique (weak) solution $T \in H^1(\Omega)$ for all $u \in H_0^1(\Omega)$. 
This solution satisfies $\varphi T \in Y_2$. Furthermore, 
the operator $\Phi\colon H_0^1(\Omega) \to Y_2$, $u \mapsto \varphi T$,
possesses the following properties: 
\begin{enumerate}[label=\roman*)]
\item\label{th_T_probs:i} It holds 
\begin{equation}
\label{eq:Phi_Lip_H1}
\begin{aligned}
	&\|\Phi(u_1) - \Phi(u_2)\|_{H_0^1(\Omega)}
        \\	
 &\quad\leq
	C_P(\Omega)\,\Lip{g} 
\left ( \|\varphi \|_{L^\infty(\Omega)}  k^{-1/2} 
+
\| |\nabla \varphi| \|_{L^\infty(\Omega)} k^{-1}
\right )\|u_1 - u_2\|_{H_0^1(\Omega)}
\quad \forall u_1, u_2 \in H_0^1(\Omega).
\end{aligned}
\end{equation}
\item\label{th_T_probs:ii}  There exists a constant $C>0$ satisfying 
\begin{equation}
\label{eq:Phi_Y_est}
	\|\Phi(u_1) - \Phi(u_2)\|_{Y_2}
	\leq
	C
	\| u_1  -  u_2\|_{H_0^1(\Omega)}\qquad \forall u_1, u_2 \in H_0^1(\Omega). 
\end{equation}
\item\label{th_T_probs:iii}  
Let $G_\Phi\colon H_0^1(\Omega) \to \LL(H_0^1(\Omega), Y_2)$
be defined by  $G_\Phi(u)h = \varphi \xi_h$ with $\xi_h$ as the weak solution of 
	\begin{equation}
	\label{eq:xi_PDE}
k\xi_h -\Delta \xi_h  - G_g(\Psi_0 +  \psi T - u)\psi \xi_h 
=
- G_g(\Psi_0 + \psi T - u) h\text{ in $\Omega$}, 
\qquad
\partial_\nu \xi_h  = 0 \text{ on $\partial \Omega$},
\end{equation}
and $T$ as the solution of \eqref{eq:T_PDE}.
Then $\Phi$ is Newton differentiable as a function from $H_0^1(\Omega)$ to $Y_2$
with derivative $G_\Phi$ and, 
for all $u \in H_0^1(\Omega)$, 
it holds 
\begin{equation}
	\label{eq:G_phi_bound}
 \begin{aligned}
&\|G_\Phi(u)\|_{\LL(H_0^1(\Omega), H_0^1(\Omega))}
\leq
C_P(\Omega)\,
\Lip{g} 
\left ( \|\varphi \|_{L^\infty(\Omega)}  k^{-1/2} 
+
\| |\nabla \varphi| \|_{L^\infty(\Omega)} k^{-1}
\right ).
\end{aligned}
\end{equation}
\end{enumerate}
\end{theorem}
\begin{proof}
From Young's inequality and 
our assumptions on $g$, $\Psi_0$, $\psi$,
and $\Omega$, 
it follows that 
\begin{equation}
\label{eq:L2gestimate}
\int_\Omega |g(\Psi_0 + \psi v - u)|^2 \mathrm{d}x
\leq
2
\int_\Omega 
|g(0)|^2
+ 
 \Lip{g}^2|\Psi_0 + \psi v - u|^2
 \mathrm{d}x
 <
 \infty
 \quad 
 \forall v \in H^1(\Omega)
 \quad \forall u \in H_0^1(\Omega).
\end{equation}
Thus, $g(\Psi_0 + \psi v - u) \in L^2(\Omega)$
for all $v \in H^1(\Omega)$ and $u\in H_0^1(\Omega)$,
and we obtain that 
the operator 
\[
	A_u\colon H^1(\Omega) \to H^1(\Omega)^*,
	\qquad
	\left \langle A_u(v), w \right \rangle_{H^1(\Omega)^*,H^1(\Omega)}
	:=
	\int_{\Omega} kvw + \nabla v \cdot \nabla w - 
	 g(\Psi_0 + \psi v - u)  w  \mathrm{d}x,
\]
is well defined for all $u \in H_0^1(\Omega)$. Due to the 
global Lipschitz continuity and 
monotonicity of $g$ and the 
nonnegativity and essential boundedness of $\psi$,
we further have
\begin{equation*}
\begin{aligned}
&\left \| A_u(v_1) - A_u(v_2) \right \|_{H^1(\Omega)^*}
\\
&\quad\leq
\max(k,1)
\|v_1 - v_2\|_{H^1(\Omega)}
+
\|g(\Psi_0 + \psi v_1 - u) - g(\Psi_0 + \psi v_2 - u)\|_{L^2(\Omega)}
\\
&\quad\leq 
\left (\max(k,1) + \Lip{g} \|\psi\|_{L^\infty(\Omega)} \right )
\|v_1 - v_2\|_{H^1(\Omega)}
\quad \forall v_1, v_2 \in H^1(\Omega)\quad \forall u \in H_0^1(\Omega)
\end{aligned}
\end{equation*}
and 
\begin{equation}
\label{eq:Aucoercive}
\begin{aligned}
&\left \langle A_u(v_1) - A_u(v_2) , v_1 - v_2\right \rangle_{H^1(\Omega)^*,H^1(\Omega)}
\\
&\quad
\geq\int_{\Omega} k(v_1 - v_2)^2 + | \nabla (v_1 - v_2) |^2
- \left (g(\Psi_0 + \psi v_1 - u) - g(\Psi_0 + \psi v_2 - u) \right )(v_1 - v_2)  \mathrm{d}x
\\
&\quad \geq k \|v_1 - v_2\|_{L^2(\Omega)}^2
+
\||\nabla(v_1 - v_2)|\|_{L^2(\Omega)}^2
\quad \forall v_1, v_2 \in H^1(\Omega)\quad \forall u \in H_0^1(\Omega). 
\end{aligned}
\end{equation}
This shows that $A_u\colon H^1(\Omega) \to H^1(\Omega)^*$ 
is globally Lipschitz continuous 
and coercive and, 
by 
\cite[Theorem 4:3.1]{Rodrigues},
that the equation $A_u(T) = 0$
is uniquely solvable for all 
$u \in H_0^1(\Omega)$. 
Due to the definition of $A_u$,
the latter shows that 
 \eqref{eq:T_PDE}
 possesses a unique solution $T \in H^1(\Omega)$
 for all $u \in H_0^1(\Omega)$. 
That this solution 
satisfies $\varphi T \in H_0^1(\Omega)$ and 
$\Delta (\varphi T) \in L^2(\Omega)$ 
for all $u \in H_0^1(\Omega)$
follows immediately from the properties of 
$\varphi$, the definition of $\Delta$, \eqref{eq:T_PDE}, and \eqref{eq:L2gestimate}.
Thus, $\varphi T \in Y_2$ as claimed. 

It remains to prove points \ref{th_T_probs:i}, \ref{th_T_probs:ii}, and \ref{th_T_probs:iii}.
To this end, let us assume that $u_1, u_2 \in H_0^1(\Omega)$ are given 
and that $T_1, T_2 \in H^1(\Omega)$ are the associated solutions 
of \eqref{eq:T_PDE}. Then \eqref{eq:Aucoercive} yields
\begin{equation*}
\begin{aligned}
 k \|T_1 - T_2\|_{L^2(\Omega)}^2
+ \| |\nabla T_1 - \nabla T_2 |\|_{L^2(\Omega)}^2
&\leq
\left \langle A_{u_1}(T_1) - A_{u_1}(T_2), T_1- T_2\right \rangle_{H^1(\Omega)^*,H^1(\Omega)}
\\
&=
\left (
g(\Psi_0 + \psi T_2 - u_1)
-
g(\Psi_0 + \psi T_2 - u_2)
,
T_1 - T_2
\right )_{L^2(\Omega)}
\\
&\leq
\Lip{g}
\|u_1 - u_2\|_{L^2(\Omega)}
\|T_1 - T_2\|_{L^2(\Omega)}.
\end{aligned}
\end{equation*}
This implies
\begin{equation}
\label{eq:estimates:T1T2-L2}
\|T_1 - T_2\|_{L^2(\Omega)}
\leq
 \Lip{g}k^{-1}
\|u_1 - u_2\|_{L^2(\Omega)}
\end{equation}
and
\begin{equation}
\label{eq:estimates:T1T2-H1}
\| |\nabla T_1 - \nabla T_2 |\|_{L^2(\Omega)}
\leq \Lip{g} k^{-1/2} \|u_1 - u_2\|_{L^2(\Omega)}.
\end{equation}
In combination with the definition of $\Phi$, 
we now obtain 
\begin{equation*}
\begin{aligned}
\|\Phi(u_1) - \Phi(u_2)\|_{H_0^1(\Omega)}
&\leq
\|\varphi \|_{L^\infty(\Omega)} \||\nabla T_1 - \nabla  T_2 |\|_{L^2(\Omega)}
+
\| |\nabla \varphi| \|_{L^\infty(\Omega)} \|  T_1 -  T_2 \|_{L^2(\Omega)}
\\
&\leq
\Lip{g}
\left ( \|\varphi \|_{L^\infty(\Omega)}  k^{-1/2} 
+
\| |\nabla \varphi| \|_{L^\infty(\Omega)} k^{-1}
\right )\|u_1 - u_2\|_{L^2(\Omega)}
\\
&\leq
C_P(\Omega)\,\Lip{g} 
\left ( \|\varphi \|_{L^\infty(\Omega)}  k^{-1/2} 
+
\| |\nabla \varphi| \|_{L^\infty(\Omega)} k^{-1}
\right )\|u_1 - u_2\|_{H_0^1(\Omega)}.
\end{aligned}
\end{equation*}
This proves \ref{th_T_probs:i}. To prove \ref{th_T_probs:ii}, we note that 
\eqref{eq:estimates:T1T2-L2}, 
the PDEs satisfied by $T_1$ and $T_2$,
and the global Lipschitz continuity of $g$ imply that 
\begin{equation*}
\begin{aligned}
\|\Delta T_1 - \Delta T_2\|_{L^2(\Omega)}
&=
\|kT_1 - kT_2 + g(\Psi_0 + \psi T_2- u_2) - g(\Psi_0 + \psi T_1- u_1)\|_{L^2(\Omega)}
\\
&\leq
k \|T_1 - T_2\|_{L^2(\Omega)}
+
\Lip{g}\left (\|\psi\|_{L^\infty(\Omega)}  \|T_1 - T_2\|_{L^2(\Omega)} +
 \|u_1 - u_2\|_{L^2(\Omega)}\right )
 \\
 &\leq
 C \|u_1 - u_2\|_{H_0^1(\Omega)}
\end{aligned}
\end{equation*}
holds 
with some constant $C>0$. In combination with the assumptions 
on $\varphi$, \eqref{eq:estimates:T1T2-L2},
and \eqref{eq:estimates:T1T2-H1}, this yields \eqref{eq:Phi_Y_est}
(with a potentially larger constant $C>0$). It remains to prove \ref{th_T_probs:iii}.
To this end, let us denote the solution map of
the PDE \eqref{eq:T_PDE} by $P\colon H_0^1(\Omega) \to H^1(\Omega)$, $u \mapsto T$.
We claim that 
\begin{equation}
\label{eq:P_semismooth_H1}
\lim_{0 <  \|h\|_{H_0^1(\Omega)} \to 0}
\frac{ 
\|P(u + h) - P(u) - G_P(u+h)h\|_{H^1(\Omega)}}{\|h\|_{H_0^1(\Omega)}}
=
0\qquad \forall u \in H_0^1(\Omega)
\end{equation}
and
\begin{equation}
\label{eq:P_semismooth_Delta}
\lim_{0 <  \|h\|_{H_0^1(\Omega)} \to 0}
\frac{\|\Delta P(u + h) - \Delta P(u) - \Delta G_P(u+h)h\|_{L^2(\Omega)} }{\|h\|_{H_0^1(\Omega)}}
=
0\qquad \forall u \in H_0^1(\Omega)
\end{equation}
hold, 
where $G_P\colon H_0^1(\Omega) \to \LL(H_0^1(\Omega), H^1(\Omega))$
is defined by $G_P(u)h := \xi_h$ 
for all $u, h \in H_0^1(\Omega)$
with $\xi_h$ being the solution of \eqref{eq:xi_PDE}.
To establish \eqref{eq:P_semismooth_H1} 
and \eqref{eq:P_semismooth_Delta}, 
let us assume that $u \in H_0^1(\Omega)$
is fixed. 
From 
 the PDEs solved by 
 $P(u +  h)$, $P(u)$,
 and $G_P(u+h )h$, we obtain that the function
 $
 \delta_h :=  P(u + h) -   P(u) - G_P(u + h )h \in H^1(\Omega)
 $
 satisfies 
 \begin{equation}
 \label{eq:deltah_PDE}
 \begin{aligned}
& k \delta_h - \Delta \delta_h
 - G_g(\Psi_0 + \psi P(u + h)- u - h)\psi \delta_h
 \\
 &\quad= 
  g(\Psi_0 + \psi P(u + h)- u - h)
 - 
 g(\Psi_0 + \psi P(u)- u)   
 \\
 &\qquad\qquad
  -
  G_g(\Psi_0 + \psi P(u + h)- u - h)\left ( \psi P(u + h) -   \psi P(u)  - h\right )
  = : \rho_h
 \end{aligned}
 \end{equation}
 in $\Omega$ 
 and $\partial_\nu \delta_h = 0$ on $\partial \Omega$ for all $h \in H_0^1(\Omega)$.
 Due to 
 \eqref{eq:estimates:T1T2-L2}
 and \eqref{eq:estimates:T1T2-H1}
 and 
 our assumptions on 
 $\psi$, $\Psi_0$, and $\Omega$,
 we further know 
 that there exist constants 
 $C, \varepsilon > 0$ such that 
 $\Psi_0 \in L^{2+\varepsilon}(\Omega)$
 holds, such that $H^1(\Omega)$ embeds 
 continuously into $L^{2+\varepsilon}(\Omega)$,
 and such that, for all $h \in H_0^1(\Omega)$,
 we have
 \[
\|\psi P(u + h) -   \psi P(u)  - h\|_{L^{2+\varepsilon}(\Omega)}
\leq
\|\psi\|_{L^\infty(\Omega)}
\| P(u + h) - P(u) \|_{L^{2+\varepsilon}(\Omega)}
+
\|h\|_{L^{2+\varepsilon}(\Omega)}
\leq
C \|h\|_{H_0^1(\Omega)}.
 \]
 In combination with the properties of $g$
 and \cite[Theorem 3.49]{Ulbrich2011}, 
 the above entails that
\begin{equation*}
\begin{aligned}
\limsup_{0 <  \|h\|_{H_0^1(\Omega)} \to 0}
\frac{ \left \| \rho_h
\right \|_{L^2(\Omega)}}{\|h\|_{H_0^1(\Omega)}}
&\leq
C
\limsup_{0 <  \|\tilde h\|_{L^{2+\varepsilon}(\Omega)} \to 0}
\frac{ \left \| 
g(\tilde u + \tilde h)
 - 
 g(\tilde u)   
  -
  G_g(\tilde u + \tilde h)\tilde h
\right \|_{L^2(\Omega)}}{\|\tilde h\|_{L^{2+\varepsilon}(\Omega)}}
= 0
\end{aligned}
\end{equation*}
holds, 
where 
$\tilde u$ is defined by $\tilde u := \Psi_0 + \psi P(u)- u$
and 
$\tilde h$ has been used to replace the perturbation  
$\psi P(u + h) -   \psi P(u)  - h$
appearing in the definition of $\rho_h$.
By choosing $\delta_h$ as the test function in the weak form of 
\eqref{eq:deltah_PDE}---keeping in mind that  
$G_g(s) \in \partial_c g(s) \subset [-\Lip{g}, 0]$
holds for all $s \in \R$ by 
\cref{ass:SemilinearObstacleMap}\ref{ass:SemilinearObstacleMap:vi}
and that $\psi$ is 
nonnegative---we now obtain that 
\begin{equation*}
\frac{k \|\delta_h\|_{L^2(\Omega)}^2 + \| |\nabla \delta_h|\|_{L^2(\Omega)}^2}{\|h\|_{H_0^1(\Omega)}}
\leq
\frac{\|\rho_h\|_{L^2(\Omega)}\|\delta_h\|_{L^2(\Omega)} }{\|h\|_{H_0^1(\Omega)}}
\qquad \forall h \in H_0^1(\Omega)\setminus \{0\}
\end{equation*}
and, after applying Young's inequality, that there exists a constant $C>0$ satisfying
\begin{equation}
\label{eq:random_eq_41_H^1}
 \limsup_{0 <  \|h\|_{H_0^1(\Omega)} \to 0}
\frac{\| \delta_h\|_{H^1(\Omega)}}{\|h\|_{H_0^1(\Omega)}} 
\leq
C
\limsup_{0 <  \|h\|_{H_0^1(\Omega)} \to 0}
\frac{ \left \| \rho_h
\right \|_{L^2(\Omega)}}{\|h\|_{H_0^1(\Omega)}}
=
0
.
\end{equation}
By revisiting \eqref{eq:deltah_PDE} and by exploiting 
that $|G_g|$ is bounded by $\Lip{g}$, it now also follows that 
\begin{equation}
\label{eq:random_eq_41_Delta}
\limsup_{0 <  \|h\|_{H_0^1(\Omega)} \to 0}
\frac{\| \Delta \delta_h\|_{L^2(\Omega)}}{\|h\|_{H_0^1(\Omega)}} 
\leq
\limsup_{0 <  \|h\|_{H_0^1(\Omega)} \to 0}
\frac{\left (k + \Lip{g}\|\psi\|_{L^\infty(\Omega)}\right ) \| \delta_h\|_{L^2(\Omega)} + \|\rho_h\|_{L^2(\Omega)}}{\|h\|_{H_0^1(\Omega)}} 
=
0.
\end{equation}
Due to the definition of $\delta_h$, 
the estimates 
\eqref{eq:random_eq_41_H^1}
and \eqref{eq:random_eq_41_Delta} establish
\eqref{eq:P_semismooth_H1}
and 
\eqref{eq:P_semismooth_Delta}.
Since $\Phi(u) = \varphi P(u)$ holds 
with $\varphi$
satisfying 
$\varphi \in C^2(\bar \Omega)$
and
$\varphi = 0$ on $\partial \Omega$,
this yields 
that $\Phi$ is Newton differentiable as a function from 
$H_0^1(\Omega)$ to $Y_2$ with derivative $G_\Phi(u)h := \varphi G_P(u)h$ as claimed.
It remains to prove the bound in \eqref{eq:G_phi_bound}. To this end, 
we note that, for all $u, h \in H_0^1(\Omega)$,
we can choose  $\xi_h = G_P(u) h$ as the test function in the weak 
form of \eqref{eq:xi_PDE} to obtain 
$
	k \|\xi_h\|_{L^2(\Omega)}^2
	+
	\||\nabla \xi_h |\|_{L^2(\Omega)}^2
	\leq
	\Lip{g} \|\xi_h\|_{L^2(\Omega)} \|h\|_{L^2(\Omega)}.
$
Here, we have again exploited that $G_g$ maps into the interval $[-\Lip{g}, 0]$ and that $\psi$ is nonnegative.
From the exact same arguments as in \eqref{eq:estimates:T1T2-L2} and \eqref{eq:estimates:T1T2-H1}, 
we now obtain
$
\|\xi_h\|_{L^2(\Omega)}
\leq
 \Lip{g}k^{-1}
\|h\|_{L^2(\Omega)}
$
and
$
\| |\nabla \xi_h |\|_{L^2(\Omega)}
\leq \Lip{g} k^{-1/2} \|h\|_{L^2(\Omega)}
$
for all $u, h \in H_0^1(\Omega)$
and, as a consequence,
\begin{equation*}
\begin{aligned}
\|G_\Phi(u)h\|_{H_0^1(\Omega)}
&\leq
\|\varphi \|_{L^\infty(\Omega)} \||\nabla \xi_h |\|_{L^2(\Omega)}
+
\| |\nabla \varphi| \|_{L^\infty(\Omega)} \| \xi_h \|_{L^2(\Omega)}
\\
&\leq
\Lip{g} 
\left ( \|\varphi \|_{L^\infty(\Omega)}  k^{-1/2} 
+
\| |\nabla \varphi| \|_{L^\infty(\Omega)} k^{-1}
\right )\|h\|_{L^2(\Omega)}
\\
&\leq
C_P(\Omega)\,
\Lip{g} 
\left ( \|\varphi \|_{L^\infty(\Omega)}  k^{-1/2} 
+
\| |\nabla \varphi| \|_{L^\infty(\Omega)} k^{-1}
\right )\|h\|_{H_0^1(\Omega)}.
\end{aligned}
\end{equation*}
This establishes \eqref{eq:G_phi_bound} and completes the proof. 
\end{proof}

\begin{remark}
\label{rem:Poincare}
If $\Omega$ is contained in an open cube with sides of length $l$,
then the 
Poincaré constant
satisfies 
$C_P(\Omega) \leq l$;
see \cite[Theorem 1.5]{Braess2001finite}.
For $d=1$, one easily checks that 
$C_P(\Omega) = \diam(\Omega)/\pi$.
\end{remark}

By comparing  \cref{th_T_probs} 
with  \cref{th:obst_QVI},
we arrive at:

\begin{corollary}
\label{cor:gamma_prop}
Suppose that $g$ is globally Lipschitz continuous and that 
\begin{equation*}
\gamma := 	
C_P(\Omega)\,\Lip{g} 
\left ( \|\varphi \|_{L^\infty(\Omega)}  k^{-1/2} 
+
\| |\nabla \varphi| \|_{L^\infty(\Omega)} k^{-1}
\right )
\in [0,1).
\end{equation*}
Then the map $\Phi\colon H_0^1(\Omega) \to Y_2$
in \cref{th_T_probs} satisfies the conditions \ref{th:obst_QVI:i} to \ref{th:obst_QVI:iv}
in \cref{th:obst_QVI} with the  
derivative $G_\Phi\colon H_0^1(\Omega) \to \LL(H_0^1(\Omega), Y_2)$
 in \cref{th_T_probs}\ref{th_T_probs:iii}, $p=2$, 
and $B := H_0^1(\Omega)$.
\end{corollary} 

In the special case $d=1$, we can 
exploit that $H^1(\Omega)$ embeds into $L^\infty(\Omega)$
to obtain the following variant of 
\cref{th_T_probs} that only requires 
$g$ to be locally Lipschitz continuous.

\begin{theorem}[Properties of \eqref{eq:thermo_obst_mapprob} when $d=1$]
\label{th_T_probs_1d_local}
Assume, in addition to 
\cref{ass:SemilinearObstacleMap}, 
that $d=1$  and 
$\Psi_0 \equiv 0$.
Then \eqref{eq:T_PDE}
possesses a unique (weak) solution 
$T \in H^1(\Omega)$ for all $u \in H_0^1(\Omega)$.
This solution satisfies $\varphi T \in Y_2$. 
If, further, we define
$N_R := |\Omega|^{1/2} R /2 = \diam(\Omega)^{1/2}R/2$ and 
\begin{equation}
\label{eq:MRdef}
\begin{aligned}
M_R &:=  
N_R
+
\|\psi\|_{L^\infty(\Omega)}
\left (
|\Omega|^{- 1/2} k^{-1} 
+
|\Omega|^{1/2} k^{-1/2} \right )
\Big(
\mathrm{Lip}(g, [- N_R, N_R])
\frac{2 N_R}{\pi}
+
|g(0)| 
\Big )|\Omega|^{1/2}
\end{aligned}
\end{equation}
for all $R>0$, then the following statements are true for
the operator $\Phi\colon H_0^1(\Omega) \to Y_2$, $u \mapsto \varphi T$:
\begin{enumerate}[label=\roman*)]
\item\label{th_T_probs_1d_local:i} 
For all $R>0$ and 
all $u_1, u_2 \in B_R^{H_0^1(\Omega)}(0)$,
it holds 
\begin{equation}
\label{eq:randomeq_2628}
\begin{aligned}
	&\|\Phi(u_1) - \Phi(u_2)\|_{H_0^1(\Omega)}
        \\	
 &\quad\leq
\mathrm{Lip}(g,[-M_R , M_R])
\left ( \|\varphi \|_{L^\infty(\Omega)}  k^{-1/2} 
+
\| \varphi' \|_{L^\infty(\Omega)} k^{-1}
\right )
\frac{|\Omega|}{\pi}\|u_1 - u_2\|_{H_0^1(\Omega)}.
\end{aligned}
\end{equation} 
\item\label{th_T_probs_1d_local:ii}  For all $R> 0$,
there exists a constant $C_R>0$ satisfying 
\begin{equation*}
	\|\Phi(u_1) - \Phi(u_2)\|_{Y_2}
	\leq
	C_R
	\| u_1  -  u_2\|_{H_0^1(\Omega)}\qquad \forall u_1, \smash{u_2 \in B_R^{H_0^1(\Omega)}(0)}.
\end{equation*}
\item\label{th_T_probs_1d_local:iii}
Let $G_\Phi$
be defined as in 
\cref{th_T_probs}\ref{th_T_probs:iii}.
Then $\Phi$ is Newton differentiable as a function from $H_0^1(\Omega)$ to $Y_2$
with derivative $G_\Phi$ and, for every $R>0$, it holds
\begin{equation}
\label{eq:randomeq_2628-2}
 \begin{aligned}
&\|G_\Phi(u)\|_{\LL(H_0^1(\Omega), H_0^1(\Omega))}
\\
&\quad\leq
\lim_{t \searrow M_R}
\mathrm{Lip}(g,[-t , t])
\left ( \|\varphi \|_{L^\infty(\Omega)}  k^{-1/2} 
+
\| \varphi' \|_{L^\infty(\Omega)} k^{-1}
\right )\frac{|\Omega|}{\pi}\quad \forall u \in \smash{B_R^{H_0^1(\Omega)}(0)}.
\end{aligned}
\end{equation}
\end{enumerate}
\end{theorem}
\begin{proof}
Let $R>0$ be arbitrary and 
suppose that $u \in B_R^{H_0^1(\Omega)}(0)$ is given.
Define 
\begin{align*}
\hat g(s)
:=
\begin{cases}
g(-M_R) & \text{ if } s \leq -M_R,
\\
g(s) &\text{ if } s \in (-M_R, M_R),
\\
g(M_R) &\text{ if } s \geq M_R,
\end{cases}
\qquad \forall s \in \R,
\end{align*}
and consider the differential equation
\begin{equation}
\label{eq:T_hat_ODE}
k\hat T- \hat T''  = \hat g(\psi \hat T- u) \text{ in $\Omega$}, 
\qquad
\partial_\nu \hat T  = 0 \text{ on $\partial \Omega$}.
\end{equation}
Then $\hat g$ is globally Lipschitz continuous with Lipschitz constant 
$\Lip{\hat g} = \mathrm{Lip}(g, [-M_R, M_R])$ and we obtain from \cref{th_T_probs} that 
\eqref{eq:T_hat_ODE} has a unique (weak) solution 
$\hat T \in H^1(\Omega)$. Note that,
by proceeding along the lines of \eqref{eq:Aucoercive},
we obtain  
\begin{equation*}
\begin{aligned}
k \|\hat T\|_{L^2(\Omega)}^2
+
\|\hat T'\|_{L^2(\Omega)}^2
&\leq
\langle k\hat T - \hat T'' - \hat g(\psi \hat T- u) + \hat g(- u) , \hat T\rangle_{H^1(\Omega)^*,H^1(\Omega)}
=
  ( 
\hat g(- u), \hat T
  )_{L^2(\Omega)}.
\end{aligned}
\end{equation*}
Analogously to \eqref{eq:estimates:T1T2-L2}
and \eqref{eq:estimates:T1T2-H1}, this yields
\[
\|\hat T\|_{L^2(\Omega)}
\leq 
k^{-1} \|\hat g(- u)\|_{L^2(\Omega)}
\qquad
\text{ and }
\qquad
\|\hat T'\|_{L^2(\Omega)}
\leq 
k^{-1/2} \|\hat g(- u)\|_{L^2(\Omega)}.
\]
From the mean value theorem
and the fact that 
elements of $H^1(\Omega)$
possess a 
$C(\bar \Omega)$-representative for $d=1$, we 
further obtain that there exist 
$\hat x, \bar x \in \bar \Omega$  
satisfying 
\[
\|\hat T\|_{L^\infty(\Omega)}
=
|\hat T(\hat x)|
\qquad
\text{ and }
\qquad
\hat T(\bar x) = \frac{1}{|\Omega|}
\int_\Omega \hat T \mathrm{d}x.
\]
Due to the inequality 
of Cauchy--Schwarz and the fundamental theorem 
of calculus, this yields
\begin{equation*}
\begin{aligned}
\|\hat T\|_{L^\infty(\Omega)}
=
\left |
\hat T(\bar x)
+
\int_{\bar x}^{\hat x} \hat T' \mathrm{d}x
\right |
=
\left |
\frac{1}{|\Omega| }
\int_\Omega \hat T \mathrm{d}x
+
\int_{\bar x}^{\hat x} \hat T' \mathrm{d}x
\right |
\leq
|\Omega|^{- 1/2} \|\hat T\|_{L^2(\Omega)}
+
|\Omega|^{1/2} \|\hat T'\|_{L^2(\Omega)}.
\end{aligned}
\end{equation*}
In combination with
the (sharp) estimate
$\|v\|_{L^\infty(\Omega)} \leq  |\Omega|^{1/2}
\|v\|_{H_0^1(\Omega)}/2$
for all $v \in H_0^1(\Omega)$
(that is easily established by variational calculus),
$\|u\|_{H_0^1(\Omega)}  \leq  R$,
and  \cref{rem:Poincare},
it now follows that
\begin{equation*}
\begin{aligned}
\|\hat T\|_{L^\infty(\Omega)}
&\leq
|\Omega|^{- 1/2} \|\hat T\|_{L^2(\Omega)}
+
|\Omega|^{1/2} \|\hat T'\|_{L^2(\Omega)}
\\
&\leq
\left (
|\Omega|^{- 1/2} k^{-1} 
+
|\Omega|^{1/2} k^{-1/2} \right )\|\hat g(- u)\|_{L^2(\Omega)}
\\
&\leq
\left (
|\Omega|^{- 1/2} k^{-1} 
+
|\Omega|^{1/2} k^{-1/2} \right )
\Big (
\mathrm{Lip}(\hat g, [- \|u\|_{L^\infty(\Omega)}, \|u\|_{L^\infty(\Omega)}])
\|u\|_{L^2(\Omega)}
+
\|\hat g(0)\|_{L^2(\Omega)}
\Big )
\\
&\leq
\left (
|\Omega|^{- 1/2} k^{-1} 
+
|\Omega|^{1/2} k^{-1/2} \right )
\Big(
\mathrm{Lip}(\hat g, [- N_R, N_R])
\frac{2 N_R}{\pi}
+
|g(0)| 
\Big )|\Omega|^{1/2},
\end{aligned}
\end{equation*}
and, due to the identity 
$\mathrm{Lip}(\hat g, [- N_R, N_R]) 
= \mathrm{Lip}(
g, [- N_R, N_R])$, that
\begin{equation*}
\begin{aligned}
\|\psi \hat T- u\|_{L^\infty(\Omega)}
\leq
\|\psi\|_{L^\infty(\Omega)}
\|\hat T\|_{L^\infty(\Omega)}
+
\|u\|_{L^\infty(\Omega)}
\leq
\|\psi\|_{L^\infty(\Omega)}
\|\hat T\|_{L^\infty(\Omega)}
+
N_R
\leq
M_R.
\end{aligned}
\end{equation*}
As $\hat g$ coincides with
$g$ on $[-M_R, M_R]$, the last estimate implies 
that $\hat T$ is also a solution of \eqref{eq:T_PDE}.
Since \eqref{eq:T_PDE} can have at most one solution
(as one may easily check by means of a contradiction argument),
this shows that \eqref{eq:T_PDE} 
is uniquely solvable for all $u \in H_0^1(\Omega)$.
Note that,
if we denote by 
$P\colon H_0^1(\Omega) \to H^1(\Omega)$, $u \mapsto T$,
the solution operator of \eqref{eq:T_PDE},
by $P_R\colon H_0^1(\Omega) \to H^1(\Omega)$,
$u \mapsto \hat T$, the 
solution operator of \eqref{eq:T_hat_ODE},
and by $\Phi, \Phi_R \colon H_0^1(\Omega) \to H_0^1(\Omega)$ the functions 
$\Phi(u) := \varphi P(u)$ and 
$\Phi_R(u) := \varphi P_R(u)$,
respectively, 
then it follows from the above 
considerations that 
$P(u) = P_R(u)$
and
$\Phi(u) = \Phi_R(u)$
hold for all $u \in H_0^1(\Omega)$ with 
$\|u\|_{H_0^1(\Omega)} \leq R$. As 
\cref{th_T_probs} applies to 
$\Phi_R$ for all $R>0$ and
since $\Lip{\hat g} = \mathrm{Lip}(g, [-M_R, M_R])$,
it now follows immediately that 
$\varphi P(u) \in Y_2$ holds for all 
$u \in H_0^1(\Omega)$ and that 
$\Phi$ satisfies 
the assertions in \ref{th_T_probs_1d_local:i},
\ref{th_T_probs_1d_local:ii},
and \ref{th_T_probs_1d_local:iii}. 
(Note that,
in \eqref{eq:randomeq_2628} and 
\eqref{eq:randomeq_2628-2},
we have used 
the identity 
$ C_P(\Omega) =\diam(\Omega)/\pi = |\Omega|/\pi$ 
obtained from \cref{rem:Poincare},
and 
that the limit in \eqref{eq:randomeq_2628-2} 
is necessary since
$\partial_c\hat g(s) = \partial_c g(s)$
is only true for $s \in (-M_R, M_R)$.)
\end{proof}

\begin{remark}
\Cref{th_T_probs_1d_local} can be extended straightforwardly
to the case $\Psi_0 \in L^\infty(\Omega)$.
We have assumed that $\Psi_0 \equiv 0$ holds for the sake of 
simplicity and to avoid additional technicalities. 
\end{remark}

\begin{corollary}
\label{cor:local_thermo_obstacle_map}
Suppose that $d=1$ holds,
that $\Psi_0 \equiv 0$, and that $R>0$ is a number such that 
\begin{equation*}
\gamma_R := 	
\lim_{t \searrow M_R}
\mathrm{Lip}(g,[-t , t])
\left ( \|\varphi \|_{L^\infty(\Omega)}  k^{-1/2} 
+
\| \varphi' \|_{L^\infty(\Omega)} k^{-1}
\right )\frac{|\Omega|}{\pi}
\in [0,1),
\end{equation*}
where $M_R$ is defined as in \eqref{eq:MRdef}.
Then the map $\Phi\colon H_0^1(\Omega) \to Y_2$, $u \mapsto \varphi T$,
in \cref{th_T_probs_1d_local}
satisfies the conditions \ref{th:obst_QVI:i} to \ref{th:obst_QVI:iv}
in \cref{th:obst_QVI} with the  
derivative $G_\Phi\colon H_0^1(\Omega) \to \LL(H_0^1(\Omega), Y_2)$
 in \cref{th_T_probs_1d_local}\ref{th_T_probs:iii}, $p=2$, 
and $B := B_R^{H_0^1(\Omega)}(0)$.
\end{corollary} 

As we will see in \cref{subsec:7.3},
\Cref{th_T_probs_1d_local,cor:local_thermo_obstacle_map} 
cover examples of obstacle-type QVIs \eqref{eq:QVI}  
that possess several solutions and, as a consequence, can only be studied 
within the localized framework of \cref{sec:framework_fp_H_multiple_solns}.

\subsection{An alternative application: semilinear VIs}
\label{sec:application_nonlinearVI}
We conclude this section by demonstrating that the 
semismooth Newton framework developed in 
\cref{sec:framework_applied_to_SPhi} 
can also be applied to 
variational problems that are not of 
QVI-type. To this end, we consider the 
semilinear variational inequality
\begin{equation}
\label{eq:semilinear_VI}
    \begin{aligned}
\text{Find $u \in H_0^1(\Omega)$ such that  } 
u \in K,
~\langle -\Delta u - b_1(u)b_2(u) - f, v - u \rangle_{H^{-1}(\Omega),H_0^1(\Omega)} \geq 0~
\forall v \in K,
\end{aligned}
\end{equation}
with $K := \{ v \in H^1_0(\Omega) \mid v \leq \Phi_0\text{ a.e.\ in $\Omega$}\}$
and with $b_1(u)b_2(u)$ as the
pointwise-a.e.\ product of two 
Nemytskii operators
induced by (potentially nonsmooth) functions $b_1,b_2\colon \R \to \R$. 

Variational inequalities of the type 
\eqref{eq:semilinear_VI} and their
PDE-counterparts 
$-\Delta u = b_1(u)b_2(u)+f$ in $\Omega$,
$u=0$ on $\partial \Omega$, 
are  prototypical examples of nonlinear variational problems that are widely studied in the literature. 
The existence of solutions to such problems is typically established by proving that 
a linearization of the solution mapping 
is contractive on a
suitably chosen fixed ball. The localization assumptions in \cref{cor:composite_local} ask for precisely this kind of contraction property. Hence, 
working with \cref{cor:composite_local} (and, in particular, with the smallness condition \ref{ass:standing_general_proj:iv} in this corollary) is
very natural when studying  these kinds of nonlinear variational
problems.

Throughout this subsection, we work with the following setting:

\begin{assumption}
\label{ass:standing_semilinear_VI}
~
\begin{enumerate}[label=\roman*)]
\item $\Omega \subset \R^d$, $d \in \N$, $d \leq 5$, is a nonempty open bounded set;
\item $\Phi_0$ 
is as in \cref{ass:standing_obstacle_map}\ref{ass:standing_obstacle_map:iv};
\item\label{ass:standing_semilinear_VI:iii}
 $p$ is an exponent
 satisfying $ \max(1,2d/(d+2)) < p < \infty$
 and, in the case $d \geq 3$, $p < d/(d-2)$;
\item $f \in L^p(\Omega)$ is a given function;
\item $b_1, b_2\colon \R \to \R$ are globally 
Lipschitz continuous and 
Newton differentiable  with derivatives 
$G_{b_i}\colon \R \to \R$, $i=1,2$,
that are Borel-measurable 
and satisfy $G_{b_i}(s) \in \partial_c b_i(s)$ for all $s \in \R$.
\end{enumerate}
\end{assumption}

Note that the condition $d \leq 5$ arises naturally from 
the requirements on $p$ in
\cref{ass:standing_semilinear_VI}\ref{ass:standing_semilinear_VI:iii}
and that the assumptions on $b_1$ and $b_2$ allow to model 
nonsmooth semilinearities in \eqref{eq:semilinear_VI} with linear/quadratic growth
(e.g., $\max(0,u)(u + \cos(u))$ with $b_1(s) := \max(0,s)$, $b_2(s) := s + \cos(s)$). 
Recall further that standard Sobolev embeddings imply 
that $H^1_0(\Omega)$ is continuously 
embedded into $L^{q}(\Omega)$ for all 
\[
1\leq q \begin{cases}
    \leq \infty & \text{if } d=1,\\
    < \infty & \text{if } d=2,\\
    \leq  \frac{2d}{d-2} &\text{if } d \geq 3. 
\end{cases}
\] 
In combination with 
\cref{ass:standing_semilinear_VI}\ref{ass:standing_semilinear_VI:iii}, this yields that 
$H^1_0(\Omega)$ embeds continuously  
into $L^{2p}(\Omega)$,
that $L^p(\Omega)$ embeds continuously   
into $H^{-1}(\Omega)$, and, 
since the Hölder conjugate $p' := p/(p-1)$ of $p$
satisfies $p' < 2d/(d-2)$ for $d\geq 3$, that
$H^1_0(\Omega)$ embeds continuously into 
$L^{p'}(\Omega)$.

To see that \eqref{eq:semilinear_VI} 
is equivalent to a fixed-point problem of the form \eqref{eq:SPhi_FP}
and, thus, covered by the analysis of \cref{sec:framework_applied_to_SPhi},
we have to study the mapping 
$H_0^1(\Omega) \ni u \mapsto b_1(u)b_2(u) + f\in L^p(\Omega)$ which we shall denote by $\Phi$, i.e.,
\begin{equation}
\Phi\colon H_0^1(\Omega) \to L^p(\Omega),\qquad
    \Phi(u) := b_1(u)b_2(u) + f.\label{eq:defn_of_Phi_semilinearVI}
\end{equation}

\begin{lemma}
\label{lem:Phi_semilin}
Suppose that \cref{ass:standing_semilinear_VI} holds. Then the function $\Phi\colon H_0^1(\Omega) \to L^p(\Omega)$ in
\eqref{eq:defn_of_Phi_semilinearVI}
is well defined and Newton differentiable 
with the derivative 
$G_\Phi\colon H_0^1(\Omega) \to \LL(H_0^1(\Omega), L^p(\Omega))$ given by 
\begin{equation}
\label{eq:Phi_der_semi}
G_\Phi(u)h := b_1(u)G_{b_2}(u)h + b_2(u) G_{b_1}(u)h\qquad \forall  
u,h \in H_0^1(\Omega).
\end{equation}
Further, it holds 
\begin{equation}
\label{eq:Phi_der_bound_semi}
\norm{G_\Phi(u)}{\LL(H_0^1(\Omega), L^p(\Omega))}
\leq
\left (
\Lip{b_1} \norm{b_2(u)}{L^{2p}(\Omega)}
+
\Lip{b_2} \norm{b_1(u)}{L^{2p}(\Omega)}
\right )
\|\iota_{2p}\|_{\LL(H_0^1(\Omega), L^{2p}(\Omega))}
\end{equation}
for all $u  \in H_0^1(\Omega)$
and
\begin{equation}
\label{eq:Phi_Lip_semi}
\begin{aligned}
&\norm{\Phi(u_1) - \Phi(u_2)}{ L^p(\Omega)}
\\
&\leq
\left (
\Lip{b_1} \norm{b_2(u_2)}{L^{2p}(\Omega)}
+
\Lip{b_2} \norm{b_1(u_1)}{L^{2p}(\Omega)}
\right )
\|\iota_{2p}\|_{\LL(H_0^1(\Omega), L^{2p}(\Omega))}
\|u_1 - u_2\|_{H_0^1(\Omega)}
\end{aligned}
\end{equation}
for all $u_1, u_2 \in H_0^1(\Omega)$, where 
$\iota_{2p} \in \LL(H_0^1(\Omega), L^{2p}(\Omega))$ denotes the 
embedding of $H_0^1(\Omega)$ into $L^{2p}(\Omega)$.
\end{lemma}

\begin{proof}
From Hölder's inequality, the triangle inequality,
our assumptions on $p$,
the Lipschitz continuity of $b_1$ and $b_2$,
and the Sobolev embeddings,
we obtain that
\begin{equation*}
\begin{aligned}
&\left \| b_1(u) b_2(u) + f \right\|_{L^p(\Omega)}
\\
&\quad\leq
\left \| f \right\|_{L^p(\Omega)}
+
\left \| b_1(u) \right\|_{L^{2p}(\Omega)} 
\left \| b_2(u) \right\|_{L^{2p}(\Omega)} 
\\
&\quad\leq 
\left \| f \right\|_{L^p(\Omega)}
+
\left (
\Lip{b_1}
\left \| u \right\|_{L^{2p}(\Omega)}
+
\left \| b_1(0) \right\|_{L^{2p}(\Omega)}
\right )
\left (
\Lip{b_2}
\left \| u \right\|_{L^{2p}(\Omega)}
+
\left \| b_2(0) \right\|_{L^{2p}(\Omega)}
\right )
\\
&\quad\leq 
\left \| f \right\|_{L^p(\Omega)}
+
\prod_{i=1}^2
\left (
\Lip{b_i}
\|\iota_{2p}\|_{\LL(H_0^1(\Omega), L^{2p}(\Omega))}
\left \| u \right\|_{H_0^1(\Omega)}
+
\left \| b_i(0) \right\|_{L^{2p}(\Omega)}
\right )
\end{aligned}
\end{equation*}
holds for all $u \in H_0^1(\Omega)$.
This shows that 
$\Phi$ is well defined as a function from $H_0^1(\Omega)$
to $L^p(\Omega)$. Consider now an exponent $q > 2p$ that satisfies 
$q < \infty$ in the case $d \leq 2$ and $q < 2d/(d-2)$ in the case $d \geq 3$.
(Such a $q$ exists by 
\cref{ass:standing_semilinear_VI}\ref{ass:standing_semilinear_VI:iii}.)
Then it follows from  \cite[Theorem 3.49]{Ulbrich2011}
that the maps $F_i\colon L^q(\Omega) \to L^{2p}(\Omega)$,
$u \mapsto b_i(u)$, $i=1,2$, are Newton differentiable 
with Newton derivatives $G_{F_i}(u)h = G_{b_i}(u)h$.
In combination with \cref{lem:product_rule}
(applied to $U = L^q(\Omega)$, $V = W = L^{2p}(\Omega)$,
$Z = L^p(\Omega)$, and 
$a\colon L^{2p}(\Omega) \times L^{2p}(\Omega) \to L^p(\Omega)$ 
as the pointwise-a.e.\ multiplication), this yields that 
the function $F\colon L^q(\Omega) \to L^{p}(\Omega)$,
$u \mapsto b_1(u)b_2(u)$, is Newton differentiable with 
Newton 
derivative $G_F(u)h = b_1(u)G_{b_2}(u)h + b_2(u) G_{b_1}(u)h$.
That $\Phi$ is Newton differentiable 
as a function from $H_0^1(\Omega)$ to $L^p(\Omega)$
with the derivative in \eqref{eq:Phi_der_semi}
now follows immediately from the 
Sobolev embeddings and the sum rule 
for Newton derivatives. Note that 
\eqref{eq:Phi_der_bound_semi} is an immediate
consequence 
of the estimate 
\begin{equation*}
\begin{aligned}
&\norm{G_\Phi(u)}{\LL(H_0^1(\Omega), L^p(\Omega))}
\\
&\quad =
\sup_{h \in H_0^1(\Omega), \norm{h}{H_0^1(\Omega)}\leq 1}
\norm{b_1(u)G_{b_2}(u)h + b_2(u) G_{b_1}(u)h}{L^p(\Omega)}
\\
&\quad \leq
\sup_{h \in H_0^1(\Omega), \norm{h}{H_0^1(\Omega)}\leq 1}
\left (
 \norm{b_1(u)}{L^{2p}(\Omega)} 
 \Lip{b_2}
+
 \norm{b_2(u)}{L^{2p}(\Omega)} 
 \Lip{b_1} \right )
\norm{h}{L^{2p}(\Omega)}
\\
&\quad \leq 
\left (
\Lip{b_1} \norm{b_2(u)}{L^{2p}(\Omega)}
+
\Lip{b_2} \norm{b_1(u)}{L^{2p}(\Omega)}
\right )
\|\iota_{2p}\|_{\LL(H_0^1(\Omega), L^{2p}(\Omega))}
\quad \forall u \in H_0^1(\Omega).
\end{aligned}
\end{equation*}
Similarly, we also obtain
\begin{equation*}
\begin{aligned}
&\norm{\Phi(u_1) - \Phi(u_2)}{L^p(\Omega)}
\\
&\quad = 
\norm{b_1(u_1)b_2(u_1) - b_1(u_1)b_2(u_2) + b_1(u_1)b_2(u_2) - b_1(u_2)b_2(u_2)}{L^p(\Omega)}
\\
&\quad \leq
\norm{b_1(u_1)}{L^{2p}(\Omega)}
\norm{b_2(u_1) - b_2(u_2)}{L^{2p}(\Omega)}
+
\norm{b_2(u_2)}{L^{2p}(\Omega)}
\norm{b_1(u_1) - b_1(u_2)}{L^{2p}(\Omega)}
\\
&\quad \leq
\left (
\Lip{b_2}
\norm{b_1(u_1)}{L^{2p}(\Omega)}
+
\Lip{b_1}
\norm{b_2(u_2)}{L^{2p}(\Omega)}
\right )
\norm{u_1 - u_2}{L^{2p}(\Omega)}
\\
&\quad \leq
\left (
\Lip{b_2}
\norm{b_1(u_1)}{L^{2p}(\Omega)}
+
\Lip{b_1}
\norm{b_2(u_2)}{L^{2p}(\Omega)}
\right )
\|\iota_{2p}\|_{\LL(H_0^1(\Omega), L^{2p}(\Omega))}
\norm{u_1 - u_2}{H_0^1(\Omega)}
\end{aligned}
\end{equation*}
for all $u_1, u_2 \in H_0^1(\Omega)$. This establishes 
\eqref{eq:Phi_Lip_semi} and completes the proof. 
\end{proof}

Recall that our assumptions on $p$ imply that 
$L^p(\Omega)$
embeds continuously into 
$H^{-1}(\Omega)$.
In combination with \cref{lem:Phi_semilin},
this allows us to recast the 
semilinear VI \eqref{eq:semilinear_VI} in the form 
\begin{equation}
\label{eq:general_FP_semi}
\text{Find } u \in H_0^1(\Omega) 
\text{ such that }  u = S_0(\Phi(u)),
\end{equation}
where $\Phi\colon H_0^1(\Omega) \to L^p(\Omega)$
is defined as in \eqref{eq:defn_of_Phi_semilinearVI} and 
$S_0\colon H^{-1}(\Omega) \to H_0^1(\Omega)$
again denotes the solution map of \eqref{eq:upper_obst_prob_rhs}, 
i.e., the function that maps a 
source term $w \in H^{-1}(\Omega)$
to the solution $u$ of 
\begin{align}
\label{eq:obstacle_again_42}
\text{Find $u \in H_0^1(\Omega)$ such that  } 
u \in K,~~\langle -\Delta u - w, v - u \rangle_{H^{-1}(\Omega),H_0^1(\Omega)} \geq 0 \quad \forall v \in K.
\end{align}
From
\cite{ChristofWachsmuthBiObstacle,ChristofWachsmuthUniObstacle},
we (again) obtain that $S_0$ is Newton differentiable as a function from 
$L^p(\Omega)$ to $H_0^1(\Omega)$. 
For the convenience of the reader, we 
restate this Newton differentiability property 
in a way that fits to the 
application context of \eqref{eq:semilinear_VI}.

\begin{lemma}
\label{lem:S0_semilin}
Suppose that \cref{ass:standing_semilinear_VI} holds. Then the solution map $S_0\colon H^{-1}(\Omega) \to H_0^1(\Omega)$,
$w \mapsto u$, of the obstacle problem \eqref{eq:obstacle_again_42}
is well defined. Further, $S_0$ is 
Newton differentiable as a function 
from $L^p(\Omega)$ to $H_0^1(\Omega)$
with the Newton derivative 
$G_{S_0}\colon L^p(\Omega) \to \LL(L^p(\Omega), H_0^1(\Omega))$
defined by 
\begin{equation*}
G_{S_0}(w)h  = z \quad \text{ if and only if}
\quad z \in H^1_0(I(w)),~~
\langle -\Delta z - h, v \rangle_{H^{-1}(\Omega),H_0^1(\Omega)} = 0~\forall v \in H_0^1(I(w))
\end{equation*}
for all $w, h \in L^p(\Omega)$.
Here, $I(w) := \Omega \setminus \{ x \in \Omega \mid S_0(w)(x) = \Phi_0(x)\}$ denotes the inactive set associated with $w$,
defined in the same sense as in \eqref{eq:active_set_def}.
Moreover, it holds 
\begin{equation}
\label{eq:S0_der_bound}
\norm{G_{S_0}(w)}{\LL(L^p(\Omega), H_0^1(\Omega))}
\leq
\|\iota_{p'}\|_{\LL(H_0^1(\Omega), L^{p'}(\Omega))}
\qquad 
\forall w \in L^p(\Omega)
\end{equation}
and 
\begin{equation}
\label{eq:S0_Lip}
\norm{S_0(w_1)-S_0(w_2)}{H_0^1(\Omega)}
\leq
\|\iota_{p'}\|_{\LL(H_0^1(\Omega), L^{p'}(\Omega))}
\norm{w_1 - w_2}{L^p(\Omega)}
\qquad 
\forall w_1, w_2 \in L^p(\Omega),
\end{equation}
where $p':= p/(p-1)$ denotes the Hölder conjugate 
of $p$ and $\iota_{p'}$ the embedding 
of $H_0^1(\Omega)$ into $L^{p'}(\Omega)$.
\end{lemma}

\begin{proof}
That $S_0$ is well defined again follows 
from our assumptions on $\Phi_0$
and
\cite[Theorem 4:3.1]{Rodrigues}.
That $S_0$ is Newton differentiable 
as a function $S_0\colon L^p(\Omega) \to H_0^1(\Omega)$
with derivative $G_{S_0}$
is a consequence of 
\cite[Theorem 4.3]{RaulsWachsmuth},
the inclusions in 
\cite[Proposition 2.11]{RaulsWachsmuth},
and 
\cite[Theorem 4.4]{ChristofWachsmuthBiObstacle};
cf.\ the proof of \cref{lem:S_obs_Newton}.
The estimate \eqref{eq:S0_der_bound} follows 
trivially from 
\begin{equation*}
\begin{aligned}
\|G_{S_0}(w)h\|_{H_0^1(\Omega)}^2 
=
\langle h, G_{S_0}(w)h \rangle_{H^{-1}(\Omega),H_0^1(\Omega)}
&\leq
\|h\|_{L^p(\Omega)}
\|G_{S_0}(w)h \|_{L^{p'}(\Omega)}
\\
&\leq
\|\iota_{p'}\|_{\LL(H_0^1(\Omega), L^{p'}(\Omega))}
\|h\|_{L^p(\Omega)}
\|G_{S_0}(w)h\|_{H_0^1(\Omega)}
\end{aligned}
\end{equation*}
for all $w, h \in L^p(\Omega)$. 
The Lipschitz estimate \eqref{eq:S0_Lip} 
is obtained along the exact same lines
by choosing $S_0(w_1)$ as the test function in the 
VI for $S_0(w_2)$,
by choosing $S_0(w_2)$ as the test function in the 
VI for $S_0(w_1)$,
and by adding the 
resulting inequalities.
\end{proof}

From \cref{lem:Phi_semilin,lem:S0_semilin},
it follows immediately that 
the VI 
\eqref{eq:semilinear_VI}---or, more precisely, its reformulation \eqref{eq:general_FP_semi}---is
covered by the analysis 
of \cref{sec:framework_applied_to_SPhi}.

\begin{corollary}
\label{cor:last_cor}
Suppose that \cref{ass:standing_semilinear_VI} holds and define 
\begin{align*}
\gamma_R 
& :=
\|\iota_{2p}\|_{\LL(H_0^1(\Omega), L^{2p}(\Omega))}
\|\iota_{p'}\|_{\LL(H_0^1(\Omega), L^{p'}(\Omega))}
\\
&\quad~~ \cdot 
\sup_{v_i \in H_0^1(\Omega), \|v_i\|_{H_0^1(\Omega)}\leq R, i=1,2}
\left (
\Lip{b_1} \norm{b_2(v_1)}{L^{2p}(\Omega)}
+
\Lip{b_2} \norm{b_1(v_2)}{L^{2p}(\Omega)}
\right )\quad \forall R \in (0, \infty],
\end{align*}
where
$\iota_{2p} $
and
$\iota_{p'} $
denote the embeddings introduced in \cref{lem:Phi_semilin,lem:S0_semilin}.
Let $R \in (0,\infty)$ be given such that $\gamma_R \in [0,1)$ holds.
Then the maps 
$\Phi\colon H_0^1(\Omega) \to L^p(\Omega)$
and
$S_0\colon L^p(\Omega) \to H_0^1(\Omega)$
from \cref{lem:Phi_semilin,lem:S0_semilin}
satisfy the assumptions of  \cref{cor:composite_local}
with 
\begin{gather}
\label{eq:semilinear_defs_spaces}
X = H_0^1(\Omega),\qquad
Y = D = L^p(\Omega),\qquad
S = S_0,\qquad
B = B_R^{H_0^1(\Omega)}(0),
\quad \text{and}\quad \gamma = \gamma_R.
\end{gather}
Furthermore, if $\gamma_\infty \in [0, 1)$ holds, 
then
$\Phi\colon H_0^1(\Omega) \to L^p(\Omega)$
and
$S_0\colon L^p(\Omega) \to H_0^1(\Omega)$
satisfy the assumptions of  \cref{cor:composite_global}
with $\gamma := \gamma_\infty$ and 
$X$, $Y$, $D$, and $S$ as in \eqref{eq:semilinear_defs_spaces}.
\end{corollary}
\begin{proof}
The assertions of the corollary follow straightforwardly 
from  \cref{lem:semismooth_balls,lem:Phi_semilin,lem:S0_semilin}
by comparing with the 
assumptions of \cref{cor:composite_global,cor:composite_local}.
\end{proof}

From \cref{cor:last_cor},
we obtain that 
\eqref{eq:general_FP_semi} is amenable to \cref{algo:semiNewtonAbstract} provided the 
Lipschitz constants and the growth behavior of the 
functions $b_1$ and $b_2$ are suitable;
analogously to the obstacle-type QVIs 
in \cref{cor:gamma_prop,cor:local_thermo_obstacle_map}.

\section{Numerical experiments for obstacle-type QVIs}
\label{sec:7}

If we combine the results of 
\cref{subsec:obstacle_solution_map,subsec:solution_ops_as_obstacle},
then we obtain that the analysis of 
\cref{sec:framework_applied_to_SPhi} applies to obstacle-type 
QVIs of the form 
\begin{equation}
\label{eq:ModelQVINum}
\begin{aligned}
&\text{Find $u \in H_0^1(\Omega)$  satisfying  } 
\\
&\qquad\qquad\quad u \leq \Phi_0 + \Phi(u),
\quad\langle -\Delta u - f, v - u \rangle_{H^{-1}(\Omega),H_0^1(\Omega)} \geq 0~
\forall v \in H_0^1(\Omega), v \leq \Phi_0 + \Phi(u),
\\
&\text{with $\Phi(u)$ given by $\Phi(u) := \varphi T$ and $T$ as the solution of}
\\
&\qquad\qquad\quad kT-\Delta T = g(\Psi_0 + \psi T -u) \text{ in $\Omega$},
\qquad 
\partial_\nu T = 0 \text{ on $\partial\Omega$,}
\end{aligned}
\end{equation}
provided the quantities 
$\Omega$, $\Phi_0$, $f$, $\varphi$, $k$,
$g$, $\Psi_0$, and $\psi$
in \eqref{eq:ModelQVINum} 
satisfy the conditions 
in \cref{ass:standing_obstacle_map,ass:SemilinearObstacleMap}
for $p=2$ and 
the data is such that  \cref{lem:S_obs_well_def,lem_GS_bound_obstacle,lem:S_obs_Newton,th_T_probs,th_T_probs_1d_local}
allow to establish the 
conditions 
\eqref{eq:Phi_gamma_Lip} and \eqref{eq:GPhi_bound-2}
(or 
\eqref{eq:Phi_gamma_Lip_loc} and \eqref{eq:GPhi_bound-2_loc},
respectively)
for $X = H_0^1(\Omega)$. 
Note that \eqref{eq:ModelQVINum} 
covers the so-called thermoforming problem \cite[\S 6]{AHR}
as a special case.

\begin{remark}[Thermoforming QVI]
\label{rem:thermoforming}
For $d=2$, \eqref{eq:ModelQVINum}, with $\Psi_0 \equiv \Phi_0$ and $\psi \equiv \varphi$,
provides a simple model for the problem 
of determining the displacement $u$ of 
an elastic membrane,
clamped at the boundary $\partial \Omega$,
that has been heated and is 
pushed by means of an external force $f$
into a metallic mould with original shape $\Phi_0$ and deformation $\Phi(u)$. The deformation is due to the mould's temperature field $T$ which varies according to the
membrane's temperature.  
The function $g$ then models how the 
temperature $T$ is affected by the 
distance
between the membrane and the mould. For more details on the thermoforming problem,
we refer to \cite[\S 6]{AHR} and the references therein.
\end{remark}

In the present section, 
we provide several examples of QVIs of the type \eqref{eq:ModelQVINum}
for which all necessary assumptions are satisfied
and present the results that are obtained 
when \cref{algo:semiNewtonAbstractGeneral,algo:semiNewtonAbstract} are applied. 
Our goal is in particular to
demonstrate 
the $q$-superlinear convergence 
and mesh-independence of our semismooth Newton method. 
We begin with a detailed description 
of how we realized
and implemented \cref{algo:semiNewtonAbstractGeneral,algo:semiNewtonAbstract}
in the situation of \eqref{eq:ModelQVINum}.

\subsection{Implementation details}
\label{subsec:7.1}

\textbf{Data availability and packages used:} 
We implement our experiments in the open-source language Julia \cite{Bezanson2017}.
Our implementation makes use of the
\texttt{Gridap} package for the finite element discretization of the variational problems \cite{Badia2020, Verdugo2022} as well as the \texttt{NLsolve} \cite{nlsolve.jl}  and \texttt{LineSearches} \cite{linesearches.jl} packages, for the solution of the (smooth) 
nonlinear equations that arise when
evaluating, e.g., the solution map of
the obstacle problem \eqref{eq:upper_obst_prob}  
by means of a regularization approach. 
For the sake of reproducibility, the scripts used to generate the tables
and plots depicted in the following subsections 
can be found in the \texttt{SemismoothQVIs} package \cite{SemismoothQVIs.jl}. The version of  \texttt{SemismoothQVIs} run in our experiments is archived on Zenodo \cite{QVI-SSN.jl-zenodo}. A Python implementation  of \cref{algo:semiNewtonAbstractGeneral,algo:semiNewtonAbstract} is also available with Firedrake \cite{FiredrakeUserManual} as the finite element backend; cf.~\cite{semismoothqvis_firedrake}.

\medskip

\noindent \textbf{Evaluation of $\boldsymbol{S \circ \Phi}$:} The semismooth Newton methods in \cref{algo:semiNewtonAbstractGeneral,algo:semiNewtonAbstract} 
require the computation of the quantity 
$S(\Phi(u_i))$ in each iteration. 
(Here and in what follows, 
we write $u_i$ etc.\ instead of $x_i$ etc.\
to conform with the notation in 
\eqref{eq:ModelQVINum}.)
The 
computation of 
$S(\Phi(u_i))$
can be split into two steps:
\begin{multicols}{2}
\begin{enumerate}[label=({E}\arabic*),leftmargin=2cm]
\item Given $u_i $, compute $\Phi(u_i)$. \label{item:inner1}
\item Given $\phi$, compute $S(\phi)$. \label{item:inner2}
\end{enumerate}
\end{multicols}
For the QVI \labelcref{eq:ModelQVINum},  
\labelcref{item:inner1} is equivalent to
the solution of 
a (potentially nonsmooth) semi\-linear
PDE and hence \labelcref{item:inner1} can be realized 
by means of a (semismooth) Newton method. 

The evaluation of $S$ in \labelcref{item:inner2} is equivalent to solving 
the obstacle problem \eqref{eq:upper_obst_prob}. As 
this is a classical problem, a myriad of algorithms exist for 
its numerical solution.
In this paper, we opt for a path-following smoothed Moreau--Yosida regularization (PFMY) \cite{adam2019semismooth} followed by a feasibility restoration by means 
of iterations of the primal-dual active set (PDAS) method \cite{HintermullerIto2003}. By combining these algorithms, we find that
the number of iterations needed 
for the evaluation 
\labelcref{item:inner2} does not grow uncontrollably 
as the mesh width goes to zero and the feasibility of $S(\Phi(u_i))$
is guaranteed.

Recall that the PFMY-algorithm for the 
solution of the obstacle problem relies 
on the 
Huber-type function 
$\sigma_\rho\colon \R \to \R$ defined by
\begin{align}
\label{eq:Huber}
\sigma_\rho(u) :=
\begin{cases}
0 & \text{if} \; u \leq 0,\\
 u^2/(2\rho) & \text{if} \; 0 < u < \rho,\\
u - {\rho}/{2} & \text{if} \; u \geq \rho,
\end{cases}
\qquad \text{for } \rho > 0.
\end{align}
A PFMY-algorithm for approximating the solution 
$u = S(\phi)$ of \eqref{eq:upper_obst_prob} for 
given $f$, $\phi$, and  
$\Phi_0$
chooses a sequence of Moreau--Yosida parameters $\rho_0 > \rho_1 > \ldots > \rho_J > 0$ and solves
for each $j$ the subproblem
\begin{align}
\label{eq:MoreauYosida}
u_{\rho_j} \in H^1_0(\Omega),
\quad
( u_{\rho_j},  v )_{H_0^1(\Omega)} + \left( {\rho^{-1}_j} \sigma_{\rho_j}(u_{\rho_j}-\phi - \Phi_0) - f, v \right)_{L^2(\Omega)} =0 \quad \forall  v \in H^1_0(\Omega).
\end{align}
When $u_{\rho_j}$ is computed, it serves as the initial guess for the iterative solution of the subsequent subproblem with parameter
$\rho_{j+1}$. In our implementation, all the subproblems that appear 
were solved by means of a classical Newton algorithm. 
The last PFMY-iterate
$u_{\rho_J}$ is used as the initial guess for the PDAS-method
in our solver which restores 
the feasibility of the solution.

\medskip

\noindent \textbf{Action of 
$\boldsymbol{G_R(u_i)^{-1}}$:} Next, we detail how 
to compute 
the Newton iterate 
$u_N$  in steps \ref{algo1:x_N_general} and \ref{algo1:x_N} of \cref{algo:semiNewtonAbstractGeneral,algo:semiNewtonAbstract}, respectively. 
In view of 
\cref{lem:R:prop}\ref{lem:R:prop:ii},
given the previous iterate $u_i$ and $u_B := S(\Phi(u_i))$,
to determine $u_N$,
one has to 
(approximately)
solve the linear system 
\begin{align}
\label{eq:NewtonSysReformulated}
\left( \mathrm{Id} - G_S(\Phi(u_i))G_\Phi(u_i) \right)\delta u_N = u_B - u_i
\end{align}
for $\delta u_N$
and set 
$u_N := u_i + \delta u_N$. In order to realize the composition $G_S(\Phi(u_i))G_\Phi(u_i)$ in \eqref{eq:NewtonSysReformulated}, we introduce the auxiliary
variables 
$\eta := G_\Phi(u_i)\delta u_N$ and 
$\mu := G_S(\Phi(u_i))\eta - \eta = Z_S(\Phi(u_i))\Delta \eta$, where $Z_S$ is defined as in \cref{def:GS}. As discussed in \cref{th_T_probs}\ref{th_T_probs:iii}, 
$\eta$ then satisfies 
$\eta = \varphi \xi$ 
with $\xi \in H^1(\Omega)$
as the solution of
\eqref{eq:xi_PDE} with 
$h:=\delta u_N$. From
\cref{def:GS}, we 
further obtain 
that 
$\mu$ is the weak solution of 
$-\Delta \mu - \Delta \eta = 0$ in $I_i$,
$\mu = 0$ in $\bar \Omega \setminus I_i$,
where 
$I_i$ and $A_i$ denote the 
inactive and active set associated 
with the iterate $u_i$, respectively, i.e., 
$A_i = A(\Phi(u_{i})) = 
\{
x\in \Omega 
\mid
u_B(x) = \Phi_0(x) + {\Phi}(u_{i})(x) \}$
and $I_i = \Omega \setminus A_i$. 
By rewriting all of 
these PDEs in variational form
and substituting,
it follows
that \eqref{eq:NewtonSysReformulated} 
can be recast as:
\begin{align}
&\text{Find }
(\delta u_N, \xi, \mu)
\in 
H_0^1(\Omega) \times H^1(\Omega) \times H_0^1(I_i) \text{ such that } &&
\notag
\\
&
(\delta u_N - \mu - \varphi \xi, v )_{L^2(\Omega)} -  ( u_B - u_i, v )_{L^2(\Omega)} = 0
&&\forall v \in H^1_0(\Omega),
\label{ex:3a}
\\
&( \nabla \xi, \nabla \zeta )_{L^2(\Omega)} + ( k \xi + G_g(\Psi_0 + \psi T_{i} - u_i) (\delta u_N - \psi \xi), \zeta)_{L^2(\Omega)}  =0 &&\forall \zeta \in H^1(\Omega),
\label{ex:3b}
\\
& (\nabla \mu +  \nabla(\varphi \xi), \nabla q)_{L^2(\Omega)} = 0 
&&\forall q \in H^1_0(I_i),
\label{ex:3c}
\end{align}
where $T_i \in H^1(\Omega)$ denotes the 
solution of \eqref{eq:T_PDE} with $u = u_i$.
Note that the system \labelcref{ex:3a}--\labelcref{ex:3c} is linear in $\delta u_N$, $\xi$, and $\mu$ and, therefore, 
reduces to a linear system solve
after discretization. The act of encoding the inactive set in \labelcref{ex:3c} in the discretized linear system is discretization dependent. How this is achieved for a piecewise (bi)linear finite element discretization is discussed at the end of this subsection.
\medskip

\noindent \textbf{Comparison methods:} In our numerical 
experiments, we compare \cref{algo:semiNewtonAbstractGeneral,algo:semiNewtonAbstract} with
three alternative approaches for the numerical 
solution of \eqref{eq:ModelQVINum}, namely:
\begin{enumerate}[label=({C}\arabic*),leftmargin=1.1cm]
\item\label{item:C1} a pure fixed-point method;\label{com:fixedpoint}
\item\label{item:C2} a Newton method applied to a smoothed
Moreau--Yosida regularization
of the QVI \labelcref{eq:ModelQVINum} with fixed $\rho$; \label{com:smoothedMY}
\item\label{item:C4} \cref{algo:semiNewtonAbstractGeneral} but with a backtracking Armijo linesearch applied to each Newton update. We use the backtracking Armijo linesearch as described in \cite[\S 3.5]{nocedal2000} and implemented in the \texttt{LineSearches} package \cite{linesearches.jl} with the merit function $\|R(u_i)\|_{H^1(\Omega)}$.
\end{enumerate}
The iterates of the 
fixed-point method are defined 
by $u_{i+1} = S(\Phi(u_i))$ for $i=1,2,\ldots$, where $S(\Phi(u_i))$ is computed as described above. If the map $S\circ \Phi$ is contractive, then
this type of algorithm 
can be expected to converge linearly 
to the QVI-solution. 

For a given $\rho>0$, a smoothed Moreau--Yosida regularization of the QVI \labelcref{eq:ModelQVINum} 
corresponds to replacing 
the 
solution operator $S$
of the obstacle problem \eqref{eq:upper_obst_prob}
by 
the solution map $S_\rho$
of the 
mollified problem \eqref{eq:MoreauYosida}
wherever it appears. 
For the inner solver,
this means that 
 $u_B = S(\Phi(u_i))$ 
 is approximately calculated by 
 replacing
 \eqref{eq:upper_obst_prob}
 with its 
 Moreau--Yosida 
 regularization \eqref{eq:MoreauYosida}
 and by subsequently applying 
 Newton's method.
 In \eqref{eq:NewtonSysReformulated}, we then approximate 
 the update formula $u_N := u_i + \delta u_N$
 via $u_{N, \rho} = u_i + \delta u_{N,\rho}$, where
 $\delta u_{N, \rho}$ is obtained by 
 solving the system
\begin{align}
&\text{Find }
(\delta u_{N,\rho}, \xi, w)
\in 
H_0^1(\Omega) \times H^1(\Omega) \times H_0^1(\Omega) \text{ such that } &&
\notag
\\
&(\delta u_{N,\rho} - w, v )_{L^2(\Omega)} -  ( u_B - u_i, v )_{L^2(\Omega)} = 0
\qquad\qquad\quad
&&\forall v \in H^1_0(\Omega),
\label{ex:4a}
\\
&( \nabla \xi, \nabla \zeta )_{L^2(\Omega)} + ( k \xi + G_g(\Psi_0 + \psi T_{i} - u_i) (\delta u_{N,\rho} - \psi \xi), \zeta)_{L^2(\Omega)}  =0 &&\forall \zeta \in H^1(\Omega),
\label{ex:4b}
\\
&( \nabla w , \nabla q )_{L^2(\Omega)} + \rho^{-1}\left( G_{\sigma_\rho}(u_i-\Phi_0 - \varphi T_i)(w-\varphi \xi), q \right)_{L^2(\Omega)} = 0
&&\forall q \in H^1_0(\Omega).
\label{ex:4c}
\end{align}
Here,  $T_i \in H^1(\Omega)$ again denotes the 
solution of \eqref{eq:T_PDE} with $u = u_i$
and $G_{\sigma_\rho}\colon \R \to \R$ the 
derivative of the function $\sigma_\rho$ in \eqref{eq:Huber}.
The key difference between \labelcref{ex:4a}--\labelcref{ex:4c} and \labelcref{ex:3a}--\labelcref{ex:3c} is that \labelcref{ex:4a}--\labelcref{ex:4c} does not contain an active set. In particular,
\labelcref{com:smoothedMY} does not require the semismoothness results that we derived in this paper for the map $S\circ \Phi$. However, we shall see that this regularization is unfavourable. In particular, the approximated QVI-solution is dependent on the Moreau--Yosida parameter $\rho$ and is typically infeasible. Moreover, the convergence of the algorithm is slower than when computing $\delta u_N$ via \labelcref{ex:3a}--\labelcref{ex:3c} directly, and the convergence rate degrades in the limit  $\rho \to 0$ due to ill-conditioning; see 
\cref{fig:tf-solutions}(d) in
\cref{subsec:7.4}.

\medskip

\noindent \textbf{Finite element discretization:}
To discretize the variational problems, 
we consider uniform subdivisions 
of the domain $\Omega$ into intervals and  quadrilateral cells,
respectively, depending on 
whether the dimension is one or two. In one dimension, we discretize with a continuous piecewise linear finite element $\mathcal{P}_1$ and in two dimensions we choose the tensor-product of the one-dimensional basis $\mathcal{P}_1\times \mathcal{P}_1$ and define the mesh size $h$ as the length of the edge of a cell. 
This discretization applies to
$u$, 
$T$, and the
auxiliary functions $\xi$, $\eta$, $\mu$, and $w$. 
\medskip

\noindent \textbf{Discretization of the active set:} The choice of a $\mathcal{P}_1$-,
respectively, $\mathcal{P}_1\times \mathcal{P}_1$-discretization means that the active sets in the PDAS-algorithm for the obstacle problem as well as 
the sets
$A_i$ and $I_i$ in \labelcref{ex:3c} may be found by examining the coefficient vectors of the involved finite element functions
with respect to the nodal basis. 
In particular, enforcing \labelcref{ex:3c} for all $q \in H_0^1(I_i)$ simplifies to deleting the rows and columns in the corresponding finite element matrices and the 
rows in the right-hand side vector that are 
associated with nodal basis functions that 
belong to active nodes. More explicitly, let ${\bf u}_B$, 
$\boldsymbol{\xi}$, $\boldsymbol{\mu}$, $\boldsymbol{{\bf \delta u}}_N$,  ${\bf r}$, and ${\bf \bar{\Theta}} \in \mathbb{R}^M$ denote the coefficient vectors of the finite element functions $u_{B,h}$, $\xi_h$, $\mu_h$, $\delta u_{N,h}$, $u_h - u_{B,h}$, and $\Phi_{0,h} + \varphi_h T_h$, respectively,
that correspond to  the quantities in the system 
\labelcref{ex:3a}--\labelcref{ex:3c}.
Let $\mathfrak{N}_h = \{1,2,\ldots,M\}$ 
be the index set of the set of the degrees of freedom
and denote the discrete active set by 
$\mathfrak{A}_h = \{ i \in \mathfrak{N}_h \mid 
{\bf u}_{B,i} = {\bf \bar{\Theta}}_i \}$. 
Denote further the discrete 
inactive set by  $\mathfrak{I}_h = \mathfrak{N}_h  \setminus 
\mathfrak{A}_h$ and suppose that the finite element linear system (before removing the active set) induced by \labelcref{ex:3a}--\labelcref{ex:3c} is given by 
\begin{align}
\begin{pmatrix}
K_{11} & K_{12} & K_{13}\\
K_{21} & K_{22} & {\bf 0}\\
{\bf 0}& K_{32} & K_{33}\\
\end{pmatrix}
\begin{pmatrix}
\boldsymbol{{\bf \delta u}}_N\\
\boldsymbol{\mathbf{\xi}}\\
{\boldsymbol{\mu}}
\end{pmatrix}
=
\begin{pmatrix}
-{\bf r}\\
{\bf 0}\\
{\bf 0}\end{pmatrix},
\label{ex:5}
\end{align}
where $K_{ij} \in \mathbb{R}^{M \times M}$ for $i,j \in \{1,2,3\}$. Then 
modifying \eqref{ex:5} 
such that it corresponds to a discretization of 
\labelcref{ex:3a}--\labelcref{ex:3c} with $\mu \in H_0^1(I_i)$ and 
test functions $q \in H_0^1(I_i)$
is equivalent to deleting the rows and columns in $K_{33}$, the columns in $K_{13}$, and the rows in $K_{32}$ and 
${\boldsymbol{\mu}}$ whenever the row or column index is an element of $\mathfrak{A}_h$. In other words, \labelcref{ex:5} gets reduced to
\begin{align*}
\begin{pmatrix}
K_{11} & K_{12} & [K_{13}]_{\mathfrak{N}_h,\mathfrak{I}_h} \\
K_{21} & K_{22} & {\bf 0}_{\mathfrak{N}_h,\mathfrak{I}_h} \\
{\bf 0}_{\mathfrak{I}_h,\mathfrak{N}_h} & [K_{32}]_{\mathfrak{I}_h,\mathfrak{N}_h}  & [K_{33}]_{\mathfrak{I}_h,\mathfrak{I}_h} \\
\end{pmatrix}
\begin{pmatrix}
\boldsymbol{{\bf \delta u}}_N\\
\boldsymbol{\mathbf{\xi}}\\
{\boldsymbol{\mu}}_{\mathfrak{I}_h} 
\end{pmatrix}
=
\begin{pmatrix}
-{\bf r}\\
{\bf 0}\\
{\bf 0}_{\mathfrak{I}_h} \end{pmatrix},
\qquad {\boldsymbol{\mu}}_{\mathfrak{A}_h} = {\bf 0}.
\end{align*}
Note that 
the realization of  \labelcref{ex:3c} for higher-order discretizations is a far more delicate topic. We leave the study of such higher-order finite elements for future research.

\subsection{Test 1: a one-dimensional QVI with a known solution}
\label{subsec:7.2}

We are now
in position to present our numerical experiments.
We begin with a simple one-dimensional instance 
of the QVI \eqref{eq:ModelQVINum} which is covered by 
the global framework of \cref{sec:framework_applied_to_SPhi} and 
possesses a unique solution 
that is known in closed form. 
We choose the quantities in the QVI \eqref{eq:ModelQVINum} 
as follows:
\begin{equation}
\label{eq:ex_1_setup}
\begin{gathered}
\Omega  = (0,1),
\qquad
\Phi_0(x) = \max(0, |x - 0.5| - 0.25),
\\
f(x) = \pi^2\sin(\pi x) + 100\max(0, - |x - 0.625| + 0.125),
\qquad
\varphi(x) = \frac{1}{\alpha_1} \sin(\pi x),
\\
k=1,
\qquad
g(s) = 
\alpha_1
+ 
\arctan\left (\frac{1}{\alpha_2} \min\left (\frac{1-s}{2}, 1-s \right )\right ),
\qquad
\Psi_0 \equiv 1,
\qquad 
\psi \equiv \varphi.
\end{gathered}
\end{equation}
Here, $\alpha_1, \alpha_2 > 0$ are parameters that 
will be fixed later. 
Note that the conditions in 
\cref{ass:standing_obstacle_map,ass:SemilinearObstacleMap} 
are trivially satisfied in the situation of 
\eqref{eq:ex_1_setup} (with $p=2$).
Using direct calculations and 
\cref{rem:Poincare},
it is furthermore easy to establish the following lemma.

\begin{lemma}
\label{lem:ex1}
In the situation of \eqref{eq:ex_1_setup},
the QVI \eqref{eq:ModelQVINum} 
possesses the solution
$\bar u(x) = \sin(\pi x)$, $\bar T \equiv \alpha_1$. 
For the solution $(\bar u, \bar T)$, 
the inactive set, the strictly active set, 
and the biactive set are given by 
\begin{gather*}
\{x \in \Omega \mid \bar u(x) < \Phi_0(x) +\Phi(\bar u)(x)\} =
(0, 0.25) \cup (0.75, 1),
\\
\{x \in \Omega \mid \bar u(x) = \Phi_0(x) + \Phi(\bar u)(x), 
-\bar u''(x) - f(x) \neq 0\}
=
(0.5, 0.75),
\end{gather*}
and
\[
\{x \in \Omega \mid \bar u(x) = \Phi_0(x) +\Phi(\bar u)(x), 
-\bar u''(x) - f(x) = 0\}
= [0.25, 0.5] \cup \{0.75\},
\]
respectively. Here, the function evaluations 
are defined w.r.t.\ the continuous 
representatives of $\bar u$
$\Phi_0$, $\Phi(\bar u)$,
and $-\bar u'' - f$.
The function
$\Psi_0 + \psi \bar T -\bar u$ takes values only in the set 
of points of nondifferentiability of $g$.
Further, 
the constant 
appearing in  
\eqref{eq:Phi_Lip_H1}
and
\eqref{eq:G_phi_bound}
satisfies
\begin{equation}
\label{eq:randomestimate-373he}
C_P(\Omega)\,
\Lip{g} 
\left ( \|\varphi \|_{L^\infty(\Omega)}  k^{-1/2} 
+
\| \varphi' \|_{L^\infty(\Omega)} k^{-1}
\right ) 
=
\frac{1 + \pi}{\pi \alpha_1 \alpha_2}.
\end{equation}
\end{lemma}
Recall that the solution 
operator
of the obstacle problem 
is Gâteaux differentiable if and only if 
strict complementarity holds;
see \cite[Lemma~2.6]{RaulsWachsmuth}.
In combination with 
the fact that
the function
$\Psi_0 + \psi \bar T - \bar u$
takes values only in the set 
of nondifferentiable points of $g$, this means 
that the solution $(\bar u, \bar T)$
corresponds to a worst-case example.
From \eqref{eq:randomestimate-373he},
\cref{cor:gamma_prop},
and the global Lipschitz continuity of
$g$,
it follows further that the mapping 
$\Phi\colon H_0^1(\Omega) \to H_0^1(\Omega)$
appearing in \eqref{eq:ModelQVINum} 
in the situation of 
\eqref{eq:ex_1_setup}
satisfies the 
conditions  
\ref{th:obst_QVI:i} to \ref{th:obst_QVI:iv}
in \cref{th:obst_QVI} 
whenever $\alpha_1, \alpha_2 > 0$ are chosen 
such that 
$1+\pi^{-1} < \alpha_1 \alpha_2$ holds
(with $p=2$,
the Newton derivative 
$G_\Phi$
defined in \cref{th_T_probs}\ref{th_T_probs:iii},
and $\gamma := (1 + \pi)(\pi\alpha_1 \alpha_2)^{-1}$). 
In particular, the QVI is uniquely 
solvable for $1+\pi^{-1}  < \alpha_1 \alpha_2$,
$(\bar u, \bar T)$ is its unique solution, 
and the assertions of \cref{th:convergence}
hold when we apply \cref{algo:semiNewtonAbstract}.

The results that are obtained when 
our globalized semismooth Newton method 
is applied to the QVI \eqref{eq:ModelQVINum}
in the situation of \eqref{eq:ex_1_setup}
(or, more precisely, to the operator 
equation $u=S(\Phi(u))$ that the QVI may be recast as)
can be seen in \cref{fig:test1-convergence}.

\begin{figure}[ht!]
\centering 
\subfigure[\Cref{algo:semiNewtonAbstract} and \labelcref{item:C1}
in test configuration (I)]{\includegraphics[width =0.49 \textwidth]{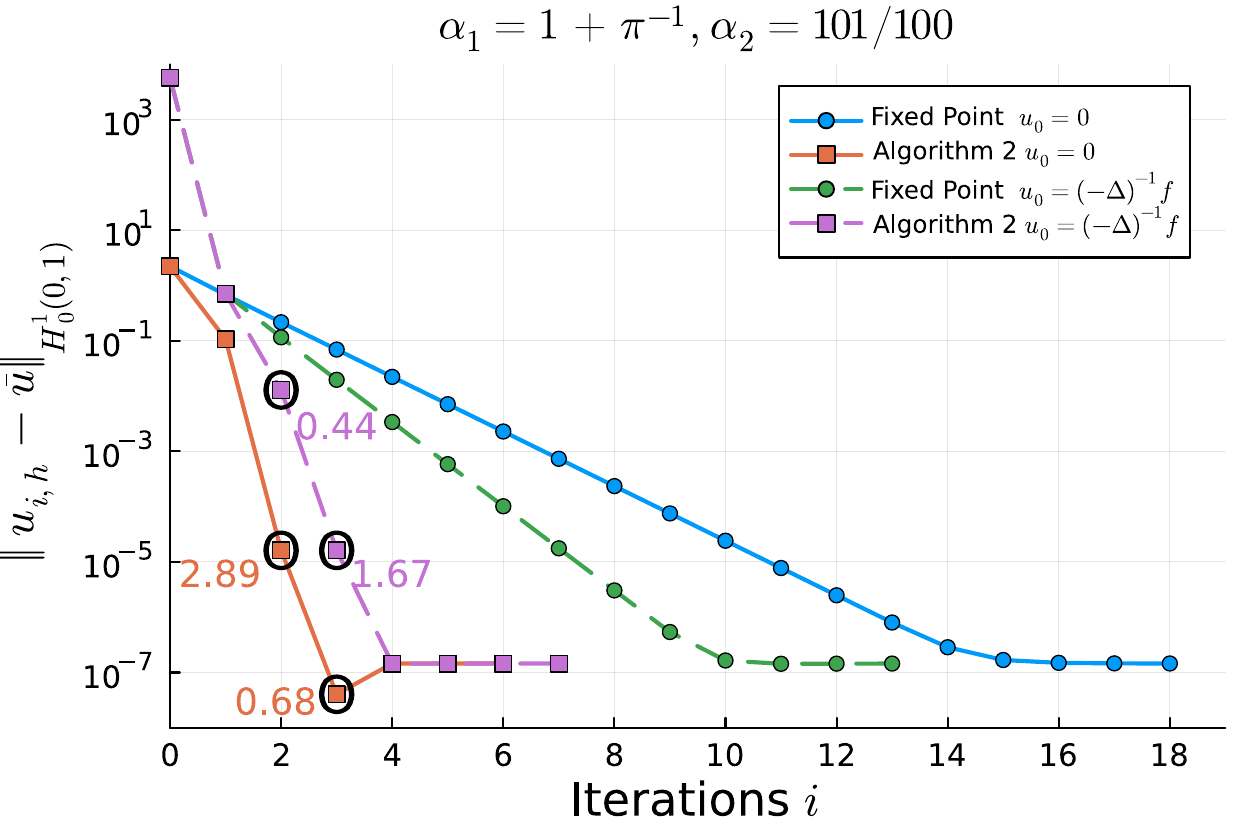}}
\subfigure[\Cref{algo:semiNewtonAbstract} and \labelcref{item:C1}
in test configuration (II)]{\includegraphics[width =0.49 \textwidth]{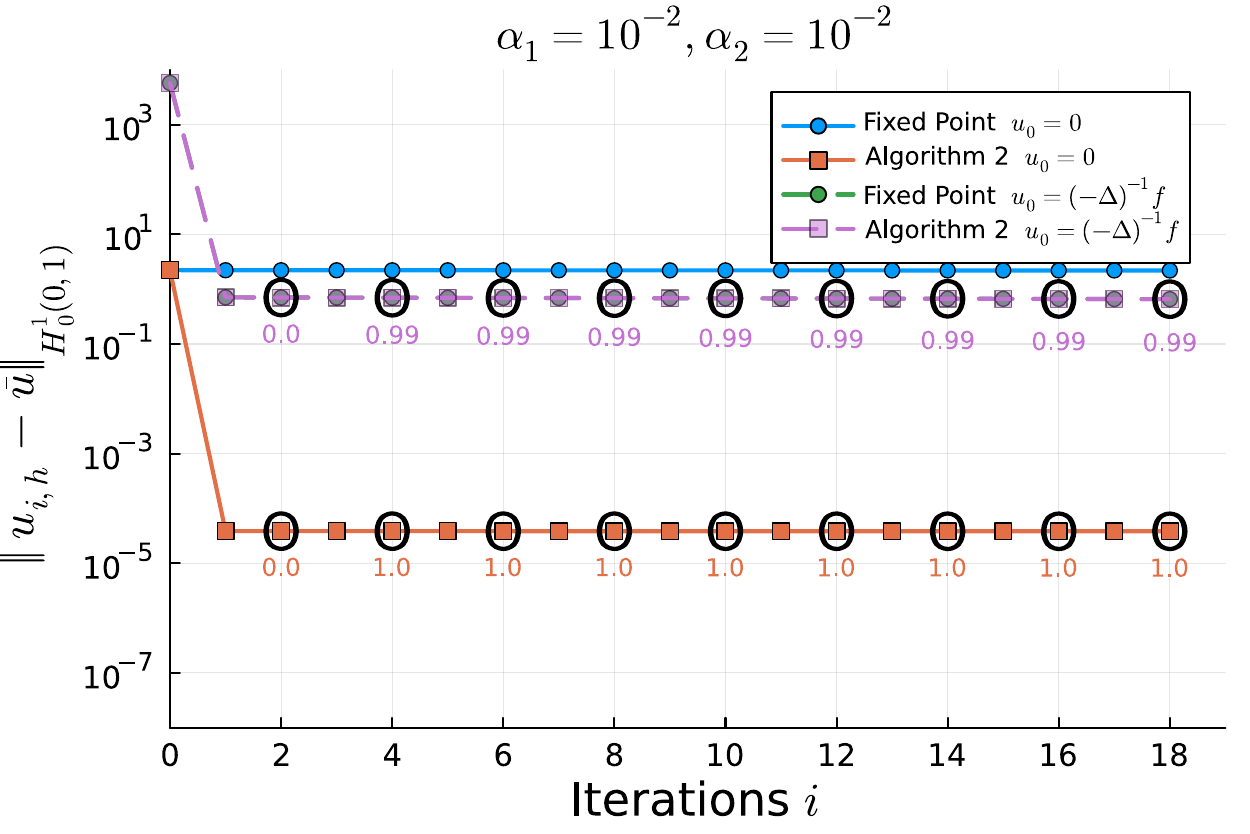}}
\\[0.4cm]
\subfigure[Vanilla \& backtracking SSN iterations 
in test case (II)]{\includegraphics[width =0.49 \textwidth]{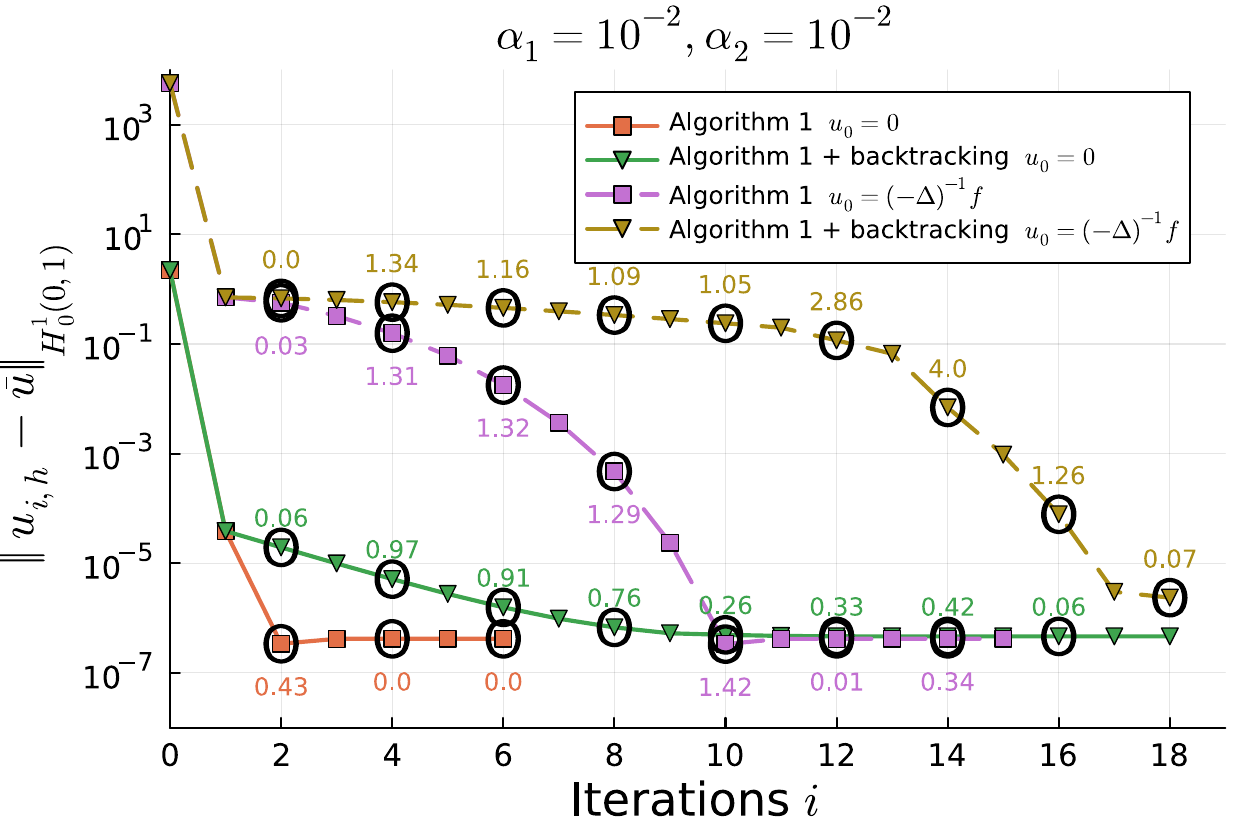}}
\subfigure[The solution $\bar u$
and its mould 
$\Phi_0 + \varphi \bar T$.]{\includegraphics[width =0.49 \textwidth]{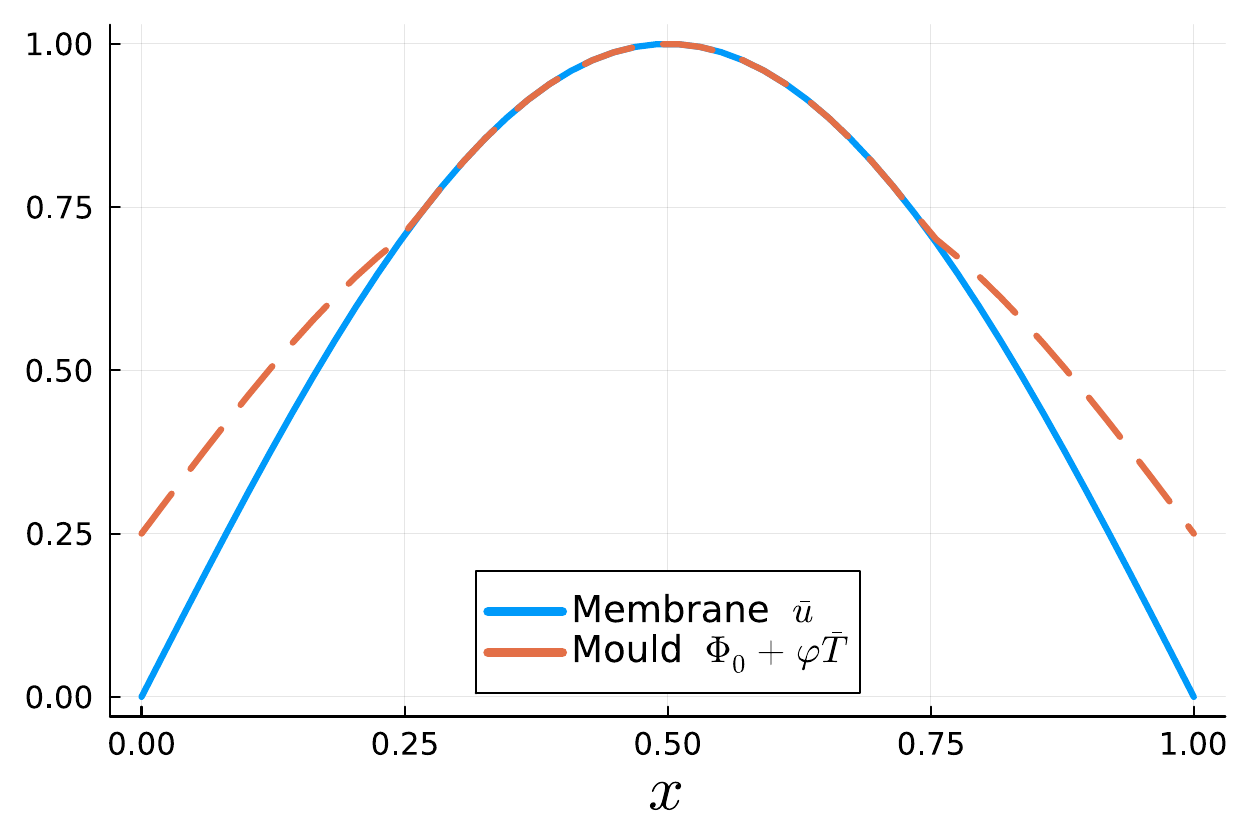}}
\caption{(Test 1) Convergence behavior and results
of the fixed-point method \ref{item:C1} and 
\cref{algo:semiNewtonAbstract}
for the parameter choices (I)
and (II) as well as \cref{algo:semiNewtonAbstractGeneral}
and \ref{item:C4} (\cref{algo:semiNewtonAbstractGeneral} with a backtracking linesearch),  for the parameter choice (II) with the initial guesses
$u_0 = 0$ and $u_0 = (-\Delta)^{-1} f$
in the situation of \eqref{eq:ex_1_setup}. 
The mesh size was chosen as 
$h= 5 \times 10^{-4}$. The numbers next to the graphs are the EOCs in \eqref{eq:EOC}.}
\label{fig:test1-convergence}
\end{figure}

Here, we 
considered two different parameter choices:
(I) 
$(\alpha_1, \alpha_2) = (1+\pi^{-1}, 101/100)$ and 
(II)
$(\alpha_1, \alpha_2) = (10^{-2}, 10^{-2})$.
Note that, for choice (I), we trivially 
have $1+\pi^{-1} < \alpha_1 \alpha_2$, 
so that 
\cref{th:convergence}
is applicable and global $q$-superlinear convergence to the problem solution $\bar u$ is guaranteed. 
For (II), 
we have $1+\pi^{-1} \gg \alpha_1 \alpha_2 = 10^{-4}$. We use this second 
configuration to investigate
how our algorithm behaves 
in situations that are
beyond the scope of our analysis. 
In addition, we also 
considered two different choices
for the initial guess $u_0$,
namely, $u_0 = 0$ and $u_0 = (-\Delta)^{-1}f$.
\Cref{fig:test1-convergence}(a)
shows the behavior 
of \cref{algo:semiNewtonAbstract}
and the pure fixed-point method 
\ref{item:C1}
for the two different starting values 
in configuration (I).
It can be seen that both algorithms 
are able to identify the solution 
$\bar u$ of the QVI up to 
the discretization error 
for both choices of $u_0$. 
The semismooth Newton method requires 
far fewer iterations, however, and exhibits 
$q$-superlinear convergence speed, 
as predicted by \cref{th:convergence}. 
This is also confirmed by the experimental orders of 
convergence (EOCs) given by the formula 
\begin{equation}
\label{eq:EOC}
\textup{EOC}_i := 
\log\left( \frac{\norm{ u_{i}-\bar u }{H_0^1(\Omega)}}{\norm{ u_{i-1}-\bar u }{H_0^1(\Omega)}} \right) 
\log\left( \frac{\norm{ u_{i-1}-\bar u }{H_0^1(\Omega)}}{\norm{ u_{i-2}-\bar u }{H_0^1(\Omega)}} \right)^{-1},
\qquad 
i \geq 2.
\end{equation}
We remark that, in this test case,
\cref{algo:semiNewtonAbstract} always chooses the iterate 
$u_N$ in step \ref{algo1:decision}.

\Cref{fig:test1-convergence}(b) shows the 
convergence behavior of the two algorithms 
for the parameter choice (II). It can be seen here that 
both \cref{algo:semiNewtonAbstract} and 
the fixed-point method stagnate/converge very slowly. 
For this configuration, \cref{algo:semiNewtonAbstract}
always chooses $x_B$ in step \ref{algo1:decision} for $i\geq 2$, 
and the lack of contractivity of the composition 
$S\circ \Phi$ results in a very poor convergence 
performance. \Cref{fig:test1-convergence}(c) shows 
what happens in (II) with \cref{algo:semiNewtonAbstractGeneral} with and without a backtracking linesearch \labelcref{item:C4}.
As can be seen, both methods are able to 
identify $\bar u$ with superlinear convergence speed 
even in the absence of contractivity. Notably, \cref{algo:semiNewtonAbstractGeneral} without a linesearch achieves convergence with the fewest number of iterations and with a significantly smaller computational cost since the backtracking linesearch requires many evaluations of $S$. This highlights that
it is worth to consider our semismooth Newton approach
for the solution of QVIs 
even in those situations where 
a rigorous convergence analysis is not possible.

\subsection{Test 2: a one-dimensional QVI with two known solutions}
\label{subsec:7.3}
Next, we consider the 
QVI \eqref{eq:ModelQVINum} 
in a situation in which there are several solutions 
and in which the localized setting of \cref{sec:framework_fp_H_multiple_solns}
has to be employed. 
We choose the data in \eqref{eq:ModelQVINum}  as follows:
\begin{equation}
\label{eq:ex_2_setup}
\begin{gathered}
\Omega = (0, 1),
\qquad
\Phi_0 \equiv 0,
\qquad
f(x) = \alpha_1 \pi^2 \sin(\pi x),
\qquad
\varphi(x) =  
\alpha_2\frac{10 \pi^2\sin(\pi x)}{5 - \cos(2 \pi x)},
\\
k = \pi^2,
\qquad
g(s)
=
\frac{4}{\alpha_1} \min(0,s)^2,
\qquad
\Psi_0 \equiv 0,
\qquad
\psi(x) =   \frac{5 \pi^2\sin(\pi x)}{5 - \cos(2 \pi x)}.
\end{gathered}
\end{equation}
Here, $\alpha_1 > 0$ and $\alpha_2 \geq 1$ 
are again parameters 
that can be chosen arbitrarily. 
Note that the conditions in 
\cref{ass:standing_obstacle_map,ass:SemilinearObstacleMap} 
are all
satisfied in the situation of \eqref{eq:ex_2_setup}
(with $p=2$)
and that, in contrast to 
the example in \cref{subsec:7.2}, 
the function $g$ in \eqref{eq:ex_2_setup}
is only locally 
Lipschitz continuous (but $C^1$). 
Similarly to 
\cref{lem:ex1}, we obtain:

\begin{lemma}
\label{lem:ex2}
In the situation of \eqref{eq:ex_2_setup},
the QVI \eqref{eq:ModelQVINum} 
possesses (at least)
the two solutions 
\[\bar u_1 \equiv 0, \qquad \bar T_1\equiv 0\]
and 
\[
\bar u_2(x)
:= 
\alpha_1 \sin(\pi x),
\qquad
\bar T_2(x)
:=
\alpha_1 
\frac{5 - \cos(2 \pi x)}{10 \pi^2}.
\]
If $M_R$
is defined as in \cref{th_T_probs_1d_local},
then it holds
\begin{equation}
\label{eq:MRformula}
M_R = \frac{1}{2}R + \frac{10(1+\pi)}{3\pi\alpha_1}R^2
\end{equation}
and 
\begin{equation}
\label{eq:loc_lip_estimate}
\begin{aligned}
&\mathrm{Lip}(g,[-M_R , M_R])
\left ( \|\varphi \|_{L^\infty(\Omega)}  k^{-1/2} 
+
\| \varphi' \|_{L^\infty(\Omega)} k^{-1}
\right )\frac{|\Omega|}{\pi}
\\
&\quad=
\lim_{t \searrow M_R}
\mathrm{Lip}(g,[-t , t])
\left ( \|\varphi \|_{L^\infty(\Omega)}  k^{-1/2} 
+
\| \varphi' \|_{L^\infty(\Omega)} k^{-1}
\right )\frac{|\Omega|}{\pi}
\\
&\quad= \frac{50\alpha_2}{3\alpha_1} \left (
R + \frac{20(1+\pi)}{3\pi\alpha_1}R^2
\right ).
\end{aligned}
\end{equation}
Furthermore, the right-hand side of 
\eqref{eq:loc_lip_estimate} 
is an element of the interval $[0,1)$ if and only if
\begin{equation}
\label{eq:Rbound_locLip}
R < 
R_{\alpha_1,\alpha_2} :=
\frac{3\alpha_1}{10(1+\pi)}
\left (\sqrt{\frac{(5\alpha_2 + 8)\pi^2 + 8\pi}{80\alpha_2}} - \frac{\pi}{4}\right )
.
\end{equation}

\end{lemma}
\begin{proof}
To see that $(\bar u_1, \bar T_1)$ solves 
\eqref{eq:ModelQVINum} 
in the situation of 
\eqref{eq:ex_2_setup},
it suffices to plug everything 
in \eqref{eq:ex_2_setup} 
and the formulas
$\bar u_1 \equiv 0$, $\bar T_1\equiv 0$
into \eqref{eq:ModelQVINum}.
That $(\bar u_2, \bar T_2)$ is another
solution is proved analogously.
Note that, to see that the differential 
equation in \eqref{eq:ModelQVINum} is 
satisfied for $(\bar u_2, \bar T_2)$,
one has to employ the double-angle-formula 
for the cosine, i.e., 
\begin{equation*}
\begin{aligned}
k \bar T_2- \bar T_2''
=
\pi^2
\left (
\alpha_1 
\frac{5 - \cos(2 \pi x)}{10 \pi^2}
\right )
-
\frac{4\pi^2 \alpha_1 \cos(2 \pi x) }{10 \pi^2} 
&=
\frac{\alpha_1 }{2}(1 - \cos(2 \pi x))
= \alpha_1 \sin(\pi x)^2
\\
&= g\left (  - \frac{\alpha_1}{2} \sin(\pi x)\right )
=
g(\Psi_0 + \psi \bar T_2 -\bar u_2).
\end{aligned}
\end{equation*}
To check that $M_R$ is given by 
\eqref{eq:MRformula} in the situation of 
\eqref{eq:ex_2_setup}, 
it suffices to note that 
$|\Omega|^{1/2} =1$ holds
for $\Omega = (0,1)$, to calculate that 
$\|\psi\|_{L^\infty(\Omega)} = 5\pi^2/6$,
to note that the function $g$ in \eqref{eq:ex_2_setup}
satisfies 
$\mathrm{Lip}(g,[-t, t]) = 8t/\alpha_1$
for all $t > 0$,
and to plug into \eqref{eq:MRdef}.
To obtain \eqref{eq:loc_lip_estimate},
one proceeds analogously, using that 
$\|\varphi\|_{L^\infty(\Omega)} = 5\alpha_2\pi^2/3$
and
$\|\varphi'\|_{L^\infty(\Omega)} = 5\alpha_2\pi^3/2$.
The estimate \eqref{eq:Rbound_locLip} 
finally follows from an application of the 
pq-formula.
\end{proof}

That the QVI \eqref{eq:ModelQVINum} 
possesses two solutions in the situation of \eqref{eq:ex_2_setup}
shows that this test case is beyond the scope 
of the analysis of \cref{subsec:global_contraction} and that 
problems of the type \eqref{eq:ModelQVINum} 
may indeed possess several 
solutions if the conditions 
in \eqref{eq:Phi_gamma_Lip} and \eqref{eq:GPhi_bound-2}
are violated.
To 
see that we may apply the results of 
\cref{sec:framework_fp_H_multiple_solns}, 
we first note that the 
spaces $X=H_0^1(\Omega)$,
$Y_2 = \{ v \in H_0^1(\Omega) \mid v'' \in L^2(\Omega)\}
= H_0^1(\Omega)\cap H^2(\Omega)$,
the set $D=Y_2$,
the 
solution operator 
$S\colon H_0^1(\Omega) \to H_0^1(\Omega)$
of \eqref{eq:upper_obst_prob},  
and the map $\Phi\colon H_0^1(\Omega) \to Y_2$
in \eqref{eq:ModelQVINum}
satisfy the conditions 
in 
points \ref{ass:standing_general_proj:i} and
\ref{ass:standing_general_proj:ii}
of \cref{cor:composite_local}
in the situation of 
\eqref{eq:ModelQVINum} 
by 
\cref{lem:S_obs_Newton,lem_GS_bound_obstacle,th_T_probs_1d_local}.
It remains to find a set $B$ with the properties 
in points \ref{ass:standing_general_proj:iv} and \ref{ass:standing_general_proj:v} of \cref{cor:composite_local}. 
To this end, let us suppose that 
$B$ is a closed ball 
in $H_0^1(\Omega)$ (not necessarily centered 
at zero) that is contained 
in $\{v \in H_0^1(\Omega) \mid 
\norm{v}{H_0^1(\Omega)} < R_{\alpha_1,\alpha_2} \}$
with $R_{\alpha_1,\alpha_2}$ defined 
by the formula on the right-hand side of   \eqref{eq:Rbound_locLip}.
From \cref{lem:semismooth_balls}, 
it follows that the projection 
$P_B\colon H_0^1(\Omega) \to B$ in 
$H_0^1(\Omega)$ onto $B$ satisfies 
all of the conditions in 
point \ref{ass:standing_general_proj:v} of \cref{cor:composite_local}.
From \cref{th_T_probs_1d_local},
the properties of $R_{\alpha_1,\alpha_2}$,
and \cref{lem:S_obs_well_def,lem_GS_bound_obstacle},
we further obtain that 
there exists $\gamma_B \in [0,1)$
such that 
\begin{equation}
\label{eq:ex:2:contra}
\| S(\Phi(u_1)) - S(\Phi(u_2))\|_{H_0^1(\Omega)}
\leq
\| \Phi(u_1) - \Phi(u_2)\|_{H_0^1(\Omega)}
\leq
\gamma_B \| u_1 - u_2 \|_{H_0^1(\Omega)}
\quad \forall u_1, u_2 \in B
\end{equation}
and
\[
\sup_{u \in B} \left \|G_S (\Phi(u))G_\Phi(u)\right\|_{\LL(H_0^1(\Omega),H_0^1(\Omega))}
\leq
\sup_{u \in B} \left \|
G_\Phi(u)\right\|_{\LL(H_0^1(\Omega),H_0^1(\Omega))}
\leq\gamma_B.
\]
Putting everything together, 
it follows that 
the conditions in 
\cref{cor:composite_local}
are all satisfied for $B$ 
in the situation of 
\eqref{eq:ex_2_setup}.
In particular, \cref{th:loc_conv} 
is applicable 
and we may employ 
\cref{algo:semiNewtonAbstract} 
to determine 
the intersection of the 
solution set 
$\{u \in H_0^1(\Omega) \mid u = S(\Phi(u))\}$
of \eqref{eq:ModelQVINum} 
with $B$
by solving the localized 
fixed-point equation 
\begin{equation}\label{eq:fp_problem_proj}
\text{Find } \hat x \in X 
\text{ such that } \hat x = S( \Phi_B(\hat x)),
\end{equation}
where $\Phi_B := \Phi \circ P_B$.
(This  is just \eqref{eq:fp_problem_proj_H} for $H=S\circ \Phi$.)
Observe that since we already know
that $\bar u_1 \equiv 0$ solves 
the QVI \eqref{eq:ModelQVINum},
we obtain from \cref{th:loc_conv} 
that our algorithm should 
converge $q$-superlinearly 
to $\bar u_1$ if $0 \in B$ and to 
a point $\hat u \not \in B$
if $0 \not\in B$ when applied 
to \eqref{eq:fp_problem_proj}.

To put these predictions to the test, 
we have implemented \cref{algo:semiNewtonAbstract} 
for the solution of \eqref{eq:fp_problem_proj}
in the situation of \eqref{eq:ex_2_setup}
for the case that $B$ is a closed ball $B_R(0)$
of radius $R>0$ in $H_0^1(\Omega)$ centered 
at the origin. As it turns out, 
for this configuration, there are 
two distinct regimes: (I) $\alpha_2 = 1$ and (II) $\alpha_2 > 1$. In regime (I), the biactive set of the 
second solution $(\bar u_2, \bar T_2)$ 
is the whole domain, 
i.e., it holds $\bar u_2 = (-\Delta)^{-1} f  = \Phi_0 + \varphi \bar T_2$ in $\Omega$. 
For this case, the solution $(\bar u_2, \bar T_2)$ 
is highly unstable with respect to perturbations of the data. 
Furthermore, the QVIs obtained from the discretization of 
\eqref{eq:ModelQVINum} apparently do not possess 
any discrete solutions that approximate $(\bar u_2, \bar T_2)$. The outcome is that, after a discretization, the second closed-form solution $(\bar u_2, \bar T_2)$ cannot be discovered by any method considered in this work, even if the interpolant $I_h \bar u_2$ of $\bar u_2$ is used as the initial guess
and the mesh width $h$ is very small; see \cref{fig:test2-I}.

\begin{figure}[ht!]
\centering 
\includegraphics[width =0.49 \textwidth]{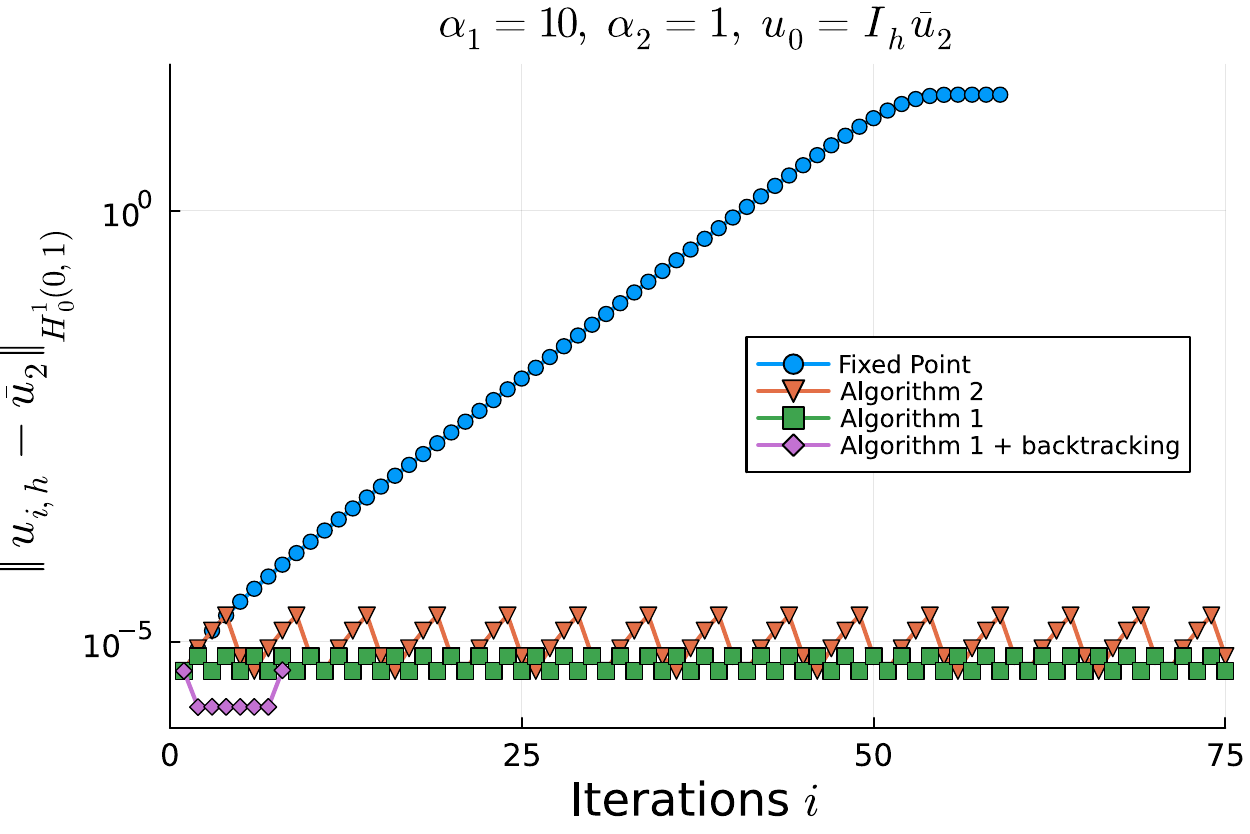}
\caption{(Test 2) Divergence behavior of the fixed-point method \labelcref{item:C1}, \cref{algo:semiNewtonAbstract}, and \cref{algo:semiNewtonAbstractGeneral} with and without a backtracking linesearch \labelcref{item:C4} in the situation of 
the QVI \eqref{eq:ModelQVINum} with setup \labelcref{eq:ex_2_setup} and 
 $(\alpha_1,\alpha_2) = (10, 1)$. No projection is used and the initial guess is
 (the Lagrange interpolant of) $\bar u_2$. The mesh size is 
 \smash{$h=5 \times 10^{-4}$}.
 It can be checked that 
 the pure fixed-point method actually converges to 
 $\bar u_1$ although it is 
 initialized as close to $\bar u_2$ as possible.
 \Cref{algo:semiNewtonAbstract} and \Cref{algo:semiNewtonAbstractGeneral} both with and without a backtracking linesearch
 stagnate/enter a cycle 
 and are unable to reduce the residue below 
 $10^{-6}$ in this test 
 case---a bound much larger than 
 observed in the other experiments;
 cf.\ \cref{fig:tf-solutions}.
 This indicates that, 
 on the discrete level, 
 there is no solution 
 corresponding to $\bar u_2$.
 }
\label{fig:test2-I}
\end{figure}

Note that this observation is 
also of independent interest as it shows that, 
for $\alpha_2 = 1$ and the data in \eqref{eq:ex_2_setup},
the QVI \eqref{eq:ModelQVINum} is ill posed; 
cf.\ \cite{AHR,WachsmuthQVIs,Alphonse2020,AHRW,Alphonse2022-2,Alphonse2022}.
In the second regime (II), the solution $(\bar u_2, \bar T_2)$ satisfies 
$(-\Delta)^{-1} f = \bar u_2   < \Phi_0 + \varphi \bar T_2$ in $(0,1)$ and
the active set of $\bar u_2$ is empty.
Here, 
the numerical identification 
of $\bar u_2$ works without problems; see below.

For our numerical experiments in regime (II), 
we considered the parameters 
$(\alpha_1,\alpha_2) = (10, 3/2)$
and the initial guess $u_0(x) = 100(1-x)x$. 
With this choice of $\alpha_1$ and $\alpha_2$, we have $R_{\alpha_1,\alpha_2} \approx 0.3136$ and $\| \bar u_2\|_{H^1_0(0,1)} = \sqrt{50}\pi \approx 22.21$.
In particular,
the $q$-superlinear convergence 
of \cref{algo:semiNewtonAbstract}
to $\bar u_1$ is guaranteed by our analysis for all $R<0.3136$. 
Note that,
due to \eqref{eq:ex:2:contra}, 
we also immediately obtain that the
fixed-point method 
\ref{item:C1} converges to $\bar u_1$
for all $R<0.3136$
when applied to  \eqref{eq:fp_problem_proj}.

When running \cref{algo:semiNewtonAbstract}
and the fixed-point method \ref{item:C1}
in this configuration, one observes that 
both algorithms converge in one iteration 
to $(\bar u_2, \bar T_2)$ for all 
$R \geq \sqrt{50}\pi$. This effect is caused 
by the inactivity of $\bar u_2$ in regime (II). 
The behavior for $1/100 \leq R < \sqrt{50}\pi$ is tabulated in \cref{tab:test2-convergence}.
\vspace{0.1cm}

\begin{table}[ht!]
\small
\centering
\begin{tabular}{c |c |c |c |c |c |c |c |c} 
\hline\noalign{\smallskip}
$R$& 0.01& 2.477& 4.944& 7.411& 9.879& 12.346& 14.813& 17.28\\
\noalign{\smallskip}
\hline\noalign{\smallskip}\begin{tabular}{@{}c@{}} \ref{item:C1} \end{tabular}  
& \begin{tabular}{@{}c@{}} 4 \\ (1.91)  \end{tabular}   
& \begin{tabular}{@{}c@{}} 7 \\ (2.01)  \end{tabular}   
& \begin{tabular}{@{}c@{}} 9 \\ (2.0)  \end{tabular}   
& \begin{tabular}{@{}c@{}} 12 \\ (2.01)  \end{tabular}   
& \begin{tabular}{@{}c@{}} -  \end{tabular}   
& \begin{tabular}{@{}c@{}} -  \end{tabular}   
& \begin{tabular}{@{}c@{}} -  \end{tabular}   
& \begin{tabular}{@{}c@{}} -  \end{tabular}
\\
\noalign{\smallskip}
\hline\noalign{\smallskip}\begin{tabular}{@{}c@{}} \cref{algo:semiNewtonAbstract} \end{tabular}  
& \begin{tabular}{@{}c@{}} 2 \\ (1.56)  \end{tabular}   
& \begin{tabular}{@{}c@{}} 6 \\ (2.01)  \end{tabular}   
& \begin{tabular}{@{}c@{}} 8 \\ (2.0)  \end{tabular}   
& \begin{tabular}{@{}c@{}} 11 \\ (2.01)  \end{tabular}   
& \begin{tabular}{@{}c@{}} -  \end{tabular}   
& \begin{tabular}{@{}c@{}} -  \end{tabular}   
& \begin{tabular}{@{}c@{}} -  \end{tabular}   
& \begin{tabular}{@{}c@{}} -  \end{tabular}
\\
\noalign{\smallskip}
\hline\noalign{\smallskip}\begin{tabular}{@{}c@{}}\cref{algo:semiNewtonAbstractGeneral}\end{tabular}  
& \begin{tabular}{@{}c@{}} 2 \\ (1.56)  \end{tabular}   
& \begin{tabular}{@{}c@{}} 3 \\ (6.91)  \end{tabular}   
& \begin{tabular}{@{}c@{}} 3 \\ (9.16)  \end{tabular}   
& \begin{tabular}{@{}c@{}} 3 \\ (11.29)  \end{tabular}   
& \begin{tabular}{@{}c@{}} 3 \\ (13.59)  \end{tabular}   
& \begin{tabular}{@{}c@{}} 3 \\ (16.42)  \end{tabular}   
& \begin{tabular}{@{}c@{}} 3 \\ (20.26)  \end{tabular}   
& \begin{tabular}{@{}c@{}} -  \end{tabular}   
\\
\noalign{\smallskip}
\hline\noalign{\smallskip}\begin{tabular}{@{}c@{}}\ref{item:C4}\end{tabular}  
& \begin{tabular}{@{}c@{}} 2 \\ (1.56)  \end{tabular}   
& \begin{tabular}{@{}c@{}} 3 \\ (6.91)  \end{tabular}   
& \begin{tabular}{@{}c@{}} 3 \\ (9.16)  \end{tabular}   
& \begin{tabular}{@{}c@{}} 3 \\ (11.29)  \end{tabular}   
& \begin{tabular}{@{}c@{}} 3 \\ (13.59)  \end{tabular}   
& \begin{tabular}{@{}c@{}} 3 \\ (16.42)  \end{tabular}   
& \begin{tabular}{@{}c@{}} 3 \\ (20.26)  \end{tabular}   
& \begin{tabular}{@{}c@{}} -  \end{tabular}  
\\
\noalign{\smallskip}\hline
\end{tabular}
\medskip
\caption{(Test 2) Number of iterations 
needed by the pure fixed-point method \ref{item:C1}, \cref{algo:semiNewtonAbstract,algo:semiNewtonAbstractGeneral},
and \ref{item:C4} (\cref{algo:semiNewtonAbstractGeneral} with a backtracking linesearch)
to converge to $\bar u_1$ for the QVI-problem 
 \eqref{eq:ModelQVINum} with setup \labelcref{eq:ex_2_setup} and 
 $(\alpha_1,\alpha_2) = (10, 3/2)$ 
 in dependence of the projection radius $R$. The bracketed numbers are the largest EOC-values 
 measured for the respective regime. We chose a mesh size of $h=5 \times 10^{-4}$ and terminated the algorithm when $\|u_i\|_{\smash{H^1_0(0,1)}} \leq 10^{-13}$. 
 A dash implies that the algorithm stagnated without convergence. 
 All of the considered algorithms 
 failed to converge for $17 \lesssim R < \sqrt{50}\pi$.} 
\label{tab:test2-convergence}
\end{table}

As can be seen,
the number of iterations 
that \cref{algo:semiNewtonAbstract}
and the fixed-point method 
\ref{item:C1} need to approximate $\bar u_1$
in $H_0^1(\Omega)$ up to the tolerance $10^{-13}$
grows as the projection radius increases. Interestingly, both methods still converge to the zero solution for $R \lesssim 9$ despite $R$ being greater than the critical radius $R_{\alpha_1,\alpha_2}$. 
For larger $R$, both methods stagnate without convergence.
This behavior is caused by the loss 
of contractivity on $B_R(0)$ for large
projection radii $R$;
cf.\ the behavior in \cref{subsec:7.2}. 

As can be seen in \cref{tab:test2-convergence}, \cref{algo:semiNewtonAbstractGeneral} with and without a backtracking linesearch 
converges to 
$\bar u_1$ for all $R \lesssim 15$.
Similarly to \cref{subsec:7.2}, they both
handle the lack of contractivity 
for large $R$
far better than \cref{algo:semiNewtonAbstract}
and the fixed-point method \ref{item:C1}
and both converge in a maximum of three iterations when
$R \lesssim 15$.
To allow for a comparison of convergence speeds, 
\cref{tab:test2-convergence}
also shows the largest EOC-values that were
measured for the four
algorithms during each run
(defined as in \eqref{eq:EOC}).
As can be seen, \cref{algo:semiNewtonAbstract} and \cref{algo:semiNewtonAbstractGeneral} with and without a backtracking linesearch exhibit superlinear convergence speed. However, in contrast 
to the example in \cref{subsec:7.2},
this time we also observe  $q$-superlinear 
convergence for the fixed-point method \ref{item:C1}.
The reason for this effect is that the 
Lipschitz constant $\mathrm{Lip}(g,[-t, t]) = 8t/\alpha_1$ of the nonlinearity in \eqref{eq:ex_2_setup} 
goes to zero for $0 < t \to 0$.
(Note that, for  
\cref{algo:semiNewtonAbstractGeneral} and \ref{item:C4}, the number of iterations is so small that the EOC-values 
in \cref{tab:test2-convergence} should be taken with a grain of salt.)

\subsection{Test 3: the thermoforming QVI in two
dimensions}
\label{subsec:7.4}
As a third test case, we consider 
the QVI \eqref{eq:ModelQVINum} in the following situation:
\begin{equation}
\label{eq:ex_3_setup}
\begin{gathered}
\Omega  = (0,1)^2,
\qquad
\Phi_0(x_1, x_2)
=
1 - 2 \max(|x_1 - 0.5|,|x_2 - 0.5|),
\qquad
f \equiv 25, 
\\
\varphi(x_1, x_2)
=
\sin(\pi x_1)
\sin(\pi x_2)
,\qquad
k=1,
\qquad 
\Psi_0 \equiv \Phi_0,
\qquad 
\psi \equiv \varphi,
\\
g(s) = \begin{cases}
1/5 & \text{if} \; s\leq 0,\\
(1-s)/5& \text{if} \; 0< s < 1,\\
0 & \text{otherwise}.
\end{cases}
\end{gathered}
\end{equation}
Note that the above setting 
corresponds to 
an example of the thermoforming problem described in \cref{rem:thermoforming}.
From the application point of view, 
it models the situation that a hot thin square
plastic sheet is pushed by a 
constant pressure into a pyramidal metal 
mould. For the data in \eqref{eq:ex_3_setup},
one can check that the 
conditions in 
\cref{ass:standing_obstacle_map,ass:SemilinearObstacleMap} are satisfied with $p=2$.
From a straightforward calculation and 
\cref{rem:Poincare}, we further obtain that 
\[
\smash{C_P(\Omega)\,\Lip{g} 
\left ( \|\varphi \|_{L^\infty(\Omega)}  k^{-1/2} 
+
\| |\nabla \varphi| \|_{L^\infty(\Omega)} k^{-1}
\right )
\leq 
\frac{1+\pi}{5}
\approx 0.8283 < 1}.
\]
In combination with \cref{cor:gamma_prop,th:obst_QVI,th_T_probs},
this shows that the setting \eqref{eq:ex_3_setup}
is covered by the analysis of \cref{sec:framework_applied_to_SPhi}
and that our global convergence result
for \cref{algo:semiNewtonAbstract} applies. 
In particular, \eqref{eq:ModelQVINum} 
possesses a unique solution 
$(\bar u, \bar T)$ in the case of \eqref{eq:ex_3_setup}
by \cref{lem:uniqueF_H}.

The results that we have  obtained in the situation
\eqref{eq:ex_3_setup} for the QVI \eqref{eq:ModelQVINum} can be seen 
in \cref{fig:tf-solutions}
and \cref{tab:tf-mesh-independence}.
\Cref{fig:tf-solutions} shows surface plots of the membrane $\bar u$ and the mould $\Phi_0 + \varphi \bar T$ as well as a slice of the membrane, mould, temperature $\bar T$, and mould deformation 
$\varphi \bar T$ at $x_2=1/2$. It also depicts the convergence behavior of
\cref{algo:semiNewtonAbstract}, the fixed-point method \ref{item:C1}, 
and the Moreau--Yosida-based regularization method
\ref{item:C2} for various choices of $\rho$. We see that \cref{algo:semiNewtonAbstract} converges the fastest
and that the fixed-point method \ref{item:C1} is the slowest, exhibiting linear convergence speed (as expected). For the
regularization method
\ref{item:C2}, the convergence degrades for $\rho \to 0$. 
Recall that 
$\rho$ must be driven to zero in \ref{item:C2}
in order to get an approximate solution close to the true QVI-solution. 
This, however, causes ill-conditioning effects that reduce the convergence speed.

\Cref{tab:tf-mesh-independence}
shows the number of 
outer semismooth
Newton steps as well 
as the overall number of inner 
semismooth Newton iterations, 
PFMY-iterations, and PDAS-feasibility 
restoration steps 
that 
\cref{algo:semiNewtonAbstract}
requires 
in the situation of \eqref{eq:ex_3_setup}
for various mesh widths $h$
to drive the residue 
$\|R(u_{i})\|_{H^1(\Omega)}$
below the tolerance $10^{-12}$.
As can be seen, 
the number of 
semismooth Newton steps 
that \cref{algo:semiNewtonAbstract} needs
is four or fewer for all considered mesh
widths $h$. This shows that---on the level of the solver for the fixed-point equation 
\eqref{eq:SPhi_FP}---we observe
\emph{mesh-independence} for our semi\-smooth 
Newton method. This very desirable property 
is characteristic for solution algorithms 
whose convergence can be established 
not only for the discrete problems 
obtained from a discretization but also on the function space level;
see \cref{sec:framework_applied_to_SPhi}. Note that,
for the inner semismooth-Newton solver 
used for the evaluation of $\Phi$
(i.e., the solution of the 
nonsmooth semilinear PDE in \eqref{eq:ModelQVINum}), we also observe 
a mesh-independent convergence behavior. 
For the 
hybrid PFMY-PDAS-algorithm used
for the evaluation of the solution map 
$S$ of the obstacle problem, 
a mild form of mesh-dependence is present 
that, however, is manageable. 
Note that the mesh-dependence 
of solution algorithms for the obstacle problem 
is a known problem. Nevertheless other experimentally mesh-independent solution
methods, for the evaluation of $S$, exist than just the one considered in this paper \cite{keith2023, graser2009}. The fact that \cref{algo:semiNewtonAbstractGeneral,algo:semiNewtonAbstract} reduce the number of outer iterations and, thus, 
required evaluations of $S$ to a minimum
is a main advantage of our
semi\-smooth Newton approach 
in comparison with fixed-point-based methods that have to evaluate $S$ far more often 
due to their slow convergence speed;
cf.\ \cref{fig:tf-solutions}(d).

\begin{figure}[ht!]
\centering
\subfigure[The membrane $\bar u$]{\includegraphics[width =0.49 \textwidth, keepaspectratio]{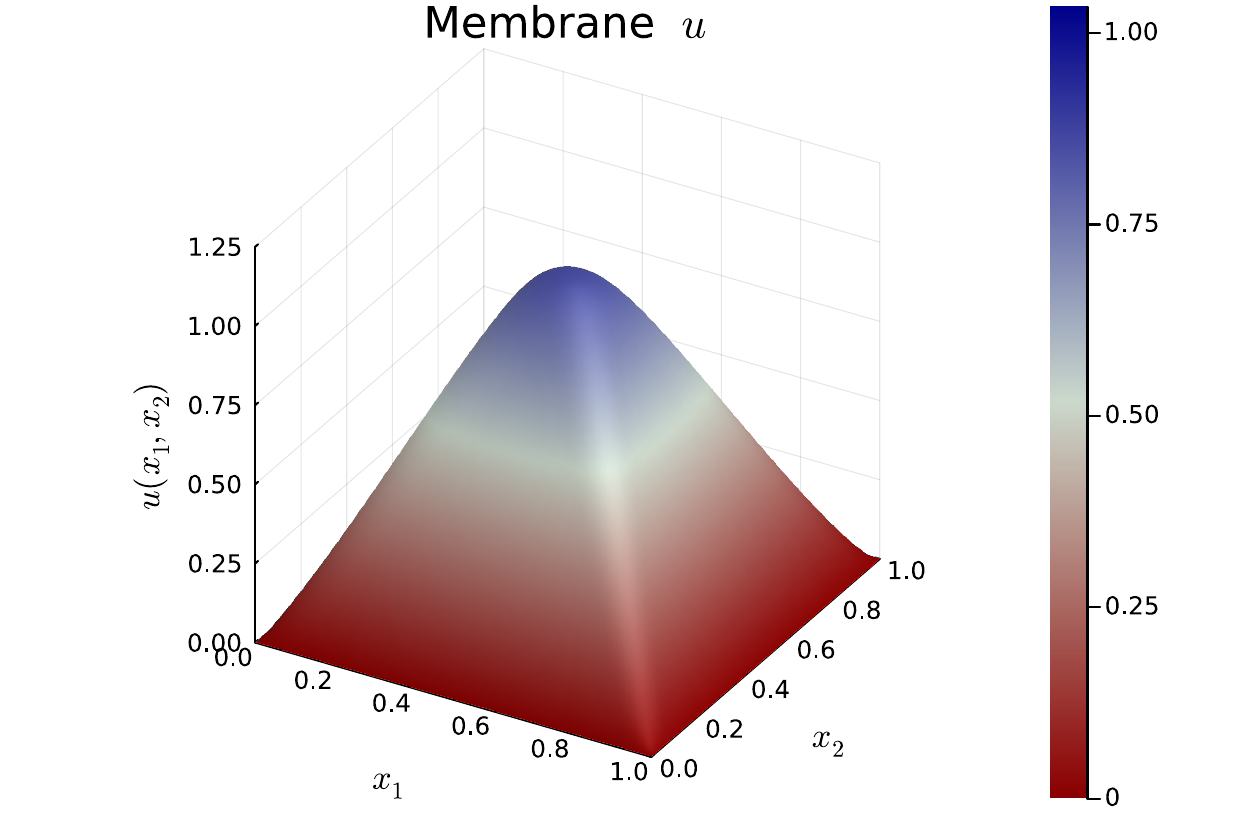}}
\subfigure[The mould $\Phi_0 + \varphi\bar T$]{\includegraphics[width =0.49 \textwidth, keepaspectratio]{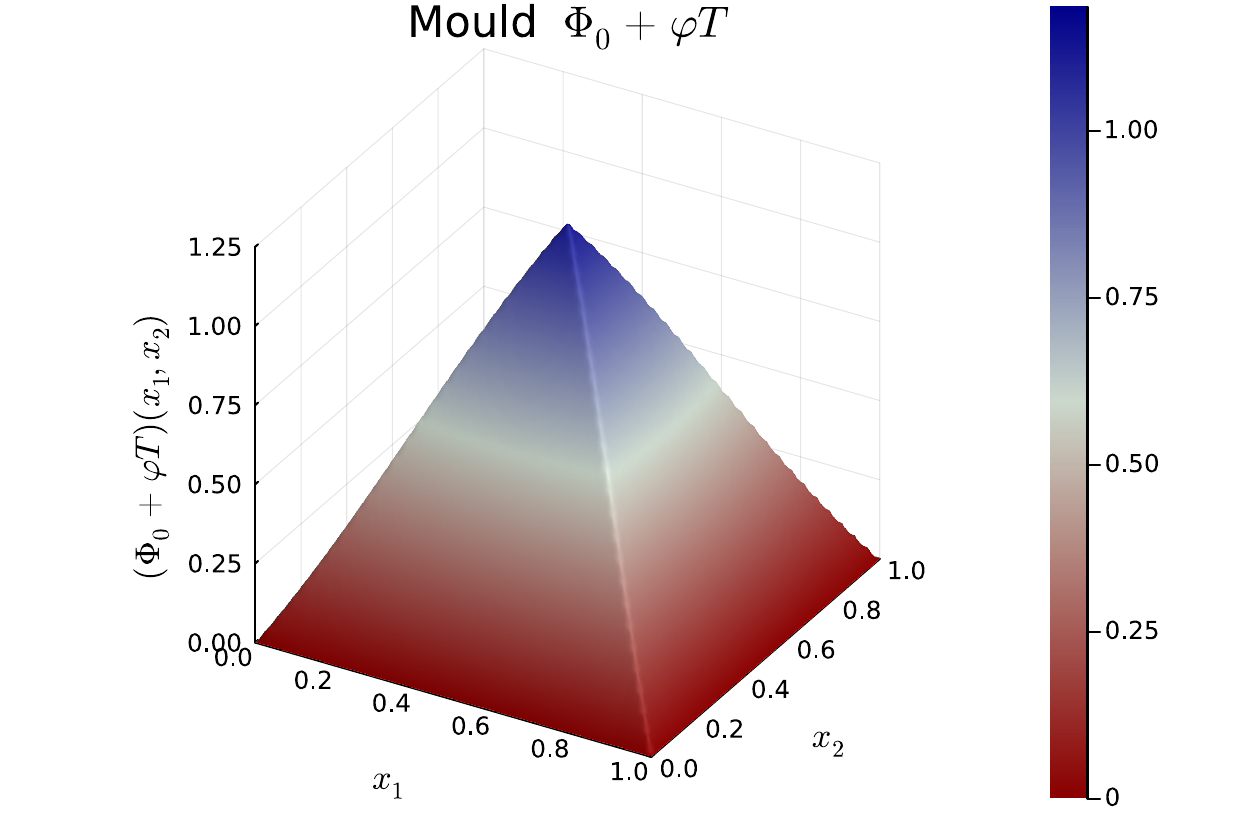}}
\\[0.4cm]
\subfigure[Slice at $x_2 = 1/2$]{\includegraphics[width =0.49 \textwidth]{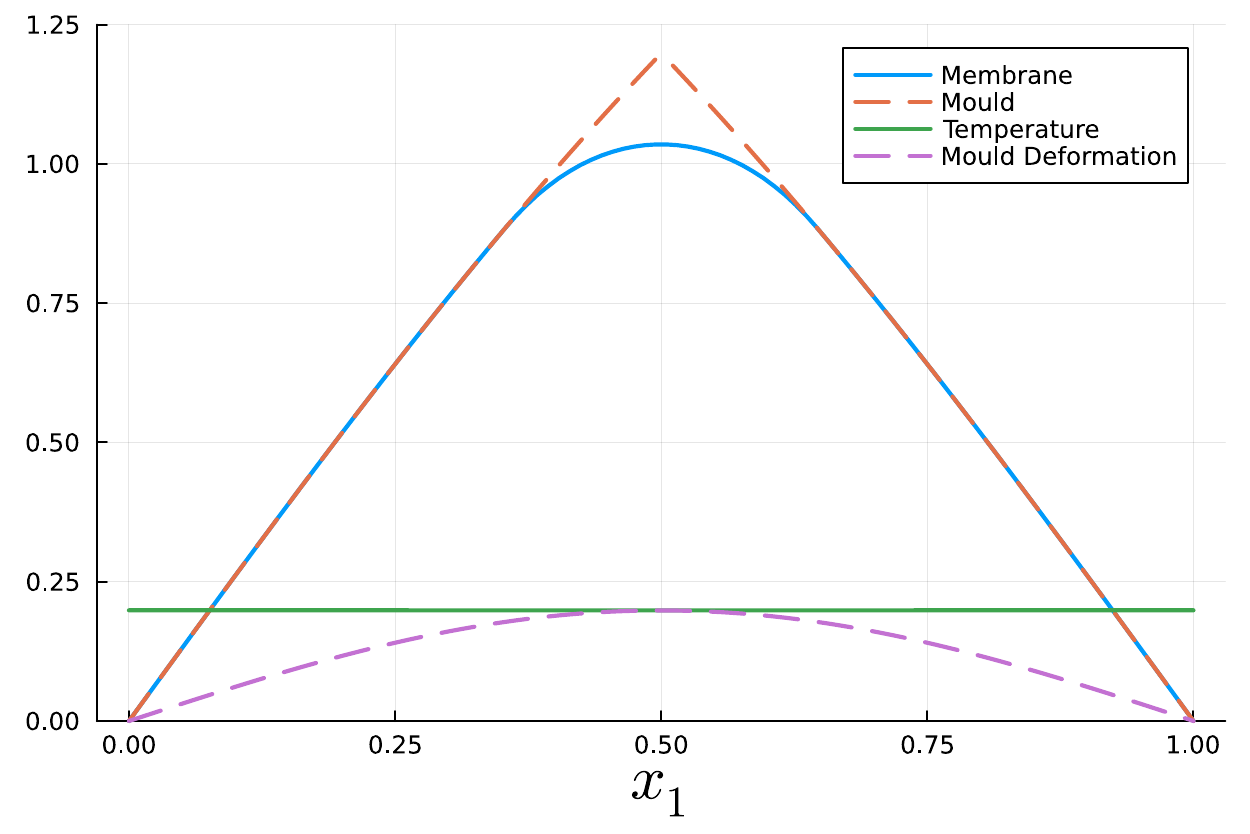}}
\subfigure[Convergence behavior]{\includegraphics[width =0.49 \textwidth]{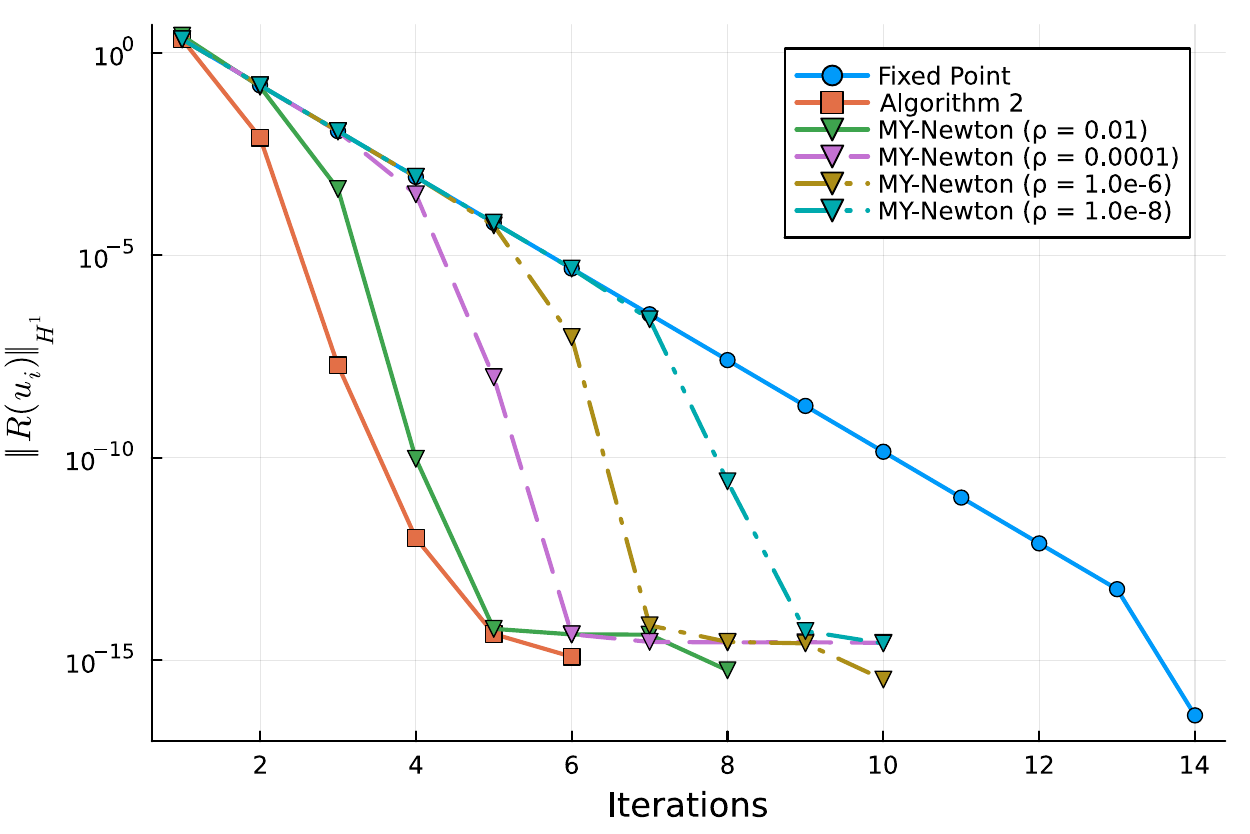}}

\caption{(Test 3) Surface plots of the membrane (a) and corresponding mould (b) together with a slice plot at $x_2=1/2$ (c) for the thermoforming setting \eqref{eq:ex_3_setup}. 
Figure (d) depicts the outer loop convergence of \cref{algo:semiNewtonAbstract}, 
the fixed-point method 
\ref{item:C1}, and the 
regularization method \ref{item:C2}
for $h=0.02$ and various $\rho$. 
The residue depicted for  \ref{item:C2}
is that of the smoothed system.
It can be seen that \cref{algo:semiNewtonAbstract} converges the fastest and that 
the convergence speed of 
the Moreau--Yosida-based method
\ref{item:C2} degrades for 
$\rho \to 0$ due to ill-conditioning effects.}
\label{fig:tf-solutions}
\vspace{0.4cm}
\end{figure}
\begin{table}[ht]
\small
\centering
\begin{tabular}{l|l|l|ll}
\toprule
&Outer loop & Inner solver to evaluate $\Phi$ & \multicolumn{2}{|c}{Inner VI-solver to evaluate $S$}  \\
\midrule
$h$    &  Semismooth Newton &  Semismooth Newton & PFMY  & PDAS  \\ \midrule
0.04  &4 (4) &13 (9) &293 (159) &17 (10)\\
0.02  &4 (4) &13 (9) &340 (185) &31 (17)\\
0.01  &3 (3) &12 (8) &277 (150) &20 (11)\\
0.00667  &3 (3) &12 (8) &283 (158) &17 (11)\\
0.005  &3 (3) &12 (8) &285 (158) &29 (17)\\
0.004  &3 (4) &11 (8) &289 (199) &29 (21)\\
0.00333  &3 (4) &10 (7) &262 (184) &29 (21)\\
\bottomrule
\end{tabular}
\medskip
\caption{(Test 3) Number of iterations of the outer loop and cumulative number of iterations for the inner loops when \cref{algo:semiNewtonAbstract} and \cref{algo:semiNewtonAbstractGeneral} (in brackets)  are applied to \eqref{eq:ModelQVINum}
in the situation of  \eqref{eq:ex_3_setup}
for various mesh widths $h$. 
The algorithm is terminated once $\|R(u_{i})\|_{H^1(\Omega)} \leq 10^{-12}$. We observe that the number of outer loop iterations 
and the number of inner semismooth Newton steps 
needed for the evaluation of $\Phi$ are mesh-independent
and that the number of PFMY- and PDAS-iterations 
does not grow in an uncontrollable manner.} 
\label{tab:tf-mesh-independence}
\end{table}

\subsection{Test 4: a nonlinear VI}
\label{subsec:6.5}
Finally, we briefly consider the nonlinear VI-example \eqref{eq:general_FP_semi} of \cref{sec:application_nonlinearVI}. We omit conducting 
detailed 
numerical experiments 
 along the lines of \cref{subsec:7.1,subsec:7.2,subsec:7.3,subsec:7.4} here
to avoid overloading the paper
and simply present a short feasibility study
for a particular instance of \eqref{eq:general_FP_semi};
see \cref{tab:test4-mesh-independence}
and \cref{fig:test4-solution}. 
It can be observed that 
\cref{algo:semiNewtonAbstract} behaves 
similar for the semilinear VI 
\eqref{eq:general_FP_semi} as for the 
QVI-examples in \cref{subsec:7.2,subsec:7.3,subsec:7.4}.

\begin{table}[ht]
\small
\centering
\begin{tabular}{l|l|ll}
\toprule
&Outer loop  & \multicolumn{2}{|c}{Inner VI-solver to evaluate $S$}  \\
\midrule
$h$    &  Semismooth Newton & PFMY  & PDAS  \\ \midrule
0.02  &4 (5) &205 (136) &10 (7)\\
0.01  &4 (4) &248 (142) &17 (9)\\
0.005  &3 (4) &210 (152) &7 (5)\\
0.0025  &4 (4) &264 (149) &9 (5)\\
0.00125  &4 (4) &296 (166) &17 (9)\\
0.00062  &4 (4) &314 (176) &9 (5)\\
0.00031  &3 (4) &262 (188) &7 (5)\\
0.00016  &3 (4) &269 (194) &7 (5)\\
\bottomrule
\end{tabular}
\medskip
\caption{(Test 4) Number of iterations of the outer loop and cumulative numbers of iterations for the inner loops when \cref{algo:semiNewtonAbstract} and \cref{algo:semiNewtonAbstractGeneral} (in brackets) are applied to \eqref{eq:semilinear_VI}
for various mesh widths $h$. The data was chosen as 
$\Omega = (0,1)$,
$f(x) = 50\sin(2\pi x)$, $b_1(s) = \max(0,s)$, $b_2(s) = s+\cos s$, and $\Phi_0 \equiv 1$.
The 
initial guess was $u_0(x) = 0$.
The algorithm was terminated once $\|R(u_{i})\|_{H^1(\Omega)} \leq 10^{-10}$. 
It can be seen that the algorithms converge and
that the number of outer loop semismooth Newton steps
and inner PFMY-/PDAS-iterations does not grow in an uncontrollable manner; 
similarly to the results in \cref{tab:tf-mesh-independence}.} 
\label{tab:test4-mesh-independence}
\end{table}
 
\begin{figure}[ht!]
\centering
\includegraphics[width =0.49 \textwidth]{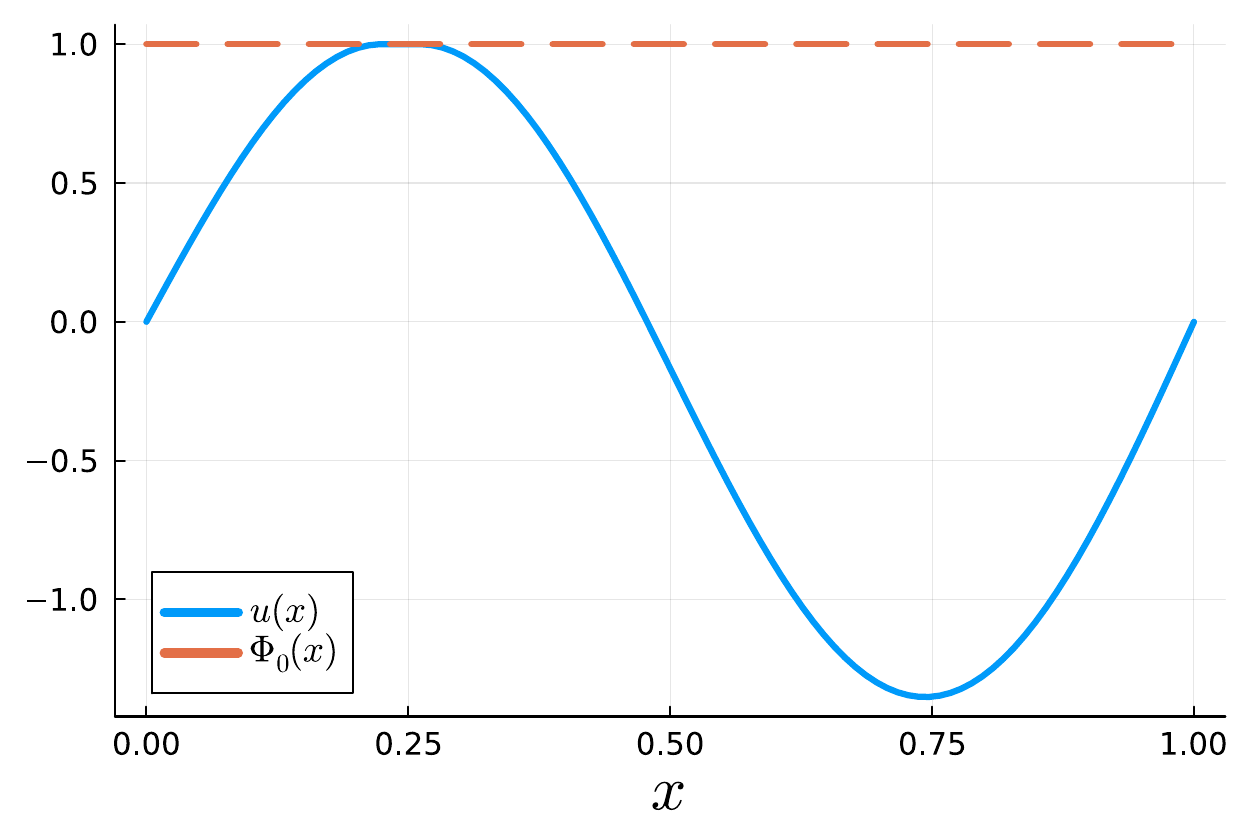} 
\caption{(Test 4) Solution of the semilinear VI \eqref{eq:semilinear_VI}
in the situation of \cref{tab:test4-mesh-independence}.
The mesh size is $h=1.56 \times 10^{-4}$.}
\label{fig:test4-solution}
\end{figure}

We remark that, analogously to the approach presented in \cref{sec:application_nonlinearVI},
one can also establish that the analysis of 
\cref{sec:framework_applied_to_SPhi} covers nonsmooth semilinear and 
quasilinear partial differential equations
(for which, as mentioned, 
contraction assumptions 
of the type used in \cref{sec:framework_applied_to_SPhi} are 
standard tools in existence proofs).
A further interesting application area for the framework 
of \cref{sec:framework_applied_to_SPhi} is the study of fixed-point problems 
\eqref{eq:SPhi_FP} in which both $S$ and $\Phi$ arise 
from a variational inequality. We leave both these topics for
future research.

\section*{Funding}
IP acknowledges support by the Deutsche Forschungsgemeinschaft (DFG, German Research Foundation) under Germany's Excellence Strategy -- The Berlin Mathematics Research Center MATH+ (EXC-2046/1, project ID: 390685689).

\begin{appendices}
\crefalias{section}{appendix}
\section{Calculus rules for Newton derivatives}\label{sec:appendix}

The following calculus rules are well known and can be found in various variants
in the literature; see, e.g., \cite[\S 3.3.7]{Ulbrich2011}
for a version of the chain rule in Lebesgue spaces. 

\begin{lemma}[Chain rule for Newton derivatives]\label{lem:chain_rule}
Let $(U, \norm{\cdot}{U})$, 
$(V, \norm{\cdot}{V})$, and $(W, \norm{\cdot}{W})$ be real normed spaces 
and let $E \subset V$ be nonempty.
Suppose that $\Psi\colon U \to E$
and $T\colon E \to W$ are Newton differentiable 
functions with Newton derivatives 
$G_\Psi \colon U \to \LL(U,V)$
and
$G_T\colon E \to \LL(V, W)$,
respectively. 
Suppose that, for every $u \in U$, there exist constants  $ C, \varepsilon > 0$ such that
\begin{equation}
\label{eq:Psi_Lip_gen}
\|  \Psi(u + h)  - \Psi(u) \|_V
\leq
C\|h\|_U
\quad
\forall h \in B_\varepsilon^U(0),
\end{equation}
and that, for every $v \in E$,  there exist constants  $ C, \varepsilon > 0$ such that 
\begin{equation}
\label{eq:GT_loc_uniform_bound}
	\sup_{w \in E \cap B_\varepsilon^V(v)} \left \| G_T(w) \right \|_{\LL(V, W)} \leq C.
\end{equation}
Define $K\colon U \to W$ by 
$K := T \circ \Psi$. Then the function 
$K$ is Newton differentiable with Newton derivative $G_K\colon U \to \LL(U,W)$,
$G_K(u) := G_T (\Psi(u))G_\Psi(u)$.
\end{lemma}

\begin{proof}
If $u \in U$ is fixed and $\{h_n\} \subset U$ is an arbitrary sequence 
satisfying $\Psi(u+h_n) - \Psi(u) \neq 0$ for all $n \in \mathbb{N}$ and 
$0 <\|h_n\|_U \to 0$ for $n \to \infty$, then, for all large enough $n$,
we have 
\begin{equation}
\label{eq:randomeq_11}
\begin{aligned}
0 &\leq\frac{\|K(u + h_n) - K(u) - G_K(u+h_n)h_n\|_{W}}{\|h_n\|_{U}}
\\
&=
\frac{\|  T(\Psi(u + h_n))  - T(\Psi(u))  
 - G_T (\Psi(u + h_n))G_\Psi(u + h_n)h_n\|_{W}}{\|h_n\|_{U}}
 \\
 &\leq
 \frac{\|  T(\Psi(u + h_n))  - T(\Psi(u))  
 - G_T (\Psi(u + h_n))(\Psi(u + h_n) - \Psi(u))\|_{W}}{\|\Psi(u + h_n) - \Psi(u)\|_{V}}
 \frac{\|\Psi(u + h_n) - \Psi(u)\|_{V}}{\|h_n\|_{U}}
 \\
 &\qquad 
 +
 \|G_T (\Psi(u + h_n))\|_{\LL(V,W)}
 \frac{\|  
 \Psi(u + h_n) - \Psi(u) - G_\Psi(u + h_n)h_n\|_{V}}{\|h_n\|_{U}} 
 \\
 &\leq
 C \frac{\|  T(\Psi(u + h_n))  - T(\Psi(u))  
 - G_T (\Psi(u + h_n))(\Psi(u + h_n) - \Psi(u))\|_{W}}{\|\Psi(u + h_n) - \Psi(u)\|_{V}}
 \\
 &\qquad+
 C \frac{\|  
 \Psi(u + h_n) - \Psi(u) - G_\Psi(u + h_n)h_n\|_{V}}{\|h_n\|_{U}}. 
\end{aligned}
\end{equation}
Here, we have used 
\eqref{eq:Psi_Lip_gen} and \eqref{eq:GT_loc_uniform_bound}
in the last step.
Due to the Newton differentiability of $T$ and $\Psi$,
the right-hand side of \eqref{eq:randomeq_11} 
goes to zero for $n \to \infty$. 
By adjusting the estimates in \eqref{eq:randomeq_11} slightly,
we also obtain this convergence for sequences 
satisfying $\Psi(u+h_n) - \Psi(u) = 0$ for some/all $n$. 
In combination with the arbitrariness of 
 $\{h_n\}$ and $u \in U$, this proves the lemma.
\end{proof}

\begin{lemma}[Product rule for Newton derivatives]\label{lem:product_rule}
Let
$(U, \norm{\cdot}{U})$, 
$(V, \norm{\cdot}{V})$,
$(W, \norm{\cdot}{W})$,
and 
$(Z, \norm{\cdot}{Z})$
be normed spaces and let 
$a\colon V \times W \to Z$
be a bilinear and continuous mapping,
i.e., it holds
$
\norm{a(v,w)}{Z} \leq C_a \|v\|_V \|w\|_W$
for all
$(v,w) \in V \times W$
with a constant $C_a > 0$.
Assume that 
$P\colon U \to V$ and 
$Q\colon U \to W$ are Newton differentiable 
with Newton derivatives 
$G_P \colon U \to \LL(U,V)$
and
$G_Q\colon U \to \LL(U, W)$,
respectively. 
Suppose that $P$ 
and  $Q$ are continuous with one of them 
being locally Lipschitz.
Then the function
$K\colon U \to Z$, 
$K(u) := a(P(u), Q(u))$,
is Newton differentiable with
derivative 
\[
G_K\colon U \to \LL(U,Z),
\qquad 
G_K(u)h := a(P(u), G_Q(u)h ) + a(G_P(u)h, Q(u))
\quad \forall u,h \in U.
\]
\end{lemma}
\begin{proof}
For all $u\in U$ and $h \in U \setminus \{0\}$,
we have
\begin{equation*}
\begin{aligned}
&K(u+h) - K(u) - G_K(u+h)h
\\
&
\begin{aligned}
= a(P(u+h), Q(u+h)) 
- a(P(u), Q(u))
- a(P(u+h), G_Q(u+h)h ) - a(G_P(u+h)h, Q(u+h))
\end{aligned}
\\
&\begin{aligned}
=
a(P(u+h) - P(u) - G_P(u+h)h, Q(u+h)) 
&+ a(P(u+h), Q(u+h) - Q(u) - G_Q(u+h)h )
\\
&+ a(P(u)-P(u+h), Q(u+h) - Q(u)),
\end{aligned}
\end{aligned}
\end{equation*}
and, thus, 
\begin{equation}
\label{eq:randomeq2636}
\begin{aligned}
\frac{\norm{K(u+h)-K(u)-G_K(u+h)h}{Z}}{\norm{h}{U}} 
&\leq 
C_a \norm{Q(u+h)}{W}\frac{\norm{P(u+h) - P(u) - G_P(u+h)h}{V}}{\norm{h}{U}}
\\
&\quad +
C_a \norm{P(u+h)}{V}\frac{\norm{Q(u+h) - Q(u) - G_Q(u+h)h}{W}}{\norm{h}{U}}
\\
&\quad +
C_a \frac{\norm{P(u+h) - P(u)}{V}\norm{Q(u+h) - Q(u)}{W}}{\norm{h}{U}}.
\end{aligned}
\end{equation}
Due to the Newton differentiability and continuity 
properties of $P$ and $Q$, the right-hand side 
of \eqref{eq:randomeq2636} tends to zero 
for $0 < \|h\|_U \to 0$. This proves the assertion;
see \eqref{eq:NewtonDiffDef}.
\end{proof} 

\section{Convergence of semismooth Newton methods}\label{sec:proofVanillaSSN}
 
\begin{proof}[Proof of \cref{th:convergenceGeneral}]
The proof is standard; see, e.g., \cite[Theorem 3]{MARTINEZ1995127} but we give it here for
the sake of completeness. The openness of $B$ and the 
 Newton differentiability of $R$ on $B$
imply that, for every $\epsilon > 0$, there exists $r > 0$ such that 
$B^X_r(\bar x) \subset B$ and 
\begin{equation}
\label{eq:rnd_implication}
x \in B^X_r(\bar x) \implies \norm{R(x)-R(\bar x) - G_R(x)(x-\bar x)}{X} \leq \epsilon\norm{x-\bar x}{X}.
\end{equation}
Choose $\epsilon > 0$ such that 
$\alpha := M\epsilon+ M L\rho^* < 1$  holds 
and let $r>0$ be such that \eqref{eq:rnd_implication} is satisfied for this $\epsilon$.
Let $x_i \in B^X_r(\bar x)$ be arbitrary
and let $x_{i+1}$ and $z_i$ be as in steps \ref{algo1:x_N_general} and \ref{algo1:update_formula} of \cref{algo:semiNewtonAbstractGeneral}. Then
\begin{equation}
\label{eq:gen_est}
\begin{aligned}
\norm{x_{i+1}-\bar x}{X} 
&= \norm{x_i - \bar x - G_R(x_i)^{-1}R(x_i) + G_R(x_i)^{-1}(G_R(x_i)z_i + R(x_i))}{X}\\
&= \norm{G_R(x_i)^{-1}G_R(x_i)(x_i - \bar x) - G_R(x_i)^{-1} R(x_i) + G_R(x_i)^{-1}(G_R(x_i)z_i + R(x_i))}{X}\\
&\leq \norm{G_R(x_i)^{-1}}{\LL(X,X)}\big(\norm{R(x_i)-R(\bar x) - G_R(x_i)(x_i-\bar x)}{X} + \norm{G_R(x_i)z_i + R(x_i)}{X}\big)\\
&\leq M\left(\epsilon\norm{x_i-\bar x}{X} + \rho_i\norm{R(x_i)}{X}\right)\\
&\leq (M\epsilon+ M L\rho_i )\norm{x_i-\bar x}{X}
\\
&\leq  \alpha \norm{x_i-\bar x}{X}.
 \end{aligned}
 \end{equation}
This shows that $x_{i+1} \in B_r^X(\bar x)$ holds and, after a trivial induction, 
that the iterates produced by \cref{algo:semiNewtonAbstractGeneral}
satisfy $\norm{x_i-\bar x}{X} \leq \alpha^i \norm{x_0-\bar x}{X}$
for all $x_0 \in B_r^X(\bar x)$ and all $i$. The assertions in 
\ref{vanilla:assertion:i} and \ref{vanilla:assertion:ii}
follow immediately from this estimate and \eqref{eq:Phi_L_Lip_inexact}.
The $q$-superlinear convergence in \ref{vanilla:assertion:iii}
is obtained by revisiting the estimates in \eqref{eq:gen_est}
with the knowledge that $x_i \to \bar x$. 
\end{proof} 

Next, we wish to prove \cref{th:convergence} concerning the convergence of \cref{algo:semiNewtonAbstract}. We first 
require the following lemma on the properties of 
the residue function $R$ of \eqref{eq:general_FP_H_intro}. 

\begin{lemma}[Properties of $R$]
\label{lem:R:prop}Suppose that \cref{ass:standing_general_H} holds. Then
the function $R\colon X \to X$ 
satisfies the following:
\begin{enumerate}[label=\roman*)]
\item\label{lem:R:prop:i} $R$ is bijective and its inverse $R^{-1}\colon X \to X$ satisfies 
\begin{equation}
\label{eq:R_estimate_1}
\| R^{-1}(x_1) - R^{-1}(x_2)\|_X \leq \frac{1}{1-\gamma} \|x_1 - x_2\|_X
\quad \forall x_1, x_2 \in X
\end{equation}
and
\begin{equation}
\label{eq:R_estimate_2}
\|x_1 - x_2\|_X
\leq
(1 + \gamma)
\| R^{-1}(x_1) - R^{-1}(x_2)\|_X  \quad \forall x_1, x_2 \in X.\smallskip
\end{equation}
\item\label{lem:R:prop:ii} $R$ is Newton differentiable 
on $X$ with Newton derivative 
$G_R(x) := \mathrm{Id} - G_H(x)$.
\item\label{lem:R:prop:iii} For every $x \in X$, the inverse $G_R(x)^{-1}$ exists and it holds 
$\left \| G_R(x)^{-1} \right \|_{\LL(X,X)} \leq (1-\gamma)^{-1}$.
\end{enumerate}
\end{lemma}
\begin{proof} 
To prove \ref{lem:R:prop:i}, suppose that $y \in X$ is given. 
Then $y = R(x)$ is equivalent to the fixed-point equation $x = y + H(x)$,
which possesses a unique solution $x := R^{-1}(y)$ by the Banach fixed-point theorem. 
Thus, $R$ is bijective and $R^{-1}\colon X \to X$ exists. 
Consider now some $x_1, x_2 \in X$. Then it holds 
$x_j = R(R^{-1}(x_j)) = R^{-1}(x_j) - H(R^{-1}(x_j))$, $j=1,2$,
by the definition of $R$. Thus, by  the 
triangle inequality and \eqref{eq:H_gamma_Lip},
\begin{equation*}
\begin{aligned}
\| R^{-1}(x_1) - R^{-1}(x_2)\|_X
&=
\| x_1 - x_2 + H(R^{-1}(x_1)) - H(R^{-1}(x_2)) \|_X
\\
&\leq
\| x_1 - x_2\|_X 
+
\gamma \| R^{-1}(x_1) - R^{-1}(x_2)\|_X.
\end{aligned}
\end{equation*}
This establishes \eqref{eq:R_estimate_1}.
The second estimate \eqref{eq:R_estimate_2} follows from 
\[
\| x_1 - x_2\|_X 
=
\| R(R^{-1}(x_1)) - R(R^{-1}(x_2))\|_X
\leq
(1+\gamma)
 \| R^{-1}(x_1) - R^{-1}(x_2)\|_X. 
\]
This proves \ref{lem:R:prop:i}. 
The assertion of \ref{lem:R:prop:ii} follows 
from the sum rule for Newton differentiable functions and 
the Newton differentiability of $H$.
To finally establish \ref{lem:R:prop:iii}, it suffices to note that 
the $\LL(X,X)$-norm of the operator $G_H (x)$ appearing in the definition of
$G_R(x)$ is bounded by $\gamma \in [0,1)$
due to \cref{ass:standing_general_H}\ref{ass:standing_general:Hiii} 
and to use Neumann's series. 
This completes the proof.
\end{proof}

Using the properties in \cref{lem:R:prop}, we can now prove \cref{th:convergence}.
\smallskip

\begin{proof}[Proof of \cref{th:convergence}]
Let $x_i \in X$ be given 
and let $x_B, x_N \in X$ be chosen 
such that \eqref{eq:xB_update} and \eqref{eq:xN_update} hold.
Then the $\gamma$-Lipschitz continuity of $H$ in \eqref{eq:H_gamma_Lip},
\eqref{eq:xB_update}, the triangle inequality,
and the definition of $R$
imply that 
\begin{equation*}
\begin{aligned}
\min \left (
\|R(x_B)\|_X,
\|R(x_N)\|_X 
\right )
&\leq 
\|R(x_B)\|_X
\\
&= 
\| x_B - H(x_i) + H(x_i) - H(x_B)\|_X
\\
&\leq
\tau_i
\|R(x_i)\|_X
+
\gamma 
\| x_i - H(x_i) + H(x_i) - x_B\|_X
\\
&\leq
\left (
\tau_i
+ \gamma + \gamma \tau_i
\right )
\|R(x_i)\|_X
\\
&= \beta_i \|R(x_i)\|_X
\end{aligned}
\end{equation*}
holds, 
where $\beta_i := \tau_i + \gamma + \gamma\tau_i$.
Due to the acceptance criterion in \cref{algo:semiNewtonAbstract},
the above yields that, if 
\cref{algo:semiNewtonAbstract}
does not terminate in the iterations $i=0,1,\ldots,n-1$, $n \in \N$, then 
$x_n$ satisfies 
\begin{equation}
\label{eq:Rn_42}
\|R(x_n)\|_X \leq \theta_{n-1} \|R(x_0)\|_X 
\end{equation}
with $\theta_{n-1} :=  \beta_0\beta_1\cdots \beta_{n-1}$.

Let us now first consider the situation in \ref{th:convergence:i}, 
i.e., the case $\Ntol > 0$  
and $\tau^* \leq (\lambda-\gamma)/(1+\gamma)$
for some $\lambda \in (\gamma, 1)$. 
Then we have 
\[
\beta_i =  \tau_i + \gamma + \gamma\tau_i
\leq
\tau^* + \gamma + \gamma \tau^*
\leq
\frac{\lambda-\gamma }{1+\gamma}
 + \gamma + \gamma\frac{\lambda-\gamma }{1+\gamma}
 = \lambda \in (\gamma, 1)\quad \forall i \in \N_0,
\]
and we obtain from \eqref{eq:Rn_42} that 
$\|R(x_n)\|_X \leq \lambda^n \|R(x_0)\|_X$ holds when 
the termination criterion $\|R(x_i)\|_{X} \leq \Ntol$ in  step \ref{algo1:term}
is not triggered for $i=0,1,\ldots,n-1$, $n \in \N$. 
That \cref{algo:semiNewtonAbstract} 
has to terminate after the number of iterations in \eqref{eq:iter_bound} 
follows immediately from this estimate. 
From \eqref{eq:R_estimate_1}, we further obtain 
that the iterate $x^*$ that \cref{algo:semiNewtonAbstract} 
returns in this situation satisfies 
\begin{equation}
\label{eq:randomeq_2626}
\begin{aligned}
\|x^*- \bar x\|_X
&=
\|R^{-1}(R(x^*)) - R^{-1}(R(\bar x))\|_X
\\
&\leq
\frac{1}{1-\gamma}
\| R(x^*) -  R(\bar x)\|_X
\\
&=
\frac{1}{1-\gamma}
\| R(x^*)\|_X
\\
&\leq
\frac{\Ntol}{1-\gamma}.
\end{aligned}
\end{equation}
This completes the proof of \ref{th:convergence:i}.

Let us now assume that $\Ntol = 0$ holds and that 
 \cref{algo:semiNewtonAbstract} does not terminate 
 after finitely many iterations. Then
 it follows from the inequality 
 $\|R(x_i)\|_X \leq \theta_{i-1}\|R(x_0)\|_X$
 for all $i \in \N$, the definitions 
 $\theta_{i} :=  \beta_0\beta_1\cdots \beta_{i}$
 and
 $\beta_i := \tau_i + \gamma + \gamma\tau_i$,
 the convergence $\tau_i \to 0$,
 and $\gamma \in (0,1)$,
 that 
  $\beta_i$ converges to $\gamma$, that $\theta_{i}$ goes  to zero, 
  and that $\| R(x_i)\|_X \to 0$ holds for $i \to \infty$.
  In combination with the estimate $\|x_i - \bar x\|_X
 \leq (1-\gamma)^{-1}\| R(x_i)\|_X$, that is obtained 
 along the exact same lines as \eqref{eq:randomeq_2626}, 
 this yields that $0 < \|x_i - \bar x\|_X \to 0$ for $i \to \infty$. 
 From \eqref{eq:xN_update},
 the properties in \cref{lem:R:prop},
 the acceptance criterion 
 in \cref{algo:semiNewtonAbstract}, 
 and the convergence $\rho_i \to 0$,
 we further obtain that
 \begin{equation*}
 \begin{aligned}
 \|R(x_{i+1})\|_X
 &\leq
 \|R(x_{N})\|_X
 \\
&= 
 \|R(x_{N}) - R(\bar x)\|_X
 \\
 &\leq
 (1 + \gamma)
 \|x_N -  \bar x\|_X
 \\
&\leq
\frac{1+\gamma}{1-\gamma}
\|G_R(x_i)(x_N -  \bar x)\|_X
\\
&\leq
\frac{1+\gamma}{1-\gamma}
\|R(\bar x) -  R(x_i) +  G_R(x_i)(x_i -\bar x)\|_X
+
\frac{1+\gamma}{1-\gamma} \|R(x_i) + G_R(x_i)(x_N -x_i)\|_X
\\
&\leq
\frac{1+\gamma}{1-\gamma}
 \frac{\|R(\bar x) -  R(x_i) +  G_R(x_i)(x_i -\bar x)\|_X}{\|x_i - \bar x\|_X} \|x_i - \bar x\|_X
 +
\frac{1+\gamma}{1-\gamma}  \rho_i 
 \|R(x_i)  \|_X  
 \\
 &\leq
   \frac{1+\gamma}{(1-\gamma)^2}
 \frac{\|R(\bar x) -  R(x_i) +  G_R(x_i)(x_i -\bar x)\|_X}{\|x_i - \bar x\|_X} 
 \|R(x_i) - R(\bar x)\|_X
 +
\frac{1+\gamma}{1-\gamma} \rho_i \|R(x_i)\|_X
  \\
 &=
   o(1)
 \|R(x_i)\|_X\qquad \forall i \in \N_0, 
 \end{aligned}
 \end{equation*}  
 where the Landau notation $o(1)$ refers to the limit $i \to \infty$. 
 This proves that $\| R(x_i)\|_X$ 
 converges
 $q$-superlinearly to zero. To
 obtain that the convergence $\|x_i - \bar x\|_X \to 0$
 is $q$-superlinear too, it suffices to note that 
 the same arguments as in \eqref{eq:randomeq_2626} 
 and the estimates \eqref{eq:R_estimate_1}
 and \eqref{eq:R_estimate_2} imply
 \[
    \|x_{i+1} - \bar x\|_X
    \leq
    \frac{1}{1-\gamma}
\| R(x_{i+1})\|_X
\leq
    o(1) \| R(x_i)\|_X
    =
    o(1) \| R(x_i) - R(\bar x)\|_X
    \leq
    o(1) \|x_i - \bar x\|_X
 \]
 for  $i \to \infty$. 
 This establishes \ref{th:convergence:ii} and completes the proof
 of \cref{th:convergence}. 
\end{proof}
\end{appendices}

\phantomsection
\bibliography{QVISSN-bibliography}
\end{document}